\documentclass[11 pt, twoside, a4paper]{amsart}
\usepackage[headsepline,plainheadsepline]{scrlayer-scrpage}
\usepackage{amsfonts, amsmath, amssymb, amsthm, amsopn, latexsym, graphicx, mathrsfs,  verbatim, enumerate, url, stmaryrd, enumitem, tikz} 
\usepackage{tikz-cd}
\usetikzlibrary{fit, shapes}

\usepackage{hyperref}
\usepackage[all]{hypcap} 
\usepackage{pst-all} 
\usepackage{array}
\usepackage{Style} 
\usepackage{subcaption}
\usepackage{lastpage}
\usepackage{rotating}
\usepackage{mathtools}

\usepackage{lineno} 
\usepackage{etoolbox}
\patchcmd{\subsection}{-.5em}{.5em}{}{}
\patchcmd{\subsubsection}{-.5em}{.5em}{}{}
\usepackage[a4paper, left=2cm,right=2cm,top=2cm,bottom=2cm]{geometry}

\usepackage{index}			
\usepackage{pstricks}			
\usepackage[all,knot,arc]{xy}		

\usepackage{color}

\definecolor{DarkGreen}{RGB}{0, 164, 0}

\DeclareMathAlphabet{\mathbbm}{U}{bbm}{m}{n}

\graphicspath{{}}

\newcommand{\executeiffilenewer}[3]{%
	\ifnum\pdfs
	
	trcmp{\pdffilemoddate{#1}}%
	{\pdffilemoddate{#2}}>0%
	{\immediate\write18{#3}}\fi%
}

\hypersetup{
colorlinks=true,
linkcolor=blue,
citecolor=red
}

\newtheorem{introthm}{Theorem}

\newtheorem{thm}{Theorem}[section]
\newtheorem{lem}[thm]{Lemma}
\newtheorem{cor}[thm]{Corollary}
\newtheorem{prop}[thm]{Proposition}

\theoremstyle{definition}
\newtheorem{rmk}[thm]{Remark}
\newtheorem{defi}[thm]{Definition}

\theoremstyle{definition}

\newcommand{\thmstep}[1]
{
\begin{tikzpicture}[baseline=(char.base)]
\node(char)[draw,fill=white,
shape=rounded rectangle,
minimum width=1.3cm]
{#1};
\end{tikzpicture}
}

\newcommand{\thmsteprect}[1]
{
\begin{tikzpicture}[baseline=(char.base)]
\node(char)[draw,fill=white,
shape=rectangle,
minimum width=1.3cm]
{#1};
\end{tikzpicture}
}

\newcommand{\itemchdir}[1]{\item[\thmstep{$#1$}]}
\newcommand{\itemcase}[1]{\item[\thmsteprect{#1}]}

\newcommand{\overtype}[2]{\underset{}{\overset{#1}{#2}}}

\newcommand{\mytitle}{Birational properties of tangent to the identity germs without non-degenerate singular directions}
\newcommand{\myshorttitle}{Birational properties of germs without non-degenerate singular directions}

\title{\mytitle}

\author{Samuele Mongodi}
\address{Dipartimento di Matematica, Politecnico di Milano, Via Bonardi 9, I-20133, Milano, Italy}
\email{\href{mailto:samuele.mongodi@gmail.com}{samuele.mongodi@gmail.com}}

\author{Matteo Ruggiero}
\address{Université de Paris, Sorbonne Université, CNRS, Institut de Mathématiques de Jussieu-Paris Rive Gauche, F-75013 Paris, France
}
\email{\href{mailto:matteo.ruggiero@imj-prg.fr}{matteo.ruggiero@imj-prg.fr}}

\date{\today}

\newcommand{\refprop}[1]{\hyperref[prop:#1]{Proposition~\ref*{prop:#1}}}
\newcommand{\refthm}[1]{\hyperref[thm:#1]{Theorem~\ref*{thm:#1}}}
\newcommand{\refcor}[1]{\hyperref[cor:#1]{Corollary~\ref*{cor:#1}}}
\newcommand{\refdef}[1]{\hyperref[def:#1]{Definition~\ref*{def:#1}}}
\newcommand{\refrmk}[1]{\hyperref[rmk:#1]{Remark~\ref*{rmk:#1}}}
\newcommand{\reflem}[1]{\hyperref[lem:#1]{Lemma~\ref*{lem:#1}}}
\newcommand{\refsec}[1]{\hyperref[sec:#1]{Section~\ref*{sec:#1}}}
\newcommand{\refssec}[1]{\hyperref[ssec:#1]{Section~\ref*{ssec:#1}}}
\newcommand{\reffig}[1]{\hyperref[fig:#1]{Figure~\ref*{fig:#1}}}
\newcommand{\refex}[1]{\hyperref[ex:#1]{Example~\ref*{ex:#1}}}
\newcommand{\refeqn}[1]{\eqref{eqn:#1}}
\newcommand{\vx}{x}

\newcommand{\refsssec}[1]{\hyperref[sssec:#1]{Subsection~\ref*{sssec:#1}}}

\def\tid{tangent to the identity}
\def\degspike{degenerate spike}
\def\Degspike{Degenerate spike}
\def\spincorner{spinning corner}
\def\Spincorner{Spinning corner}
\def\Halfcorner{Half corner}
\def\halfcorner{half corner}
\def\hcsimple{simple}
\def\hcnonsimple{non-simple}

\def\iforma{$N_1$-form}
\def\iiforma{$N_2$-form}
\def\iiiforma{$N_3$-form}

\def\aforma{{$H_1$-form}}

\newcommand{\ZeroC}{{\text{R}_0}} 
\newcommand{\DS}{{\text{R}_1}} 
\newcommand{\SpinC}{{\text{R}_2}}
\newcommand{\HalfC}{\text{R}_3}

\newcommand{\IIfor}{\text{N}_2}
\newcommand{\IIIfor}{\text{N}_3}
\newcommand{\OneC}{\text{H}_1}

\newcommand{\locint}[3]{\langle {#1}, {#2} \rangle_{#3} }

\renewcommand{\vect}[1]{\underset{}{\overrightarrow{#1}}}

\newcommand{\fps}[1]{{\left\llbracket {#1} \right\rrbracket}}

\clearscrheadfoot
\lohead{\normalfont  \myshorttitle (DRAFT) \hfill \pagemark}

\lehead{\normalfont\pagemark\hfill S. Mongodi, M. Ruggiero}
\begin{document}

\maketitle
\thispagestyle{empty}

\begin{abstract}\noindent 
We provide a family of isolated {\tid} germs $f:(\nC^3,0) \to (\nC^3,0)$ which possess only degenerate characteristic directions, and for which the lift of $f$ to any modification (with suitable properties) has only degenerate characteristic directions.
This is in sharp contrast with the situation in dimension $2$, where any isolated {\tid} germ $f$ admits a modification where the lift of $f$ has a non-degenerate characteristic direction.
We compare this situation with the resolution of singularities of the infinitesimal generator of $f$, showing that this phenomenon is not related to the non-existence of complex separatrices for vector fields of Gomez-Mont and Luengo.
Finally, we describe the set of formal $f$-invariant curves, and the associated parabolic manifolds, using the techniques recently developed by L\'opez-Hernanz, Raissy, Rib\'on, Sanz S\'anchez, Vivas.
\end{abstract}


\section*{Introduction}

In this paper, we investigate birational properties of {\tid} germs in $\nC^3$, in relation with the construction of (strong) separatrices and parabolic manifolds.

A holomorphic germ $f:(\nC^d,0)\to (\nC^d,0)$ is {\tid} when its differential at $0$ is the identity.
In the one-dimensional case, Leau-Fatou's flower theorem \cite{leau:etudeequationsfonctionnelles1,leau:etudeequationsfonctionnelles2,fatou:equationsfonctionnelles1} ensures the existence of simply connected invariant domains (petals) containing the origin at their boundary, where $f$ is conjugated to a translation. Petals for $f$ and $f^{-1}$ cover a pointed neighborhood of the origin, and allow a precise description of the local dynamics of these germs.
The properties of {\tid} germs and their petals are fundamental both in the global (see e.g. the monography \cite{milnor:dyn1cplxvar}) and in the local (see e.g. the topological classification of {\tid} germs \cite{camacho:tidgermstopoclass}) aspects of the theory of holomorphic dynamical systems in dimension $1$, as well as for understanding bifurcations via parabolic implosion (see e.g. the survey \cite{shishikura:bifurcationparabfixpoint}).

In higher dimensions, it is not possible to give, in general, such a precise description of the dynamics near the origin, but one can still aim at describing higher dimensional analogues of the petals, called \emph{parabolic manifolds}.

They are attached to complex tangent directions at $0$, called \emph{characteristic directions}.
Characteristic directions can be described as fixed points for the action induced by $f-\id$ on the exceptional divisor of the blow-up of the origin (or equivalently, of the action of the homogeneous part $H$ of smallest degree of $f-\id$ on $\nP^{d-1}$).
These fixed points can be either holomorphic (\emph{non-degenerate} case) or meromorphic (\emph{degenerate} case).
A fundamental result by Hakim \cite{hakim:analytictransformations} shows the existence of parabolic curves tangent to non-degenerate characteristic directions (in any dimension).

Later, Abate \cite{abate:residualindexdynholmaps} 
shows the existence of parabolic curves for isolated {\tid} germs in dimension $d=2$.
In analogy with Camacho-Sad's construction of complex separatrices for holomorphic foliations in dimension $2$ \cite{camacho-sad:invariantvarieties}, the proof consists in showing that, after a finite number of blow-ups along characteristic (and in fact singular, see \refdef{directions}) directions, one can always find a regular modification (i.e., a composition of blow-ups) where the lift of $f$ has at least one non-degenerate characteristic direction. This allows to apply Hakim's result to get a parabolic curve for the lifted germs, transversal to the exceptional divisor of the modification, so that it descends to a parabolic curve for $f$.
Several authors addressed the problem of finding stable manifolds for (possibly non-isolated) $2$-dimensional {\tid} germs, and the picture is quite complete now, see e.g. \cite{ecalle:fonctionsresurgentes3, weickert:attractingbasinsdim2, hakim:analytictransformations,abate:residualindexdynholmaps, abate-bracci-tovena:indextheorems,brochero-cano-lopezhernanz:parabolicrcurvesdiffeodim2, molino:dynamicstidnonvanishingindex, vivas:degchardirtid, rong:newinvattrdomainstiddim2, lopezhernanz-sanzsanchez:paraboliccurvesdiffeoasymptotic, lopezhernanz-raissy-ribon-sanzsanchez:stablemfld2dim, lopezhernanz-rosas:chardirdim2}.
The description of parabolic manifolds has been recently instrumental for the construction of examples of wandering domains, see \cite{astorg-buff-dujardin-peters-raissy:twodimpolywandering,astorg-bocthaler-peters:wanderingdomainslavaursmapssiegeldisks}.
(Semi-)parabolic implosion in dimension $2$ (or higher) and applications to bifurcation theory can also be found in the literature (see e.g. \cite{bedford-smillie-ueda:semiparabifurcdim2, dujardin-lyubich:stabbifurcdissipativepolyautoC2, bianchi:parabimplosionendoC2}), and mainly rely on a careful study of the dynamics on parabolic curves.

We briefly expose here some of the reasons why the study of {\tid} germs and their parabolic manifolds is much harder in higher dimensions.
Firstly, the homogenous part $H$ introduced above acts on $\nP^{d-1}$: for $d=2$ all indeterminacy points can be avoided by saturation, while they persist when $d \geq 3$.
Since $2$-dimensional modifications are composition of point blow-ups, most of the phenomenon are combinatorial. In higher dimensions, we can blow-up higher dimensional centers, and their geometry needs to be taken into account. Moreover, we only have a weak factorization theorem (see \cite{abramovich-karu-matsuki-wlodarczyk:torificationfactorizationbirationalmaps,bonavero:factorisationfaible}).
Resolution theorems for vector fields are available in dimension $2$ (see \cite{seidenberg:redsing}), and recently dimension $3$ (see \cite{panazzolo:ressingrealanalvfdim3, mcquillan-panazzolo:almostetaleresfol}): here we need in general to introduce singularities on the ambient space, by considering weighted blow-ups and orbifolds.
Finally, the \emph{infinitesimal generator} of a {\tid} germ may not admit complex separatrices when $d \geq 3$, as showed by Gomez-Mont and Luengo \cite{gomezmont-luengo:vectorfieldsnoseparatrix}.
Adapting their construction to {\tid} germs, Abate and Tovena \cite{abate-tovena:paraboliccurvesC3} give examples of {\tid} germs in dimension $3$ that do not admit \emph{robust} parabolic curves, i.e., parabolic curves attached to invariant formal curves, the analogue of (formal) complex separatrices in this setting.
In their examples, all characteristic directions are non-degenerate, and (non-robust) parabolic curves exist thanks to Hakim's theorem.

\smallskip

In this paper, we investigate the existence of parabolic manifolds attached to degenerate characteristic directions in dimension $3$, by studying the following family of {\tid} germs:
\begin{equation}\label{eqn:example}
f(x,y,z)=(x+yz(y-z)+P, y+x(x^2-z^2)+Q, z+xz(y-z)+R)	\text.
\end{equation}
Here $P,Q,R$ are holomorphic germs with order at least $4$ at the origin.
The coefficients of the formal power series expansion of $P,Q,R$ are considered as parameters of the family. We say that a certain property holds for a generic (resp., $\nR$-generic) element of the family if it holds for an open dense subset of the parameters with respect to the Zariski topology over $\nC$ (resp., over $\nR$).

Since characteristic directions are determined only by the homogeneous part of smallest degree of $f-\id$, all these maps share the same characteristic directions: there are five of them, which we label $v_1, v_2, v_3, v_4, v_5$, all of them degenerate. We denote by $p_1, p_2, p_3, p_4, p_5$ the corresponding points on the exceptional divisor of the blow-up of the origin.
Other examples are easy to construct, building on the examples of rational maps in $\nP^2$ with no holomorphic fixed points given by \cite{ivashkovich:fixedpointsratmaps}.

\medskip

For the maps described by \refeqn{example}, we investigate two possible strategies to find parabolic manifolds.
The first strategy, following \cite{abate:residualindexdynholmaps}, consists in looking for a suitable birational model, where we can find non-degenerate characteristic directions (that are non-exceptional, i.e., transverse to the exceptional divisor).
Since non-degenerate characteristic directions correspond to eigenvectors of the linear part of the saturated infintesimal generator $\hat{\chi}$ of $f$, it is natural to start our study from a birational model $\pi_0:X_{\pi_0} \to (\nC^3,0)$, which provides a resolution of the singularities of the infinitesimal generator.

Our example shows that, unlike dimension $2$, this first strategy may fail in higher dimensions.

\begin{introthm}\label{thm:mainnonondeg}
A generic element $f:(\nC^3,0)\to(\nC^3,0)$ of the family \refeqn{example} satisfies the following property:
\begin{quote}
For any regular modification $\pi:X \to (\nC^3,0)$ strongly adapted to $f$ and dominating $\pi_0$, and for any point $p \in \pi^{-1}(0)$ in the exceptional divisor, the lift $\wt{f}:X \to X$ of $f$ at $p$ has only degenerate characteristic directions.
\end{quote} 
\end{introthm}

Here ``regular'' means that we only allow sequences of blow-ups of smooth centers, while ``adapted to $f$'' means that we only allow blow-up of centers that are invariants by the saturated infinitesimal generator $\hat{\chi}$, and with ``strongly adapted'' we only allow to blow-up points or curves belonging to the singular locus of $\hat{\chi}$.

We can actually say a little more about this family: one cannot find any non-degenerate characteristic direction also for any regular modification (not necessarily strongly) adapted to $f$ above the points $p_1$ and $p_2$ (see \refprop{degspikecurve}), nor for any point modification (see \refsssec{pointmodif}).

\smallskip

The second strategy is in line with the recent works \cite{lopezhernanz-raissy-ribon-sanzsanchez:stablemfld2dim,lopezhernanz-rosas:chardirdim2,lopezhernanz-ribon-sanzsanchez-vivas:stablemanifoldsbiholoCn}.
It consists in looking for complex separatrices for the dynamics, and study parabolic manifolds attached to them.
While we know that this second strategy may fail in general by \cite{gomezmont-luengo:vectorfieldsnoseparatrix, abate-tovena:paraboliccurvesC3}, it proves quite fruitful in this case.
We are able to find formal invariant curves tangent to the directions $v_1, \ldots, v_4$, and deduce the existence of parabolic manifolds by \cite{lopezhernanz-ribon-sanzsanchez-vivas:stablemanifoldsbiholoCn}.

\begin{introthm}\label{thm:mainparmfld}
For generic elements $f:(\nC^3,0)\to(\nC^3,0)$ of the family \refeqn{example}, there exists formal invariant curves $C_1, \ldots, C_4$ tangent to $v_1, \ldots, v_4$ respectively.
These curves are smooth, and they are the only formal invariant curves tangent to any direction (but possibly $v_5$).

Finally, there are $3$ (resp., $5$) parabolic manifolds asymptotic to $C_1$ and $C_2$ (resp., $C_3$ and $C_4$), of dimension either $1$ or $2$ ($\nR$-generically they are all of dimension $2$).
\end{introthm}

The dynamics above $p_5$ remains more complicated to describe.
We are able to exclude the existence of formal invariant curves that are transverse to the exceptional divisor of the model $X_{\pi_0}$, while we are able to find a formal invariant surface $S$ tangent to $v_5$.

Besides being only formal, the surface $S$ is also singular, and \cite{lopezhernanz-rosas:chardirdim2} cannot be applied to $f|_S$ even if $S$ were convergent.
When working on the model $X_{\pi_0}$, the strict transform $\wt{S}$ of $S$ is smooth and invariant by the lift $\wt{f}$ of $f$. A direct computation shows that $\wt{f}|_{\wt{S}}$ has only two characteristic directions, corresponding to the tangent space of the exceptional divisor $\pi_0^{-1}(0)$. In general $\pi_0^{-1}(0)$ could provide the only separatrices of $\wt{f}$, and constructing parabolic manifolds would require other techniques (similar to \cite{lopezhernanz-rosas:chardirdim2}).

\medskip

The techniques used to prove \refthm{mainnonondeg} are mainly combinatorial.
In particular, we identify three new classes of {\tid} germs, namely \emph{{\degspike}s}, \emph{{\spincorner}s} and \emph{{\halfcorner}s}, and show that all singularities in a suitable model dominating $X_{\pi_0}$ belong to one of these classes (or \emph{simple corners} introduced in \cite{abate-tovena:paraboliccurvesC3}).
Then we show that these classes are invariant by (strongly) adapted regular modifications, and they do not admit non-degenerate non-exceptional characteristic directions.

To prove \refthm{mainparmfld}, we use the combinatorial knowledge achieved in the previous step, and some computations using normal forms, to describe the set of formal invariant curves attached to the classes introduced above.
Moreover, we compute the reduction to Ramis-Sibuya normal form, and apply the results in \cite{lopezhernanz-ribon-sanzsanchez-vivas:stablemanifoldsbiholoCn} to deduce the existence of parabolic manifolds attached to these formal invariant curves.

In both results, the genericity conditions are explicit and easy to check. They are not essential to the results: they are taken to simplify the birational study and the exposition of the dynamical properties of germs of the form \refeqn{example}.

\smallskip

Besides giving an explicit way to find formal invariant curves and parabolic manifolds in a non-trivial example, the identification of classes invariant by (adapted) modifications provide ideal candidates to replace the final reduced forms $\star 1$ and $\star 2$ of \cite{abate:residualindexdynholmaps}.
The reduction to these classes would be a fundamental step towards proving in general the existence of parabolic manifolds in higher dimensions.

\medskip

The paper is organized as follows.
In \refsec{background}, we recall some basics about {\tid} germs, vector fields, birational geometry and construction of formal curves, as well as the theory of Ramis-Sibuya normal forms and the construction of parabolic manifolds in the case of {\tid} germs. 

In \refsec{example} we introduce the family of maps \refeqn{example}, study characteristic directions, and exhibit the resolution $\pi_0:X_{\pi_0} \to (\nC^3,0)$ of the infinitesimal generator.

In \refsec{biratstudyaboveresol} we recall the definition of simple corners, and introduce the three new classes.
We then study their combinatorics in terms of point blow-ups.

In \refsec{patterns} we study the behavior of these classes under regular modifications strongly adapted to the dynamics, and conclude the proof of \refthm{mainnonondeg}.

Finally, in \refsec{invariantcurves} we use the combinatorial picture portrayed in the previous section to construct formal invariant curves, compute Ramis-Sibuya normal forms, and conclude the proof of \refthm{mainparmfld}.
We end this section by some remarks on not strongly adapted modifications, point modifications, and on the dynamical picture above $p_5$.

\bigskip

\noindent\textbf{Acknowledgements.}
This project has been developed during a collaboration with Leandro Arosio. The authors thank Leandro for the insightful discussions. 
The interest for the phenomena described in this paper arose after a fruitful discussion with Filippo Bracci, whom we thank as well.

Part of this work has been developed while the first named author was visiting the University of Paris Diderot, supported by the LIA LYSM AMU CNRS ECM INdAM funding. The second named author has been partially supported by the ANR grant Fatou ANR-17-CE40-0002-01.


\section{Background}\label{sec:background}

\subsection{Modifications}

We start by some terminology about sequences of blow-ups.

\begin{defi}
A modification of $(\nC^d,0)$ is a proper bimeromorphic map $\pi: X_\pi \to (\nC^d,0)$ which is a biholomorphism outside the exceptional divisor $\pi^{-1}(0)$.
A modification is called \emph{smooth} if $X_\pi$ is smooth, \emph{regular} if $X_\pi$ is obtained as a composition of blow-ups of smooth centers.
\end{defi}

If $d=2$, any smooth modification is obtained as a finite composition of point blow-ups.
In general, building blocks of modifications are still given by blow-ups, whose centers have codimension at least $2$.
In particular for $d=3$, we can blow-up both points and curves. The study of the birational geometry of {\tid} germs needs to take into account the geometry of such curves, and not only the combinatorial data of blow-ups.

Moreover, it is not anymore true that smooth modifications are given by composition of blow-ups (see \cite{abramovich-karu-matsuki-wlodarczyk:torificationfactorizationbirationalmaps,bonavero:factorisationfaible}), which gives a further technical difficulty to deal with generic modifications.

Most of the modifications we will consider will be \emph{point modifications}, i.e., composition of point blow-ups, since they are more directly related to characteristic directions.

When doing so, we will perform local computations on suitable charts.

Let $\pi:X_\pi \to (\nC^d,0)$ be the blow-up of the origin, and fix local coordinates $(x_1, \ldots, x_d)$ at $0 \in \nC^d$.
The total space $X_\pi$ of the blow-up is covered by $d$ charts $U_j$, with $j=1, \ldots, d$, corresponding to the complementary in $X_\pi$ of the strict transform of the hyperplane $\{x_j = 0\}$.
With abuse of notation, we denote by $(x_1, \ldots, x_d)$ also the coordinates in $U_j$, for which the map $\pi$ takes the form $$\pi(x_1, \ldots, x_d)=(x_1 x_j, \ldots, x_{j-1}x_j, x_j, x_j x_{j+1}, \ldots, x_j x_d)\text.$$
In this case, we will say that we work \emph{in the $x_j$-chart}.
A point $p$ corresponding to a direction $v=[a_1: \ldots : a_d]$ belongs to $U_j$ if and only if $a_j \neq 0$. If this is the case, $p$ has coordinates $\left(\frac{a_1}{a_j}, \ldots, \frac{a_{j-1}}{a_j}, 0, \frac{a_{j+1}}{a_j}, \ldots, \frac{a_{d}}{a_j}\right)$ in the $x_j$-chart.

\subsection{Characteristic directions}
We introduce here some terminology about characteristic directions for {\tid}  germs in $(\nC^d,0)$.

\begin{defi}\label{def:directions}
Let $f:(\nC^d,0) \to (\nC^d,0)$ be a {\tid} germ.
Denote by $H$ the homogeneous part of smallest degree of $f-\id$, by $\ell$ the greatest common divisor of the $d$ coordinates of $f-\id$ (defined up to units), and let $H_\ell$ be the homogeneous part of smallest degree of $\ell^{-1}(f-\id)$.
A tangent direction $v \in \nP_\nC^{d-1}$ is called
\begin{itemize}
\item \emph{characteristic} if there exists $\lambda \in \nC$ so that $H(v)=\lambda v$;
\item \emph{singular} if there exists $\lambda \in \nC$ so that $H_\ell(v)=\lambda v$.
\end{itemize}
In both cases, $v$ is called \emph{non-degenerate} if $\lambda \neq 0$, and \emph{degenerate} if $\lambda = 0$.

The degree of $H$ is called the \emph{order} of $f$, while the degree of $H_\ell$ is called the \emph{pure order} of $f$.
\end{defi}

\begin{rmk}
The value $\lambda$ in the previous definition is sometimes called \emph{multiplier} of the characteristic direction. Notice that such value is not well defined up to change of coordinates, but its vanishing is.

Notice also that if $v$ is a singular direction, then it is a characteristic direction. In fact, if $L$ is the homogeneous part of $\ell$ of smallest degree, then $H=LH_\ell$ and $H(v)=L(v)H_\ell(v)=\lambda L(v)v$.
We also infer that any characteristic direction that is tangent to $\{L=0\}$ (or equivalently to $\{\ell=0\}$) is automatically degenerate (as a characteristic direction).
\end{rmk}

\begin{rmk}
Borrowing some terminology from algebraic geometry, one can see characteristic and singular directions as the same object.

Let $X$ be a smooth manifold, $p \in X$, $f:(X,p) \to (X,p)$ be a {\tid} germ, and $D=\{\psi=0\}$ be an effective divisor.
Assume that its support is contained in $\on{Fix}(f)$.
Then locally at $p$ we can write
$$
f(\vx) = \vx + \psi(\vx) \cdot (H_\psi(\vx) + \textrm{h.o.t.})\text,
$$
where $\textrm{h.o.t.}$ stands for ``higher other terms''.
In this situation, a \emph{$D$-characteristic direction} (or a \emph{singular direction with respect to $D$}) is an element $v \in \nP(T_p X)$ such that $H_\psi(v)=\lambda v$ for some $\lambda \in \nC$.
Characteristic directions are obtained when $D=0$ (or equivalently $\psi=1$), while singular directions are obtained when $\psi=\ell$ as above (in this case the support of $D=\on{div}(\ell)$ is the pure $(d-1)$-dimensional part of $\on{Fix}(f)$).
\end{rmk}

We need some terminology to describe the interaction between the exceptional divisor of a given modification and characteristic and singular directions.

\begin{defi}
Let $f:(\nC^d,0) \to (\nC^d,0)$ be a {\tid} germ, and $\pi:X_\pi \to (\nC^d,0)$ be a smooth modification. 
Denote by $E$ the exceptional divisor of $\pi$, and let $f_\pi:X_\pi \dashrightarrow X_\pi$ be the lift of $f$ in $X_\pi$. Let $p \in E$ be a point in the exceptional divisor so that the germ of $f_\pi$ at $p$ defines a {\tid} germ.
We say that a characteristic direction of $f_\pi$ is \emph{exceptional} if it belongs to the projectivization of the tangent space of $E$ at $p$.
\end{defi}
In other terms, we can consider the blow-up of $p$, getting another modification $\pi':X_{\pi'} \to (\nC^d,0)$ dominating $\pi$: $\pi'=\pi \circ \eta$ with $\eta$ the blow-up at $p$. Then a characteristic direction $v$ of $f_\pi$ is exceptional if the corresponding point $p_v$ in $\eta^{-1}(0)$ belongs to the strict transform of the exceptional divisor $E$.
If $v$ is a characteristic (resp., singular) direction for $f_\pi$, we will call the corresponding point $p_v$ a \emph{charcteristic} (resp., \emph{singular}) \emph{point}.

Clearly, the set of singular points describe an algebraic subvariety of $\nC\nP^{d-1}$.
If the maximal dimension of the irreducible components of this subvariety is $k$, we say that the germ $f$ is \emph{$k$-dicritical}.
Notice that $0$-dicritical germs have only finitely many singular directions.
When $f$ is $(d-1)$-dicritical, the set of singular points coincide with $\nC\nP^{d-1}$, and we simply say that $f$ is \emph{dicritical}.

\begin{rmk}
Assume that $f:(\nC^d,0)\to (\nC^d,0)$ is a {\tid} germ, with an isolated fixed point.
Let $\pi:X_\pi \to (\nC^d,0)$ be any modification.
Since $\pi$ defines a local isomorphism outside the exceptional divisor, the lift $f_\pi$ of $f$ at $X_\pi$ satisfies $\on{Fix}(f_\pi)=\pi^{-1}(0)$. More generally if $\on{Fix}(f)$ has no divisorial components, then $\on{Fix}(f_\pi)$ has no divisorial components outside the exceptional divisor $E=\pi^{-1}(0)$.

In the families we will study in the next chapters, we will often consider exceptional the directions tangent to the divisor $F$ of fixed points of $f$, since these families arise when studying the lift of the maps of \refeqn{example} with respect to some modification.
\end{rmk}

\subsection{Infinitesimal generators}\label{ssec:infgen}

To any {\tid} germ $f \colon (\nC^d,0) \to (\nC^d,0)$ is associated a unique (formal, possibly non-convergent) vector field $\chi$, that has multiplicity at $0$ at least $2$, and satisfying $f= \exp \chi$ (see e.g. \cite{brochero-cano-lopezhernanz:parabolicrcurvesdiffeodim2} for the construction in dimension $2$). 
We recall that if $\phi$ is a local coordinate (hence defining a germ of smooth hypersurface $\{\phi=0\}$) at $0$, we have
$$
\phi \circ \exp \chi = \sum_{n=0}^\infty \frac{\chi^n(\phi)}{n!},
$$
where $\chi^n$ denotes the derivation $\chi$ applied $n$ times.
The vector field $\chi$ is called the \emph{infinitesimal generator} of $f$, and denoted by $\chi= \log f$.

\begin{rmk}\label{rmk:homocontribution}
Notice that the homogeneous part of degree $k$ of $\chi$ contributes to the factor $\chi^n(\phi)$ only starting from order $n(k-1)+1$.
	
In particular, if $f$ is given by the example \eqref{eqn:example}, then $\phi \circ (f-\id)$ and $\chi(\phi)$ coincide up to order $4$. 
\end{rmk}

Consider now a smooth manifold $X$, a compact (smooth) submanifold $Z \subset X$ of codimension at least $2$, and $\pi:X_\pi \to X$ the blow-up of $X$ along $Z$.
If $\chi$ is a vector field on $X$, then $\chi$ lifts to a vector field $\chi_\pi$ on $X_\pi$, satisfying $\chi_q =(d\pi)_p (\chi_\pi)_p$ for any $q=f(p)$, $p \in X_\pi$, as far as $\chi$ is tangent to $Z$.
When $Z$ is reduced to a point $\{p\}$ this happens exactly when $p$ is a singular point of $\chi$.

Applying this situation to the infinitesimal generator $\chi$ of a {\tid} germ $f$, we get:
\begin{prop}\label{prop:infgenblowup}
Let $X$ be a smooth manifold, $Z \subset X$ be a compact smooth submanifold of codimension at least $2$, and $\pi:X_\pi \to X$ the blow-up of $X$ along $Z$.
Let $f:X \to X$ be a holomorphic map fixing $Z$ pointwise, and such that the germ of $f$ at any point of $Z$ is {\tid}; denote by $\chi$ the infinitesimal generator of $f$.
Finally, denote by $f_\pi:X_\pi \to X_\pi$ the lift of $f$, and by $\chi_\pi$ the lift of $\chi$.

Then $\chi_\pi$ is the infinitesimal generator of $f_\pi$, i.e., $f_\pi = \exp \chi_\pi$.
\end{prop}
\begin{proof}
Assume for the moment that $\chi$ is analytic.
Let $\Theta$ be the flow of $\chi$, so that $f(z)=\Theta(z,1)$.
Set $z=\pi(x)$ and $\Omega(x,t) = \pi^{-1} \circ \Theta(z,t)$ for any $x \not \in E$. As $\Omega$ is analytic and bounded in a neighborhood of $E$, it extends holomorphically to $E$.
Now, let us consider $x\in X_\pi\setminus E$. On the one hand, we get
$$
\pi \circ f_\pi(x)=f(z)=\Theta(z,1)= \pi \circ \Omega(x,1)\;.
$$
On the other hand, we get
$$
\chi_{\Theta(z,t)} = \Theta'(z,t) = d\pi_{\pi^{-1}(\Theta(z,t))} \Omega'(x,t) = d\pi_{\Omega(x,t)} (\chi_\pi)_{\Omega(x,t)},
$$
and $\Omega$ is the flow of $\chi_\pi$. As this holds outside $E$ and all the maps involved extend holomorphically to $E$, we obtain the desired result for $\chi$ analytic.
	
The result for $\chi$ formal follows, by applying the previous calculation to truncations, and by \refrmk{homocontribution}.
\end{proof}

\subsection{Vector fields and characteristic directions}

In more abstract terms, \refprop{infgenblowup} says that the operator associating to a {\tid} germ its infinitesimal generator is functiorial (with respect to regular modifications adapted to $f$).

The next proposition explicits the link between characteristic/singular directions, and singularities of the infinitesimal generator.

\begin{prop}\label{prop:singularitiestidinfgen}
Let $f:(\nC^d,0) \to (\nC^d,0)$ be a {\tid} germ, and $v \in \nP_\nC^{d-1}$ be a tangent direction at $0$.
Denote by $\pi:X_\pi \to (\nC^d,0)$ the blow-up of the origin, by $f_\pi$ the lift of $f$ at $X_\pi$, and by $\chi_\pi$ the infinitesimal generator of $f_\pi$.
Let $D$ be an effective divisor whose support is contained in $\on{Fix}(f)$.
Then $v$ is a $D$-characteristic direction for $f$ if and only if the saturation of $\chi_\pi$ with respect to $\pi^{-1}D$ is singular at the corresponding point $p_v \in \pi^{-1}(0)$.
\end{prop}
\begin{proof}
Fix coordinates $\vx=(x_1,\ldots, x_d)$ on $(\nC^d,0)$ so that $v=[1:0:\cdots:0]$.
Write $D = \{\ell=0\}$. Then we can write $f$ as:
$$
f(\vx)=\vx + \ell(\vx) \big(H(\vx) + \mf{m}^{h+1}\big),
$$
where $\mf{m}$ is the maximal ideal at $0$, and $H:=H_\ell$ is a non-vanishing homogeneous polynomial of degree $h \geq 0$.
Then $v$ is $D$-characteristic if and only if $H_\ell(v)=\lambda v$ for some $\lambda \in \nC$.
We work in the $x_1$-chart. The lift $f_\pi$ of $f$ satisfies
\begin{align*}
x_1 \circ f_\pi &= x_1 + \ell \circ \pi(\vx) \Big(x_1^h H_1(1, x_2, \ldots, x_d) + \langle x_1^{h+1}\rangle\Big)\text,\\
x_j \circ f_\pi &= \frac{x_j + \ell \circ \pi(\vx) \Big(x_1^{h-1} H_j(1, x_2, \ldots, x_d) + \langle x_1^{h}\rangle\Big)}{1 + \ell \circ \pi(\vx) \Big(x_1^{h-1} H_1(1, x_2, \ldots, x_d) + \langle x_1^{h}\rangle\Big)}\text,
\end{align*}
where $H=(H_1, \ldots, H_d)$ and $j=2, \ldots, d$.
Notice that $\ord_0(\ell)+h \geq 2$, hence $f_\pi$ leaves $\pi^{-1}(0)$ fixed.
If $L$ denotes the homogeneous part of smallest degree of $\ell$, by \refprop{infgenblowup} the infinitesimal generator $\chi_\pi$ has the following form when developed near $p_v$ (corresponding to the origin in the coordinates $(x_1,\ldots, x_d)$):
\begin{equation*}
\chi_\pi=\Big( x_1^{h-1} \ell \circ \pi(\vx)\Big) \Bigg(x_1 H_1(1, x_2, \ldots, x_d) \partial_1 + \sum_{j=2}^d (H_j - x_j H_1)(1,x_2, \ldots, x_d) \partial_j + x_1\xi\Bigg).
\end{equation*}
where $\xi$ is a suitable vector field.
Notice also that the saturation of $\chi_\pi$ with respect to $\pi^{-1}(D)$ is exactly given by $\hat{\chi}_\pi = \Big( x_1^{h-1} \ell \circ \pi(\vx)\Big)^{-1} \hspace{-3mm}\chi_\pi$. 

Then, $v=[1:0:\cdots:0]$ is characteristic if and only if $H_j(1,0, \ldots, 0)=0$ for all $j=2, \ldots, d$.
But this happens if and only if $\hat{\chi}_\pi$ has a singularity at the origin.
\end{proof}

\begin{rmk}
Denote by $F_\pi$ the divisor of fixed points of $f_\pi$. 
On the one hand, when $f$ is not dicritical, then $F_\pi = \pi^{-1}F$, and the saturated infintesimal generator $\hat{\chi}_\pi$ is tangent to the exceptional divisor.
On the other hand, when $f$ is dicritical, then $F_\pi > \pi^{-1}F$ (as divisors). In this case all points of the exceptional divisor are singularities for $\hat{\chi}_\pi$, which is not saturated.

Notice that by \refrmk{homocontribution}, the set of singular points of a {\tid} germ $f$ coincides with the set of singular points of the saturated infinitesimal generator $\hat{\chi}_\pi$.
\end{rmk}

We extend the notion of singular points, using the interpretation in terms of saturated infinitesimal generator, for models not obtained as point modifications.

\begin{defi}
Let $X$ be a complex manifold, $Z \subset X$ a compact submanifold of $X$, and $f:(X,Z) \to (X,Z)$ a germ of holomorphic map fixing $Z$ pointwise, and for which $f$ is a {\tid} germ at any $p \in Z$.
Let $\pi:X_\pi \to (X,Z)$ be a modification over $Z$, adapted to $f$.
Denote by $f_\pi$ the lift of $f$ at $X_\pi$, and by $\hat{\chi}_\pi$ its saturated infinitesimal generator.
Then we say that $f_\pi$ is \emph{singular} at $p \in \pi^{-1}(Z)$ if $p$ is a singularity of $\hat{\chi}_\pi$.
\end{defi}

\subsection{Resolution of singularities of vector fields}

In \cite{panazzolo:ressingrealanalvfdim3}, the author provides an algorithm to resolve singularities for analytic vector fields locally defined at the origin of $\nR^3$.
He shows that, up to a finite sequence of weighted blow-ups, any real analytic vector field can be assumed to have \emph{elementary} singularities.
Up to further blow-ups, one can get even better final normal forms, called \emph{strongly elementary}.

In \cite[Theorem p.281]{mcquillan-panazzolo:almostetaleresfol}, these results have been transported to the complex-analytic case.
In this case the singularities are classified, following the minimal model problem for algebraic varieties, according to positivity properties of the canonical bundle of the associated foliation. Elementary singularities are called here \emph{log-canonical} (see \cite[I.ii.1 Definition]{mcquillan-panazzolo:almostetaleresfol}).
Again, a further improvement can be achieved, obtaining \emph{canonical} singularities.

One of the major difficulties in this setting is that weighted blow-ups don't preserve the class of smooth manifolds: one has to consider some mild singularities, namely, cyclic quotients, which correspond to working with orbifolds.

When studying our example given by \refeqn{example}, we will only need smooth models (see \refprop{exampleresolution}). We recall here the definition of log-canonical singularities in this setting.

\begin{defi}
Let $X$ be a smooth $3$-fold, $D$ a SNC divisor on $X$, and $\chi$ a vector field locally defined at a point $p \in D$.
Then $\chi$ is called \emph{log-canonical} if its $D$-saturation is tangent to $D$, and either regular, or singular at $p$ with a non-nilpotent linear part.
\end{defi}

In general, when working with a cyclic quotient singularity $(X,p)$, we can see it as the quotient of $(\nC^3,0)$ by the action of some finite group $\Gamma$.
Then a log-canonical foliation on $(X,p)$ is induced by a log-canonical $\Gamma$-invariant foliation on $(\nC^3,0)$ (see \cite[I.ii.5 Fact/Definition]{mcquillan-panazzolo:almostetaleresfol}).

We also need to recall the definition of (isolated) canonical singularities, (see \cite[III.i.2 Definition and III.i.3 Fact]{mcquillan-panazzolo:almostetaleresfol}).

\begin{defi}
Let $X$ be a smooth $3$-fold, $D$ a SNC divisor on $X$, and $\chi$ a saturated vector field at $X$ with an isolated singularity at $p \in D$.
Then $\chi$	is called ($D$-)\emph{radial} if it is tangent to $D$ and its linear part has eigenvalues $(\lambda_1, \lambda_2, \lambda_3) \in \lambda (\nN^*)^3$ for some $\lambda \neq 0$.

A vector field $\chi$ as above is called ($D$-)\emph{canonical} if it is ($D$-)log-canonical, but not ($D$-)radial.
\end{defi}

The reduction of singularities for vector fields can be then stated as follows.

\begin{thm}[{\cite[Theorem p.281]{mcquillan-panazzolo:almostetaleresfol}}]\label{thm:mcquillan-panazzolo-resolution}
Let $(X,\mc{F})$ be a holomorphic foliation by curves on a $3$-manifold $X$.
Then there exists a sequence of weighted blow-ups $\pi:(X_\pi, D_\pi, \mc{F}_\pi) \to (X,\mc{F})$ so that $\mc{F}_\pi$ has only log-canonical singularities. 
\end{thm}
Moreover, ``log-canonical'' in the previous statement can be replaced with ``canonical'' by \cite[III.ii.2 Resolution]{mcquillan-panazzolo:almostetaleresfol}.
In the present paper, both log-canonical and canonical singularities are considered (without further mention) with respect to the exceptional divisor $D$ whose support is $\pi^{-1}(0)$.

\subsection{Parabolic manifolds}
\begin{defi}
Let $f:(\nC^d,0)\to (\nC^d,0)$ be a {\tid} germ. 
A \emph{parabolic manifold} for $f$ is a complex manifold $\Delta \subseteq \nC^d$ of positive dimension such that
\begin{itemize}
\item $0 \in \partial \Delta$;
\item $\Delta$ is $f$-invariant, and $f^n(z) \to 0$ for all $z \in \Delta$ as $n \to +\infty$, uniformly on compact subsets of $\Delta$.
\end{itemize}
When moreover it has dimension $1$ (resp., dimension $d$), it is called a \emph{parabolic curve} (resp., \emph{parabolic domain}).
\end{defi}

\begin{rmk}
Sometimes parabolic manifolds are also asked to be simply connected, and not simply connected parabolic manifolds are sometimes called stable manifolds.
To avoid confusion with respect to the classical stable/unstable manifolds, we will stick with the terminology of ``parabolic manifolds'', and specify if they are simply connected if necessary.
\end{rmk}

Parabolic manifolds are often attached to complex directions, in the following sense.
\begin{defi}
Let $f:(\nC^d,0)\to (\nC^d,0)$ be a {\tid} germ. 
Denote by $[\cdot]$ the canonical projection from $\nC^d\setminus \{0\}$ to $\nP_{\nC}^{d-1}$, and let $v \in \nP_{\nC}^{d-1}$ be a tangent direction at $0$.
Let $p$ be a point in $\nC^d$.
We say that its orbit \emph{converges to the origin tangent to $v$} if $f^n(p) \to 0$ and $[f^n(p)] \to v$ when $n \to +\infty$.

We say that a parabolic manifold $\Delta$ for $f$ is tangent to $v$ if the orbit of every point $p \in \Delta$ converges to the origin tangent to $v$.
\end{defi}

\begin{prop}[{\cite[Proposition 2.3]{hakim:analytictransformations}}]
Let $f:(\nC^d,0)\to (\nC^d,0)$ be a \tid\ germ.
If the orbit of a point converges to the origin tangent to a direction $v$ then $v$ is a characteristic direction.
\end{prop}

The following result is a geometric reformulation of \cite[Proposition 3.1]{abate-tovena:paraboliccurvesC3}.

\begin{prop}\label{prop:checkonlysingularup}
Let $f:(\nC^d,0)\to (\nC^d,0)$ be a {\tid} germ.
Suppose that there exists an effective divisor $D$ with simple normal crossings (SNC) at $0$ and supported in $\on{Fix}{f}$, so that the $D$-saturated infinitesimal generator $\hat{\chi}$ of $f$ is regular at $0$, and tangent to $D$.
Then no inﬁnite orbit for $f$ can stay arbitrarily close to $0$.
\end{prop}
\begin{proof} In what follows, $x^a = x_1^{a_1} \cdots x_d^{a_d}$.
By our assumptions, we can find local coordinates at $0$ so that $D=\{x^a=0\}$ for some $a \in \nN^d$, and
$$
f(x)=(x + x^a g(x)).
$$
Here $g:(\nC^d,0) \to \nC^d$ is a holomorphic map, with homogeneous part of smallest degree denoted by $G$. The multiplication $x^a g(x)$ is meant as the product of a scalar $x^a$ and a vector $g(x)$.

The saturated infintesimal generator $\hat{\chi}$ is tangent to $D$ if and only if $x_k | x_k \circ g$ for all $k$ satisfying $a_k > 0$.
It is regular if and only if there exists $k$ so that $a_k = 0$ and $x_k \circ g(0) \neq 0$, where $x_k \circ g$ is the $k$-th coordinate of $g$.

The result follows from \cite[Proposition 3.1]{abate-tovena:paraboliccurvesC3}.
\end{proof}

Suppose we have a {\tid} germ $f:(\nC^d,0)\to (\nC^d,0)$. We apply the previous proposition to the lift of $f$ to the blow-up of $0$, obtaining the following.

\begin{cor}\label{cor:singularsuffices}
Let $f:(\nC^d,0)\to (\nC^d,0)$ be a {\tid} germ, and let $D$ be a (possibly trivial) SNC divisor with support contained in $\on{Fix}(f)$.
Suppose that $f$ is not dicritical, and the saturation of the infinitesimal generator $\hat{\chi}$ of $f$ is tangent to $D$.

If an orbit converges to $0$ tangent to a (characteristic) direction $v$, then $v$ is singular (with respect to $D$).
\end{cor}
\begin{proof}
Let $\pi:X_\pi \to (\nC^d,0)$ be the blow-up of the origin, and let $f_\pi$ be the lift of $f$ on $X_\pi$.
The condition on the non-dicriticity of $f$ corresponds to the fact that the saturation $\hat{\chi}_\pi$ of the infinitesimal generator of $f_\pi$ is tangent to the exceptional divisor $E=\pi^{-1}(0)$.
Together with the hypothesis of tangency to $D$, we get that $\hat{\chi}_\pi$ is tangent to $\pi^{-1}(D) \cup E$.

Finally, a direction $v$ is singular with respect to $D$ if and only if $\hat{\chi}_\pi$ is singular at the associated point $p \in E$.

We conclude by \refprop{checkonlysingularup}.
\end{proof}

\refcor{singularsuffices} can be restated in terms of point blow-ups.
Under the same assumptions (and using the same notations as in the proof), if a orbit converges to a point $p \in \pi^{-1}(0)$, then $p$ is a singular point for $\hat{\chi}_\pi$.

In general, \refprop{checkonlysingularup} forces $\hat{\chi}_\pi$ to be either singular at $p$, or regular at $p$ and transverse to the exceptional divisor.
The latter case is excluded thanks to the non-dicriticity hypothesis on $f$.

When working with blow-up of curves, we lack the correspondence between characteristic directions of $f$ and singular points of $f_\pi$, so we apply directly \refprop{checkonlysingularup} in this case.

Verifying these conditions during the proof of \refthm{mainnonondeg} is straightforward and left to the reader. 

\subsection{Invariant curves and point modifications}

Point modifications allow to study (formal) curves.
We first introduce some terminology.
\begin{defi}
An \emph{increasing sequence of infinitely near points} (above the origin) is a sequence $\mf{p}=(p_n)_{n \in \nN}$ of infinitely near points, which starts with $p_0 = 0 \in \nC^d$ and satisfying the following property: for any $n \in \nN$, $p_{n+1}$ is a point in the exceptional divisor of the blow-up $\pi_n:X_{n+1} \to X_n$ of $p_n$ (where $X_0= \nC^d$).
We set $\hat{\pi}_n = \pi_0 \circ \ldots \circ \pi_{n-1}:X_{n} \to X_0$.
\end{defi}

\begin{prop}\label{prop:constructioncurve}
Let $\mf{p}=(p_n)_n$ be an increasing sequence of infinitely near points.
Suppose that for any $n$, $p_n$ is a smooth point of $\hat{\pi}_n^{-1}(0)$, i.e., it belongs to $\pi_{n-1}^{-1}(p_{n-1})$ but not to the strict transform of $\hat{\pi}_{n-1}^{-1}(0)$.

Then there exists a unique (possibly non-convergent) smooth curve $C=C_{\mf{p}}$, with the property that the strict transform $C_n$ of $C$ with respect to $\hat{\pi}_n$ passes through $p_n$.
\end{prop}
\begin{proof}
This can be done explicitly as follows.
Without losing generality, we may assume that $p_1$ is the point associated to the direction $[a^{(1)}_1: \cdots : a^{(1)}_{d-1}:1]$ for some $a^{(1)}=(a^{(1)}_1, \ldots, a^{(1)}_{d-1}) \in \nC^{d-1}$.
This allows us to make computations in the $x_d$-chart, and write $\pi_1(x_1, \ldots, x_d) = (x_1 x_d, \ldots, x_{d-1} x_d, x_d)$.
We now take the local coordinates $(x_1-a^{(1)}_1, \ldots, x_{d-1}-a^{(1)}_{d-1})$ at $p_1$.
The smoothness hypothesis ensures that $p_2$ is associated to a point of the form $[a^{(2)}:1]$ with $a^{(2)} \in \nC^{d-1}$.
By induction we obtain that $p_n$ is associated to a point of the form $(a^{(n)}:1)$ for some $a^{(n)} \in \nC^{d-1}$, when all computations for $\pi_n$ are made in the $x_d$-chart (after having translated coordinates as shown above).

We consider the curve $C$, parametrized by $\big(x_1(t), \ldots, x_{d-1}(t), t\big)$, where
$$
x_k(t)=\sum_{n=1}^\infty a^{(n)}_k t^n\text.
$$
It is a simple computation to show that $C$ satisfies the statement. Moreover, a curve $C$ tangent to a vector of the form $(a:1)$ is parametrized uniquely as $(x_1(t), \ldots, x_{d-1}(t),t)$ for some formal power series $x_k(t) \in \nC\fps{t}$ for $k=1, \ldots, d-1$, whose linear terms are uniquely determined by $a\in\nC^{d-1}$,  from which we infer the uniqueness of $C$.
\end{proof}

\begin{rmk}
Notice also that if an increasing sequence of infinitely near points does not satisfy the condition of \refprop{constructioncurve}, at least starting from a certain $n_0$, then it does not identify a curve (not even singular).
In fact, any truncation $(p_n)_{n \leq m}$ identifies a set $\mc{C}_m$ of curves tangent to them.
Every time there is a singular point $p_n$ of $\hat{\pi}_n^{-1}(0)$, followed by a smooth point $p_{n+1}$, the minimal multiplicity of the curves in $\mc{C}_m$ strictly increases.
Since curves are desingularized by blowing-up points (for irreducible curves, the intersection of the strict transform of the curve with the exceptional divisor, see e.g. \cite[Section 3.2]{campillo-castellanos:curvesingularities}), and smooth curves are characterized by sequences of smooth infinitely near points, the condition in \refprop{constructioncurve} is also necessary.
\end{rmk}

Given an irreducible curve $C$, we denote by $\mf{p}=\mf{p}(C)$ the increasing sequence of infinitely near points attached to it, starting with $p_0=0\in\nC^d$. 

We want to apply \refprop{constructioncurve} to increasing sequences of infinitely near points which are singular points for the lifts of a {\tid} germ.

\begin{prop}\label{prop:invariantcurve}
Let $f:(\nC^d,0) \to (\nC^d,0)$ be a {\tid} germ, and $\mf{p}=(p_n)_n$ be an increasing sequence of infinitely near points satisfying the hypothesis of \refprop{constructioncurve}.
Let $C=C_{\mf{p}}$ be the formal curve associated to $\mf{p}$.
If $p_n$ are singular points for the lift of $f$ on $X_n$ for all $n\in \nN$, then $C$ is $f$-invariant.
\end{prop}
\begin{proof}
Denote by $f_n:X_n \dashrightarrow X_n$ the lifts of $f$ with respect to $\hat{\pi}_n$.
Being $p_n$ a singular point for $f_n$, we have in particular that $f_n(p_n)=p_n$.
It follows that $f(C)$ is an irreducible curve whose strict transform with respect to $\hat{\pi}_n$ passes through $p_n$.
By \refprop{constructioncurve}, this is exactly the curve $C_{\mf{p}}$.
\end{proof}

The invariant curves constructed here are sometimes called \emph{(strict) separatrices} for the {\tid} germ $f$ (see \cite{lopezhernanz-rosas:chardirdim2} for the analogous in dimension $2$).
They are in fact the analogous of separatrices for the (reduced) infinitesimal generator (see \cite{brochero-cano-lopezhernanz:parabolicrcurvesdiffeodim2}).

\begin{rmk}
Notice that \refprop{constructioncurve} and \refprop{invariantcurve} do not hold if we replace point modifications with sequences of blow-ups of centers with positive dimension.
The main reason is that curves are not anymore uniquely determined by the sequence of points of intersection of their strict transform with the exceptional divisor.

As an example, consider the blow-up the line $\{x=z=0\}$ in $\nC^3$, and coordinates in the blown-up space so that the projection takes the form $\pi(x,y,z)=(xz,y,z)$. Reiterate the process, so to construct a sequence $\pi_n:X_n \to (\nC^3,0)$, where each element consists in the blow-up of $n$ lines.
In this case, for any curve $C$ parametrized by $(0, y(z),z)$ (with $y \in z \nC\fps{z}$ a formal power series with vanishing constant term), its strict transform $C_n$ would intersect $\pi_n^{-1}(0)$ at the origin $p_n$ of the corresponding chart.
In particular, if the points $p_n$ are singular for the lifts $f$ of a {\tid} germ $f:(\nC^3,0)\to (\nC^3,0)$ we could only infer that $f(C)$ is another curve lying in the plane $\{x=0\}$.
\end{rmk}

\subsection{Parabolic manifolds asymptotic to invariant curves}

We have seen how orbits of points converging to the origin must be tangent to a characteristic direction.
One could be interested in controlling higher orders of tangency. This corresponds to imposing conditions on lifts to other birational models. We need here some terminology to deal with these conditions, which are expressed in terms of asymptoticity to (formal invariant) curves (see \cite{lopezhernanz-raissy-ribon-sanzsanchez:stablemfld2dim, lopezhernanz-ribon-sanzsanchez-vivas:stablemanifoldsbiholoCn}).

\begin{defi}
Let $f:(\nC^d,0)\to(\nC^d,0)$ be a {\tid} germ.
Let $\mf{p}=(p_n)$ be an increasing sequence of infinitely near points above the origin.
Denote by $\hat{\pi}_n:X_n \to (\nC^d,0)$ the composition of the blow-ups of the points $p_0, \ldots, p_{n-1}$, and by $f_n : X_n \dashrightarrow X_n$ the lift of $f$ to $X_n$.
We say that the orbit of a point $p \in \nC^d \setminus \{0\}$ converges to the origin \emph{asymptotic to $\mf{p}$} if for any $n \in \nN$, the limit of the $f_n$-orbit of $\hat{\pi}_n^{-1}(p)$ is exactly $p_n$.

If $\mf{p}=\mf{p}(C)$ for some irreducible curve $C$, we say that the orbit converges to the origin \emph{asymptotic to $C$}.

We say that a stable manifold $\Delta$ is \emph{asymptotic to $\mf{p}$} (resp., \emph{to $C$})if the orbit of $p$ is asymptotic to $\mf{p}$ (resp., to $\mf{p}(C)$) for any $p \in \Delta$.
\end{defi}

We now recall \cite[Theorem 1]{lopezhernanz-ribon-sanzsanchez-vivas:stablemanifoldsbiholoCn}, which allows to construct parabolic manifolds from formal invariant curves.

\begin{thm}[{\cite[Theorem 1]{lopezhernanz-ribon-sanzsanchez-vivas:stablemanifoldsbiholoCn}}]\label{thm:parabolicmanifolds}
Let $f:(\nC^d,0) \to (\nC^d,0)$ be a {\tid} germ, and let $C$ be a formal invariant curve for $f$.
Then either $C$ is contained in $\on{Fix}(f)$, or there exist finitely many parabolic manifolds asymptotic to $C$.
\end{thm}

\subsection{Ramis-Sibuya normal forms}

To describe precisely the number and dimension of the parabolic manifolds produced by \refthm{parabolicmanifolds}, we need to introduce some terminology.

We first introduce Ramis-Sibuya normal forms, for {\tid} germs in dimension $3$.

\begin{defi}
Let $f:(\nC^3,0) \to (\nC^3,0)$ be a tangent to the identity germ, and let $C$ be a smooth $f$-invariant formal curve.
We say that the couple $(f,C)$ is in \emph{Ramis-Sibuya normal form} with respect to local coordinates $(x,y,z)$ at the origin, if $C$ is transverse to $\{z=0\}$, and $f$ takes the form
\begin{equation}\label{eqn:RamisSibuya}
f(x,y,z) =
\begin{pmatrix}
\exp(d_1(z)) \big(x (1+ c_{11} z^r) + c_{12} y z^r\big)  + \langle z^{r+1}\rangle\\[3mm]
\exp(d_2(z)) \big(y (1+ c_{22} z^r) + c_{21} x z^r\big)  + \langle z^{r+1}\rangle\\[3mm]
z - z^{r+1} + bz^{2r+1} + \langle z^{2r+2}\rangle
\end{pmatrix}
\text,
\end{equation}
where $d_1$ and $d_2$ are polynomials of degree at most $r-1$ vanishing at the origin, and $c_{12}=c_{21}=0$ unless $d_1 \equiv d_2$.
\end{defi}

Notice that on the $z$-coordinate, assuming that $C$ has sufficiently high tangency with $\{x=y=0\}$, we find the formal normal form of the action of $f|_C$.
In particular the existence of such a normal form implies that $f|_C$ defines a parabolic $1$-dimensional germ of multiplicity $r+1$.

\begin{thm}[{\cite[Theorem 5.11]{lopezhernanz-ribon-sanzsanchez-vivas:stablemanifoldsbiholoCn}}]\label{thm:reductiontoRS}
Let $f:(\nC^3,0)$ be a {\tid} germ, admitting an (irreducible) $f$-invariant formal curve $C$.
Suppose that $f|_C \neq \id$.
Then there exists a sequence of weighted blow-ups $\pi:X_\pi \to (\nC^3,0)$ so that the strict transform $C_\pi$ of $C$ is smooth, and, if $f_\pi :(X_\pi,p) \to (X_\pi,p)$ denotes the lift of $f$ at $p=C_\pi \cap \pi^{-1}(0)$, then $(f_\pi, C_\pi)$ is in Ramis-Sibuya normal form.
\end{thm}

In \cite{lopezhernanz-ribon-sanzsanchez-vivas:stablemanifoldsbiholoCn}, the authors show the reduction (up to taking iterates) to Ramis-Sibuya normal form in the more general setting of automorphisms admitting an $f$-invariant formal curve where $f|_C$ has multiplier $1$ (in particular, the linear part of $f$ does not need to be the identity).

The reduction process consists in three steps.
The first consists in an embedded resolution of $C$.
In the second step, one applies \refthm{mcquillan-panazzolo-resolution} to solve the singularities of the infinitesimal generator of $f$.
The third step reduces the pair $(f_\pi,C_\pi)$ to the desired normal form, by performing further blow-ups. 
Notice that both the second and third steps may require weighted blow-ups.

Once the couple $(f,C)$ is reduced in Ramis-Sibuya normal form, one can describe explicitly the number and dimension of the parabolic manifolds provided by 
\refthm{parabolicmanifolds}.

With the notations of \refeqn{RamisSibuya}, write $d_j$ for $j=1,2$ as
$$
d_j(z)=\sum_{k=1}^{r-1} d_{k}^{(j)} z^k\text.
$$
Given an attracting direction $\xi$ for $f|_C$ (i.e., any complex $r$-th root of $1$), we set
\begin{equation}\label{eqn:RSdefRj}
R_j(\xi)=\big(\Re(d_1^{(j)}\xi), \ldots, \Re(d_{r-1}^{(j)}\xi^{r-1})\big)\text.
\end{equation}
\begin{defi}
We say that $\xi$ is a \emph{node direction} for the variable $x$ (resp., $y$) if
$R_1(\xi) < 0$ (resp., $R_2(\xi < 0)$), and a \emph{saddle direction} otherwise, where $<$ denotes the lexicographic order.
\end{defi}

\begin{thm}[{\cite[Theorem 6.1]{lopezhernanz-ribon-sanzsanchez-vivas:stablemanifoldsbiholoCn}}]\label{thm:RSparabolicmanifolds}
Let $f:(\nC^3,0)\to (\nC^3,0)$ be a {\tid} germ, and let $C$ be an $f$-invariant formal curve.
Suppose that $(f,C)$ is in Ramis-Sibuya normal form.
For any attracting direction $\xi$ for $f|_\Gamma$, let $s=s(\xi) \in \{0, 1, 2 \}$ be the number of variables for which $\xi$ is a node direction.
Then there exists a parabolic manifold $\Delta(\xi)$ asymptotic to $C$, of dimension $s(\xi)+1$, which is connected, simply connected, and which is a fundamental domain for the set of points whose orbit converges to $0$ asymptotic to $C$ and tangent to $\xi$.	
\end{thm}

\section{The example}\label{sec:example}

\subsection{Rational maps with no holomorphic fixed points}

We want to start with a {\tid}  germ $f$ which has a finite number of characteristic directions, all degenerate.

Recall that if $f:(\nC^3,0) \to (\nC^3,0)$ is a {\tid} germ, its characteristic directions can be reinterpreted in terms of the action induced by the homogeneous part of smallest degree of $f-\id$ on $\nP^2$:
non-degenerate characteristic directions correspond to holomorphic fixed points, while degenerate characteristic directions correspond to indeterminacy points.

By abuse of notation, we denote by $H$ both the homogeneous part of smallest degree of $f-\id$, and the action on $\nP^2$ induced by it.
We start with a rational map in $\nP^2$ which has no holomorphic fixed points (see \cite[Example 2.1]{ivashkovich:fixedpointsratmaps}):
$$
H([x:y:z])=\big[yz(y-z) : x(x^2-z^2) : z+xz(y-z)\big].
$$
Hence we focus on germs of the form:
\begin{equation*}\tag{\ref{eqn:example}}
f(x,y,z)=
\begin{pmatrix}
\displaystyle x+yz(y-z)+P\\[2mm]
\displaystyle y+x(x^2-z^2)+Q\\[2mm]
\displaystyle z+xz(y-z)+R
\end{pmatrix}
\text,
\end{equation*}
with $P,Q,R$ of order at least $4$.

When we need to develop $P, Q, R$ in formal power series, we will use the following notations:
$$
P=\sum_{i,j,k} P_{ijk} x^i y^j z^k, \quad
Q=\sum_{i,j,k} Q_{ijk} x^i y^j z^k, \quad
R=\sum_{i,j,k} R_{ijk} x^i y^j z^k,
$$
where the indices $i,j,k$ vary in $\nN$ with $i+j+k \geq 4$.
We will also denote by $P^{(h)}$ (resp., $Q^{(h)}$, $R^{(h)}$) the homogeneous part of degree $h$ of $P$ (resp., $Q$, $R$).

\subsection{Characteristic directions}

One can easily check that $f$ has $5$ characteristic directions, given by $v_1=[0:0:1], v_2=[0:1:1], v_3=[1:1:1], v_4=[-1:1:1], v_5=[0:1:0]$.
All these directions are degenerate, of multiplicities $1,1,3,3,5$ respectively.
We denote by $p_1, p_2, p_3, p_4, p_5$ the corresponding characteristic points.

For a definition the multiplicity $\mu_f(v)$ of a characteristic direction $v$, see \cite[p. 278]{abate-tovena:paraboliccurvesC3}.
For the reader's convenience, we show here how to compute the multiplicity of $v_5$. The multiplicity $\mu_f([0:1:0])$ is the local intersection multiplicity at $[0:1:0]$ of $y (x \circ f^{(3)})-x(y \circ f^{(3)})$ and $y (z \circ f^{(3)})-z(y \circ f^{(3)})$ in $\nP^2$.
We denote by $\locint{\phi}{\psi}{p}$ the local intersection multiplicity at $p$ of $\{\phi=0\}$ and $\{\psi=0\}$.
In this case, computing the intersection in the chart $\{y=1\}$, we obtain
\begin{align*}
\mu_f([0:1:0]) 
& = \locint{z(1-z)-x^2(x^2-z^2)}{xz(1-z) - xz(x^2-z^2)}{0} \\
& = \locint{z-z^2-x^4 + x^2z^2}{xz(1-z-x^2+z^2)}{0} \\
& = \locint{z-z^2-x^4 + x^2z^2}{x}{0} + \locint{z-z^2-x^4 + x^2z^2}{z}{0} 
\\
& = 1+4
=5.
\end{align*}

Computations for the other multiplicities are similar and left to the reader.

\begin{rmk}\label{rmk:symmetry}
Consider the local diffeomorphism $\sigma(x,y,z)=(-\ui x,\ui y,\ui z)$.
Then we get
$$
\sigma^{-1} \circ f \circ \sigma(x,y,z)=
\begin{pmatrix}
\displaystyle x+yz(y-z)+\ui P \circ \sigma \\[2mm]
\displaystyle y+x(x^2-z^2) -\ui Q \circ \sigma\\[2mm]
\displaystyle z+xz(y-z)- \ui R \circ \sigma
\end{pmatrix}
\text.
$$
In particular, the $3$-jet $f^{(\leq 3)}$ is invariant by this conjugacy.
The action of $\sigma$ on the characteristic directions $v_1, \ldots, v_5$ is a bijection that fixes $v_1$, $v_2$, $v_5$ while it exchanges $v_3$ and $v_4$.

It follows that one can recover the birational study of the lifts of $f$ above the point $p_4$ associated to $v_4$ from the behaviour of the lifts of $f$ above $p_3$.
\end{rmk}

\subsection{Resolution of singularities of the infinitesimal generator}\label{ssec:resolutionexample}

In this section, we provide a resolution of the infinitesimal generator $\chi$ of a {\tid} germ $f:(\nC^3,0)\to (\nC^3,0)$ of the form \refeqn{example}, in the sense of \cite{mcquillan-panazzolo:almostetaleresfol} (see \refprop{exampleresolution}).

In \refsec{biratstudyaboveresol}, we will show that after further blow-up (see \refprop{exampleresolutionforms}), all the singularities will be isolated and belonging to one of three classes (simple corners, {\degspike}s and {\spincorner}s).
We will then study the behavior of these families under point modifications (introducing a fourth family, {\halfcorner}s).

We will finally study the behavior of these families under general admissible modifications (strongly adapted to the dynamics) in \refsec{patterns}.

\medskip

By \refrmk{homocontribution}, the infinitesimal generator $\chi$ of $f$ takes the form
$$
\chi = \big(yz(y-z)+ P^{(4)}\big) \partial_x + \big(x(x^2-z^2)+ Q^{(4)}\big) \partial_y + \big(xz(y-z)+ R^{(4)}\big) \partial_z + \xi,
$$
where $\xi$ is a (possibly formal) vector field of multiplicity at least $5$.

To study the resolution of $\chi$, we will perform computations from the point of view of maps instead of vector fields, relying on \refprop{infgenblowup}.

\medskip

\subsubsection{First blow-up.}

To resolve the singularities of $\chi$, we first blow-up the origin. We write the blow-up $\pi_1:X_1 \to (\nC^3,0)$, and compute the lift $f_1$ of $f$ with respect to $\pi_1$.
By \refprop{singularitiestidinfgen}, the singularities of the saturated infinitesimal generator $\hat{\chi}_1$ of $f_1$ are isolated, given by the points $p_1, p_2, p_3, p_4, p_5$.

We first work in the $z$-chart, for which the map $\pi_1$ is written as $\pi_1(x,y,z)=(xz,yz,z)$.
We obtain:
\begin{equation}\label{eqn:blowupIz}
f_1(x,y,z) =
\begin{pmatrix}
\displaystyle \frac{x- z^2 y (1-y) + z^{-1} P\circ \pi_1}{1-z^2x(1-y) + z^{-1} R\circ \pi_1}\\[6mm]
\displaystyle \frac{y-z^2 x (1-x^2) + z^{-1} Q\circ \pi_1}{1-z^2x(1-y) + z^{-1} R\circ \pi_1}\\[6mm]
\displaystyle z\big(1-z^2x(1-y) + z^{-1} R\circ \pi_1\big)
\end{pmatrix}
\text.
\end{equation}

We rewrite $f_1$ developing around a point in $\{z=0\}$, obtaining
\begin{equation}\label{eqn:taylorIz}
f_1(x,y,z) =
\begin{pmatrix}
x + z^2 (-y+x^2+y^2-x^2y) + z^3 (P^{(4)}-xR^{(4)})(x,y,1) + \langle z^4\rangle\\[3mm]
y + z^2 x (-1+y+x^2-y^2) + z^3 (Q^{(4)}-yR^{(4)})(x,y,1) + \langle z^4\rangle\\[3mm]
z + z^3x(-1+y) + z^4 R^{(4)}(x,y,1) + \langle z^5\rangle
\end{pmatrix}
\text.
\end{equation}

We study $f_1$ around the characteristic points $p_1, \ldots, p_4$.
The point $p_1$ corresponds to the origin in this chart. In this case the linear part of the reduced infinitesimal generator is:
$$
(-y+P_{004}z) \partial_x + (-x+Q_{004}z) \partial_y\text.
$$
Hence $\hat{\chi}_1$ has a canonical singularity at $p_1$, with eigenvalues of the linear part given by $1$, $-1$ and $0$.

Similarly, the point $p_2$ corresponds to $(0,1,0)$ in this chart. By setting $y=1+v$, we get
\begin{equation}\label{eqn:taylorIzp2}
f_1(x,v,z) =
\begin{pmatrix}
x + z^2 v(1+v-x^2) + z^3 (P^{(4)}-xR^{(4)})(x,1+v,1) + \langle z^4\rangle\\[3mm]
v + z^2 x(-1-v+x^2-v^2) + z^3 (Q^{(4)}-(1+v)R^{(4)})(x,1+v,1) + \langle z^4\rangle\\[3mm]
z + z^3xv + z^4 R^{(4)}(x,1+v,1) + \langle z^5\rangle
\end{pmatrix}
\text.
\end{equation}
We get again a canonical singularity, with eigenvalues of the linear part $\ui$, $-\ui$ and $0$.

The points $p_3$ and $p_4$ have coordinates $(1,1,0)$ and $(-1,1,0)$ respectively in the $z$-chart.
We treat $p_3$, the case of $p_4$ being completely analogous by \refrmk{symmetry}.
By setting $x=1+u$, we get
\begin{equation}\label{eqn:taylorIzp3}
f_1(u,v,z) =
\begin{pmatrix}
u + z^2 v(-2u+v-u^2) + z^3 (P^{(4)}-(1+u)R^{(4)})(1+u,1+v,1) + \langle z^4\rangle\\[3mm]
v + z^2 (1+u)(2u-v+u^2-v^2) + z^3 (Q^{(4)}-(1+v)R^{(4)})(1+u,1+v,1) + \langle z^4\rangle\\[3mm]
z + z^3(1+u)v + z^4 R^{(4)}(1+u,1+v,1) + \langle z^5\rangle
\end{pmatrix}
\text.
\end{equation}
In this case the linear part of $\hat{\chi}_1$ is
$$
z\big(P^{(4)} - R^{(4)}\big)(1,1,1) \partial_u + \Big(2u-v+z\big(Q^{(4)}-R^{(4)}\big)(1,1,1)\Big)\partial_v\text,
$$
which gives an isolated canonical singularity with eigenvalues $-1$, and $0$ (of multiplicity $2$).

It remains to study the characteristic direction $v_5=[0:1:0]$.
In this case we work in the $y$-chart, and write $\pi_1(x,y,z)=(xy,y,yz)$.
The lift $f_1$ of $f$ takes the form

\begin{equation}\label{eqn:blowupIy}
f_1(x,y,z) =
\begin{pmatrix}
\displaystyle \frac{x +y^2 z (1-z) + y^{-1} P\circ \pi}{1+y^2x(x^2-z^2) + y^{-1} Q\circ \pi}\\[6mm]
\displaystyle y\big(1+y^2x(x^2-z^2) + y^{-1} Q\circ \pi\big)\\[2mm]
\displaystyle \frac{z+y^2 xz (1-z) + y^{-1} R\circ \pi}{1+y^2x(x^2-z^2) + y^{-1} Q\circ \pi}
\end{pmatrix}
\text.
\end{equation}
The taylor expansion at the origin gives the following expression for $f_1(x,y,z)$:
\begin{equation}\label{eqn:taylorIy}
\begin{pmatrix}
x + y^2 \Big(z-z^2 -x^4+x^2z^2 + P_{040}y +(P_{130}-Q_{040})xy + P_{031}yz+P_{050}y^2 +y\mf{m}^2\Big)\\[3mm]
y + y^3 \Big(x^3-xz^2+ Q_{040}y + y\mf{m}\Big)\\[3mm]
z + y^2\Big(xz-xz^2 -x^3z+xz^3 + R_{040}y +R_{130}xy + (R_{031}-Q_{040})yz+R_{050}y^2  +y \mf{m}^2\Big)
\end{pmatrix}
\text.
\end{equation}

The linear part of the reduced infinitesimal generator is:
$$
(P_{040}y+z) \partial_x + R_{040}y \partial_z\text.
$$
We get a nilpotent linear part (of rank $1$ if $R_{040}=0$, and of rank $2$ otherwise).
In this case the singularity is not log-canonical, and we need to keep blowing-up.

\medskip

\subsubsection{Second blow-up.}

For simplicity, we will assume $R_{040} \neq 0$.
In this case, $f_1$ has only one singular direction $v_{5,1}=[1:0:0]$.
Consider the blow-up $\pi_2:X_2 \to X_1$ of the point $p_5$. In the $x$-chart we have $\pi_2(x,y,z)=(x,xy,xz)$.
Set $\hat{\pi}_{2}(x,y,z)=\pi_1 \circ \pi_2(x,y,z)=(x^2y,xy,x^2yz)$.
The lift $f_2$ of $f$ in $X_2$ is given by
\begin{equation}\label{eqn:blowupIyIIx}
f_2(x,y,z) =
\begin{pmatrix}
\displaystyle x \frac{1 + x^2y^2 z (1-xz) + x^{-2}y^{-1} P\circ \hat{\pi}_{2}}{1+x^5y^2(1-z^2) + x^{-2}y^{-1} Q\circ \hat{\pi}_{2}}\\[6mm]
\displaystyle y \frac{\big(1+x^5y^2(1-z^2) + x^{-2}y^{-1} Q\circ \hat{\pi}_{2}\big)^2}{1 + x^2y^2 z (1-xz) + x^{-2}y^{-1} P\circ \hat{\pi}_{2}}\\[6mm]
\displaystyle \frac{z+x^3y^2z (1-xz) + x^{-2}y^{-1} R\circ \hat{\pi}_{2}}{1 + x^2y^2 z (1-xz) + x^{-2}y^{-1} P\circ \hat{\pi}_{2}}
\end{pmatrix}
\text.
\end{equation}
We rewrite $f_2$ developing around the origin, obtaining
\begin{equation}\label{eqn:taylorIyIIx}
\hspace{-1cm}
f_2(x,y,z) =
\begin{pmatrix}
x + x^3y^2\Big(P_{040}y +z + \big(P_{130}-Q_{040}\big)xy+ \mf{m}^3\Big)\\[3mm]
y + x^2 y^3 \Big(-P_{040}y -z  + \big(2Q_{040}-P_{130}\big)xy+\mf{m}^3\Big)\\[3mm]
z + x^2y^2\Big(R_{040}y + R_{130}xy +xz-z^2 - P_{040} yz + \mf{m}^3\Big)
\end{pmatrix}
\text.
\end{equation}
In this case the linear part is again nilpotent, of rank $2$, and the singular directions are given by the line $[p:0:r]$, where $[p:r] \in \nP_\nC^1$.

\medskip

\subsubsection{Third blow-up.}

Let $\pi_3:X_3 \to X_2$ be the blow-up of the point $p_{5,1}$ corresponding to the origin in the last coordinate chart we considered. To study the singular points associated to $f_2$, we will need to consider two different charts.

First, in the $x$-chart we get
\begin{equation}\label{eqn:taylorIyIIxIIIx}
f_3(x,y,z) =
\begin{pmatrix}
x + x^6y^2\Big(P_{040}y+z -x^2 + \big(P_{130}-Q_{130}\big)xy -x^2z^2 + \langle x^2y\rangle\Big)\\[3mm]
y + x^5 y^3 \Big(-2P_{040}y -2z + 3x^2+\big(3Q_{040}-2P_{130}\big)xy  +2x^2z^2 +\langle x^2y\rangle\Big)\\[3mm]
z + x^4y^2\Big(R_{040}y + R_{130}xy+xz-2P_{040}xyz-2xz^2 + x^2\langle x,y\rangle\Big)
\end{pmatrix}
\text.
\end{equation}
For any $z_0 \in \nC$, the saturated infinitesimal generator $\hat{\chi}_3$ of $f_3$ has a singularity at $(0,0,z_0)$, with nilpotent linear part of rank $2$.
In this case the singular directions of $f_3$ form the line $[pR_{040}:pz_0(2z_0-1):r]$ with $[p:r]$ varying in $\nP_\nC^1$.

We now work in the $z$-chart, so that $\pi_3(x,y,z)=(xz,yz,z)$, and get
\begin{equation}\label{eqn:taylorIyIIxIIIz}
f_3(x,y,z) =
\begin{pmatrix}
x + x^3y^2z^4\Big(-R_{040}y+2z-xz+2P_{040}yz -R_{130}xyz + z^2\langle y,z\rangle\Big)\\[3mm]
y + x^2y^3z^4 \Big(-R_{040}y-xz -R_{130}xyz +z^2\langle y,z\rangle\Big)\\[3mm]
z + x^2y^2z^5\Big(R_{040}y -z +xz -P_{040}yz +R_{130}xyz + 	z^2\langle y,z\rangle\Big)
\end{pmatrix}
\text.
\end{equation}
In this case $\hat{\chi}_3$ has a singularity of order $2$ at the origin (and of order $1$ with nilpotent linear part at $(x_0,0,0)$, with $x_0 \neq 0$, that we already know about from the previous computation).
\medskip

\subsubsection{Fourth blow-up.}
Finally, we consider the blow-up $\pi_4:X_4 \to X_3$ along the line $L$ of singular points of $\hat{\chi}_3$.

The line $L$ is covered by two charts in $X_3$, the one where the exceptional divisor is $\{x=0\}$ and the line is given by $L=\{x=y=0\}$, and the one where the exceptional divisor is $\{z=0\}$ and the line is given by $L=\{y=z=0\}$.
This gives a total of four charts to be considered on $X_4$, to cover the exceptional divisor $\pi_4^{-1}(L)$.

We first consider the chart in $X_3$ that gives \refeqn{taylorIyIIxIIIx}, so that $L=\{x=y=0\}$.

We put ourselves in the chart of $X_4$ not intersecting the strict transform of the exceptional divisor $E_3=\{x=0\}$ of $\pi_3$, obtaining $\pi_4(x,y,z)=(x,xy,z)$. Computing the lift of $f_3$, we get
\begin{equation}\label{eqn:taylorIyIIxIIIxIVx}
f_4(x,y,z) =
\begin{pmatrix}
x + x^8y^2\Big(z+P_{040}xy+\langle x^2\rangle\Big)\\[3mm]
y + x^7y^3\Big(-3z -3P_{040}xy+\langle x^2\rangle\Big)\\[3mm]
z + x^7y^2\Big(R_{040}y +z -2z^2 +R_{130}xy -2P_{040}xyz + \langle x^2\rangle\Big)
\end{pmatrix}
\text.
\end{equation}
The saturation $\hat{\chi}_4^x$ of the infinitesimal generator of $f_4$ with respect to $\{x=0\}$ takes the form
$$
\hat{\chi}^x_4= xy^2z\partial_x +y^3\big(-3z -3P_{040}xy\big)\partial_y + y^2 \big(R_{040}y +z -2z^2 +R_{130}xy -2P_{040}xyz \big) \partial_z + x^2\xi\text,
$$
where $\xi$ is a suitable vector field.
We study this vector field on the point $(0,y_0,z_0)$.

If $y_0 \neq 0$, we have that $(0,y_0,z_0)$ is singular if and only if
$$
\begin{cases}
-3y_0z_0=0\text,\\
R_{040}y_0+z_0-2z_0^2=0\text.
\end{cases}
$$
Since $R_{040} \neq 0$, this system does not have solutions.

Suppose now $y_0 = 0$. Then the saturation $\hat{\chi}_4$ with respect to the exceptional divisor, locally given by $\{xy=0\}$, gives
$$
\hat{\chi}_4= xz\partial_x -3yz\partial_y + \big(R_{040}y +z -2z^2\big) \partial_z + \xi'\text,
$$ 
where $\xi'$ is a vector field whose coefficients belong to $x\langle x,y\rangle$.
First notice that $\hat{\chi}_4$ is regular unless $z_0(1-2z_0) = 0$.

At the point $q_1$ corresponding to the value $z_0=0$, $\hat{\chi}_4$ has a linear part with a non-vanishing eigenvalue (of eigenspace generated by $\partial_z$): hence we get an isolated canonical singularity.
Similarly, at the point $q_2$ corresponding to $z_0=\frac{1}{2}$, $\hat{\chi}_4$ has an isolated canonical singularity, with linear part with eigenvalues $\frac{1}{2}(1, -3, -2)$. 

With respect to suitable coordinates in a chart intersecting $E_3$, we get the form $\pi_4(x,y,z)=(xy,y,z)$. For the lift of $f_3$, we get
\begin{equation}\label{eqn:taylorIyIIxIIIxIVy}
f_4(x,y,z) =
\begin{pmatrix}
x + x^6y^7\Big(3P_{040}y+3z+\langle y^2\rangle\Big)\\[3mm]
y + x^5y^8\Big(-2P_{040}y-2z+\langle y^2\rangle\Big)\\[3mm]
z + x^4y^7\Big(R_{040}+x(z-2z^2) + \langle y\rangle\Big)
\end{pmatrix}
\text.
\end{equation}

The saturation $\hat{\chi}_4^y$ of the infinitesimal generator of $f_4$ with respect to $\{y=0\}$ takes the form
$$
\hat{\chi}^y_4= 3x^6z\partial_x + x^4 \big(R_{040} +x(z-2z^2)\big) \partial_z + y\xi\text,
$$
where $\xi$ is a suitable vector field.

We study $\hat{\chi}_4$ at points $(x_0,0,z_0)$. The case $x_0 \neq 0$ corresponds to previous computations, and we have no singularities here.

When $x_0=0$, again we get regular points, hence no singularities arise in this chart.

\medskip

We finally consider the chart in $X_3$ giving \refeqn{taylorIyIIxIIIz}, so that $L=\{y=z=0\}$.

We pick the coordinate chart of $X_4$ not intersecting the strict transform of the exceptional divisor $E_3=\{z=0\}$ of $\pi_3$, obtaining $\pi_4(x,y,z)=(x,yz,z)$.
For the lift of $f_3$, we get
\begin{equation}\label{eqn:taylorIyIIxIIIzIVz}
f_4(x,y,z) =
\begin{pmatrix}
x + x^3y^2z^7\Big(2-x-R_{040}y+\langle z\rangle\Big)\\[3mm]
y + x^2y^3z^7\Big(1-2x-2R_{040}y+\langle z\rangle\Big)\\[3mm]
z + x^2y^2z^8\Big(-1 + x + R_{040}y +\langle z\rangle\Big)
\end{pmatrix}
\text.
\end{equation}
As usual, we denote by $\hat{\chi}_4$ the saturated infinitesimal generator of $f_4$, and study its germ at points $(x_0,y_0,0)$.
At the point $q_3$ corresponding to the origin, we get an isolated canonical singularity, whose linear part has eigenvalues $(2,1,-1)$.

When $x_0=0$ and $y_0 \neq 0$, we have a singularity if and only if $1-2R_{040}y_0=0$, i.e., $y_0=\frac{1}{2R_{040}}$.
At the corresponding point $q_4$, consider local coordinates $(x,v,z)$ with $y=y_0+v$.
In these coordinates, the linear part of $\hat{\chi}_4$ takes the form (up to renormalization of a factor $y_0^2$):
\begin{equation}\label{eqn:linearpartq4}
\frac{3}{2}x \partial_x + (-2y_0x -y)\partial_y -\frac{1}{2}z\partial_z\text,
\end{equation}
hence we get another isolated canonical singularity.

The case $x_0 \neq 0$ corresponds to the study carried on above: we get again a singularity when $x_0=2$ and $y_0=0$, which corresponds to $q_2$.

To finish our study, we consider a chart of $X_4$ intersecting the strict transform of $E_3$, getting $\pi_4(x,y,z)=(x,y,yz)$.
For the lift of $f_3$, we get
\begin{equation}\label{eqn:taylorIyIIxIIIzIVy}
f_4(x,y,z) =
\begin{pmatrix}
x + x^3y^7z^4\Big(-R_{040}+2z-xz+\langle y\rangle\Big)\\[3mm]
y + x^2y^8z^4\Big(-R_{040}-xz+\langle y\rangle\Big)\\[3mm]
z + x^2y^7z^5\Big(R_{040}-z+xz+\langle y\rangle\Big)
\end{pmatrix}
\text.
\end{equation}
The only point $q_5$ that remains to be studied corresponds to the origin in this chart, and $\hat{\chi}_4$ has an isolated canonical singularity there, with linear part having eigenvalues $R_{040}(-1,-1,1)$.

To sum up, we proved the following result.

\begin{prop}\label{prop:exampleresolution}
Let $f:(\nC^3,0) \to (\nC^3,0)$ be a germ of the form \refeqn{example} with $R_{040} \neq 0$.
Let $\pi_0:X_{\pi_0} \to (\nC^3,0)$ be the regular modification obtained as the composition $\pi_0=\pi_1 \circ \pi_2 \circ \pi_3 \circ \pi_4$ descrived above (hence $X_{\pi_0}=X_4$).

Then the reduced infinitesimal generator $\hat{\chi}_{\pi_0}$ of the lift $f_{\pi_0}$ of $f$ at $X_{\pi_0}$ has only isolated canonical singularities, namely $p_1, \ldots, p_4, q_1, \ldots, q_5\in X_{\pi_0}$.
\end{prop}

Notice the abuse of notation, where we denote by $p_1, \ldots, p_4$ both the points in $X_1$, and their unique preimages through $\pi_2 \circ \pi_3 \circ \pi_4$ in $X_4$. 

\begin{figure}[h]
\centering
\begin{minipage}[htbp]{\columnwidth}
\def\svgwidth{0.49\columnwidth}
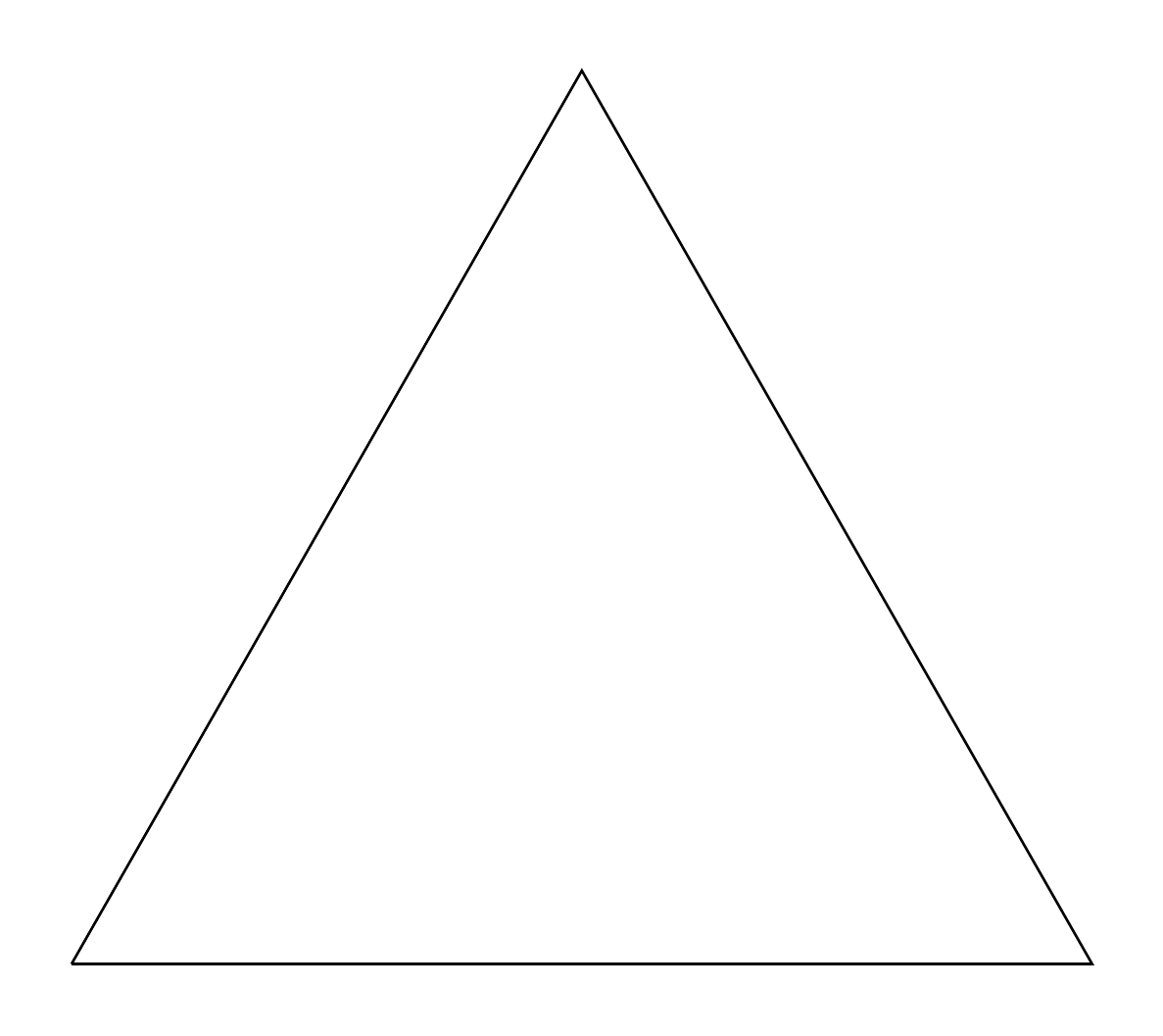
\def\svgwidth{0.49\columnwidth}
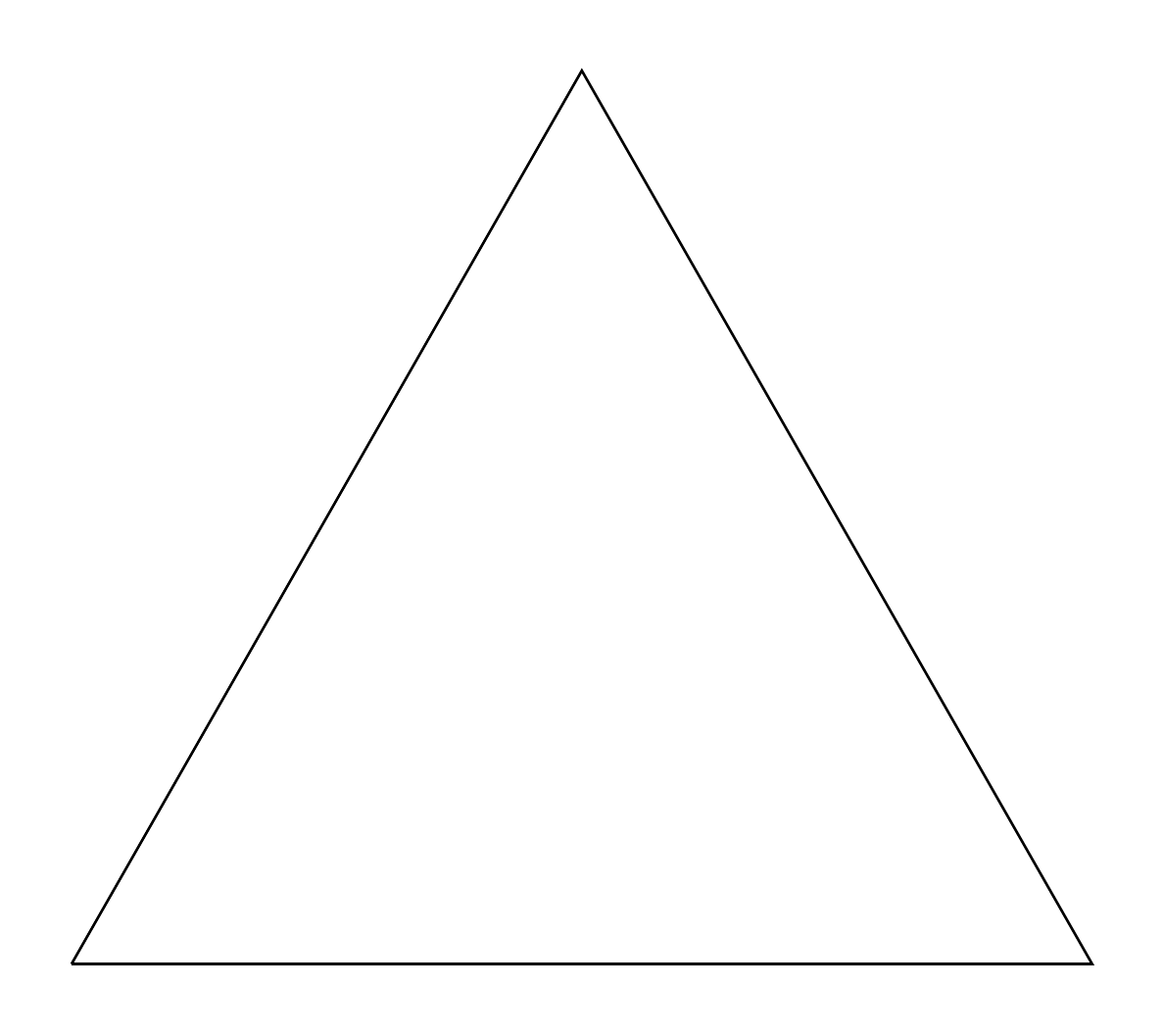
\end{minipage}
\caption{Singular points of the saturated infintesimal generator at $X_1$ (on the left) and $X_4=X_{\pi_0}$ (on the right).}
\label{fig:resolution}
\end{figure}

\begin{rmk}
One can check that Panazzolo's algorithm \cite{panazzolo:ressingvectorfieldsdim3} would perform two weigthed blow-ups to solve $\chi$: the first is the blow-up of the origin, and the second is the blow-up of the point $p_5$, with respect to the weight $\omega=(1,3,2)$, which correspond to performing the blow-ups $\pi_2 \circ \pi_3 \circ \pi_4$, and then contracting the divisors $E_2$ and $E_3$.

To compute $\omega$, notice that the Newton polyhedron (see \cite[Section 3]{panazzolo:ressingvectorfieldsdim3} for a definition) associated to the saturated infinitesimal generator of $f_1$ given by \refeqn{taylorIy} is generated by
$$
(-1,0,1), (-1,1,0), (1,0,0), (0,1,-1)\text.
$$
The first three vertices form a face, whose normal vector is exactly $\omega$.
\end{rmk}

\section{Birational study above the resolution}\label{sec:biratstudyaboveresol}

In this section, we study the dynamics of $f$ and its behavior under point modifications, starting from the model $\pi_0$ given by \refprop{exampleresolution}.

\subsection{Special families}\label{ssec:specialfamilies1}

First, we introduce some special families of {\tid} germs that will appear in the birational models.
We describe here a few notations that we will use all long the rest of the paper.
Any family $f$ of germs will be introduced by giving a name and a code. For example \emph{simple corners} [$\ZeroC$]. The code will be used in all the diagrams below.
Any family is described in some special coordinates, and there will be some formal power series $P,Q,R$, belonging to suitable ideals (that will be explicited according to cases).
We will always develop, without further mention, $P,Q,R$ in formal power series, as:
$$
P=\sum_{i,j,k} a_{ijk} x^i y^j z^k, \quad
Q=\sum_{i,j,k} b_{ijk} x^i y^j z^k, \quad
R=\sum_{i,j,k} c_{ijk} x^i y^j z^k.
$$
Unless otherwise specified, we will also replace $a_{100}$ with $a_x$, $a_{010}$ with $a_y$ and $a_{001}$ with $a_z$, and analogously for $Q$ and $R$.
Finally, we will often replace $a_{000}$, $b_{000}$ and $c_{000}$ with $a_0$, $b_0$, $c_0$, or with $\alpha, \beta, \gamma$, according to the situation. 

Recall also that $P,Q,R$ denote also the parts of degree $4$ of higher or the maps $f$ of the form \refeqn{example} that we are studying.
In this case we will keep developing them with coefficients $P_{ijk}$, $Q_{ijk}$, $R_{ijk}$, to avoid confusion.

\subsubsection{Simple corners}

We start from \emph{simple corners}, introduced for vector fields in \cite{gomezmont-luengo:vectorfieldsnoseparatrix} and adapted to {\tid} germs in  \cite{abate-tovena:paraboliccurvesC3}.

\begin{defi}[{\cite[p. 288]{abate-tovena:paraboliccurvesC3}}]
A {\tid} germ $f:(\nC^3,0)\to(\nC^3,0)$ is a \emph{simple corner} [$\ZeroC$] if there are $a,b \in \nN^*$, $c \in \nN$, $\lambda \in \nC^*$, $\mu \in \nC \setminus (\lambda \nQ_{>0})$, and local coordinates $(x,y,z)$ so that
\begin{equation}\label{eqn:simplecorner}
f(x,y,z)=
\begin{pmatrix}
x+(x^a y^b z^c) x\big(\lambda + P\big)\\
y+(x^a y^b z^c) y\big(\mu + Q\big)\\
z+(x^a y^b z^c) R
\end{pmatrix},
\end{equation}
with $P,Q,R \in \mf{m}$, and $z|R$ if $c > 0$.
\end{defi}

\begin{rmk}
We will discuss singular and exceptional directions with respect to the divisor $D=\{x^ay^bz^c=0\}$, whose support is the union or two or three coordinates planes, depending on the vanishing of $c$.

The saturated infinitesimal generator $\hat{\chi}$ of $f$ has the following properties:
\begin{enumerate}[label=(\alph*)]
\item $\hat{\chi}$ is tangent to $D$;
\item $\hat{\chi}$ is a canonical singularity.
\end{enumerate}
These properties completely characterize simple corners when $c \geq 1$.
When $c=0$, we need an additional property of the foliation on the normal bundle of the curve $\{x=y=0\}$ obtained intersecting the two irreducible components of $D$.
This property can be stated in terms of the induced foliation on $D$: the foliation induced by the saturation of $\hat{\chi}$ restricted to either $\{x=0\}$ or $\{y=0\}$ is log-canonical.
\end{rmk}

\subsubsection{{\Degspike}s}

\begin{defi}
A {\tid} germ $f:(\nC^3,0)\to(\nC^3,0)$ is a \emph{\degspike} [$\DS$] (of \emph{Siegel} type) if there are $c \in \nN^*$, $\alpha, \beta \in \nC^*$, and local coordinates $(x,y,z)$ so that
\begin{equation}\label{eqn:degspike}
f(x,y,z)=
\begin{pmatrix}
x+z^c\big(\lambda x + P\big)\\
y+z^c\big(\mu y + Q\big)\\
z+z^{c+1} R
\end{pmatrix}\text,
\end{equation}
where $\lambda \in \nC^*$, $\mu\in \lambda \nR_{<0}$, while $P,Q \in \mf{m}^2$ and $R \in \mf{m}$.
\end{defi}

\begin{rmk}\label{rmk:degspikeoften}
Notice that any germ of the form
\begin{equation}\label{eqn:degspikeoften}
\begin{pmatrix}
x+z^c\big(a_y y + a_z z + P\big)\\
y+z^c\big(b_x x + b_x z + Q\big)\\
z+z^{c+1} R
\end{pmatrix},
\end{equation}
with $a_y b_x \neq 0$, $a_z, b_z \in \nC$, $P,Q \in \mf{m}^2$, and $R \in \mf{m}$ is a degenerate spike, associated to $\lambda = -\mu = \sqrt{a_y b_x}$.
\end{rmk}
\begin{rmk}\label{rmk:degspikeinfgen}
{\Degspike}s will appear at points belonging to a unique irreducible component of the exceptional divisor $D=\{z=0\}$ on blown-up models.

In terms of the saturated infinitesimal generator $\hat{\chi}$ of $f$, degenerate spikes are characterized by the following properties:
\begin{enumerate}[label=(\alph*)]
\item $\hat{\chi}$ is tangent to $D$;
\item the induced foliation on $D$ has a Siegel singularity;
\item the linear part of $\hat{\chi}$ has exactly one vanishing eigenvalue.
\end{enumerate}
	
In particular, $\hat{\chi}|_D$ has no invariant curves passing through $p$ but for the two complex separatrices, given by $\{x=0\}$ and $\{y=0\}$ when $f$ is given by \refeqn{degspike}.
\end{rmk}

Clearly the Siegel type refers to condition (b) above. In general we could ask only for the non-resonance condition $\mu/\lambda \not \in \nQ_{\geq 0}$, i.e., the induced foliation on $D$ is canonical (and with invertible linear part).
In this paper, without further mention, all {\degspike}s are of Siegel type.

\subsubsection{{\Spincorner}s}

\begin{defi}\label{def:spincorner}
A {\tid} germ $f:(\nC^3,0)\to(\nC^3,0)$ is a \emph{\spincorner} [$\SpinC$] if there are local coordinates $(x,y,z)$ such that $f$ can be written as
\begin{equation}\label{eqn:spincorner}
f(x,y,z)=\begin{pmatrix}
x+y^bz^c(x+P)\\
y+y^{b+1}z^cQ\\
z+y^bz^{c+1}R
\end{pmatrix}\end{equation}
where $b,\ c\in\N^*$, $Q,\ R\in\mathfrak{m}$ and $P\in\mathfrak{m}^2$.
\end{defi}

\begin{rmk}\label{rmk:spincornergen}
A germ which can be written, in local coordinates $(x,y,z)$, as
\begin{equation}\label{eqn:spincornergen}
f(x,y,z)=\begin{pmatrix}
x+y^bz^c(a_x x+a_y y+a_z z+P)\\
y+y^{b+1}z^cQ\\
z+y^bz^{c+1}R
\end{pmatrix}\end{equation}
with $a_x\in\C^*$, $a_y,\ a_z\in\C$, and $b,\ c,\ P,\ Q,\ R$ as above, is in fact a {\spincorner}.
Indeed, one can assume $a_x = 1$ by a linear change of coordinates $(x,y,z) \mapsto (x, \mu y, \nu z)$ with $\mu, \nu$ satisfying $\mu^b \nu^c = \alpha$.
Moreover, in new coordinates $u=a_x x+a_y y+ a_z z$, $y$ and $z$, we get
\begin{equation*}
f(u,y,z)=\begin{pmatrix}
u +y^bz^c(a_x u+\wt{P}\circ \phi^{-1})\\
y+y^{b+1}z^cQ\circ \phi^{-1}\\
z+y^bz^{c+1}R\circ \phi^{-1}
\end{pmatrix}\;,
\end{equation*}
with $\wt{P}=a_x P+a_y y Q + a_z zR$ and $\phi^{-1}(u,y,z)=\big(a_{x}^{-1}(u-a_y y -a_z z),y,z\big)$.
\end{rmk}

\begin{rmk}\label{rmk:spincornerinfgen}
{\Spincorner}s will appear in the intersection of two irreducible components $D_1$ and $D_2$ of the exceptional divisor $D$, given in local coordinates by $\{yz=0\}$.

In terms of the reduced infinitesimal generator $\hat{\chi}$ of $f$, degenerate spikes are characterized by the following properties:
\begin{enumerate}[label=(\alph*)]
\item $\hat{\chi}$ is tangent to $D$;
\item the linear part of $\hat{\chi}$ has rank $1$, with the eigenspace of non-zero eigenvalue tangent to $D_1 \cap D_2$. 
\end{enumerate}

In fact, if $D=\{yz=0\}$ then we have $\phi \circ (f-\id) = y^{b_\phi}z^{c_\phi}A_\phi$, with $b_\phi, c_\phi \in \nN$ and $A_\phi$ a holomorphic germ that is not a multiple of $y$ or $z$, with $\phi\in \{x,y,z\}$.
The tangency condition on $\{y=0\}$ says that $b_y > b_x \wedge b_z$, while the one on $\{z=0\}$ gives $c_z > c_x \wedge x_y$.
The existence of an eigenvalue tangent to $D_1 \cap D_2$ says that $x \circ (f-\id) =y^{b_x}z^{c_x}(\alpha x + \beta y + \gamma z + P)$ with $\alpha \neq 0$ and $P \in \mf{m}^2$, and $b_x \leq b_z$ and $c_x \leq c_y$.
By setting $b=b_x$, $c=c_x$, $Q=y^{b_y-b-1}z^{c_y-c}A_y$, $y^{b_z-b}z^{c_z-c-1}A_z$, and checking the linear part of $\hat{\chi}$ in extreme cases for the parameters (i.e., if $b_y=b+1$ and $c_y=c$, or $b_z=b$ and $c_z=c+1$), we get a germ of the form \refeqn{spincornergen}.
\end{rmk}

\subsubsection{{\Halfcorner}s}

\begin{defi}\label{def:halfcorner}
A {\tid} germ $f:(\nC^3,0)\to(\nC^3,0)$ is a \emph{\halfcorner} [$\HalfC$] if there are local coordinates $(x,y,z)$ such that $f$ can be written as
\begin{equation}\label{eqn:halfcorner}
f(x,y,z)=\begin{pmatrix}
x+z^c(x+ P)\\
y+z^{c+1}(\beta + Q)\\
z+z^{c+2}R
\end{pmatrix}\end{equation}
where $c\in\N^*$, $\beta \in \nC$, $P\in\mf{m}^2$, $Q \in\mf{m}$ and $R=\gamma+\mf{m}$ with $\gamma \in \nC$.
\end{defi}

\begin{rmk}\label{rmk:halfcornergen}
As for the case of {\spincorner}s, one can show that any germ of the form 
\begin{equation}\label{eqn:halfcornergen}
f(x,y,z)=\begin{pmatrix}
x+z^c(a_x x + a_y y + a_z z + P)\\
y+z^{c+1}(\beta + Q)\\
z+z^{c+2}R
\end{pmatrix}\end{equation}
with $a_x \neq 0$, $a_y, a_z \in \nC$ and all other entries as above is indeed a {\halfcorner}.
In fact, we may assume $a_x = 1$ by a linear change of coordinates $(x,y,z)\mapsto(x, y, \nu z)$ with $\nu^c = a_x$.
Then, one can assume $a_y=0$ by performing the change of coordinates $u=x+a_y y$ (which changes the value of $a_z$ to $a'_z=a_z+\beta a_y$), and finally we can set $u'=x+a'_z z$ and get a germ of the form \refeqn{halfcorner}.
	
Notice also that when $\beta \neq 0$, we may assume it equals $1$, by performing the change of coordinates $(x,y,z) \mapsto (x, \beta y, z)$.
	
The value of $\beta$ (its vanishing) will be important in the sequel. We will say that a {\halfcorner} is \emph{\hcsimple} if $\beta \neq 0$, \emph{\hcnonsimple} otherwise.

In fact, we can independently normalize (by conjugating by linear diagonal maps) both the second and third coordinates, for example by assuming that $\beta\in\{0,1\}$ and $\gamma:=R(0,0,0) \in\{0,1\}$.
\end{rmk}
\begin{rmk}\label{rmk:halfcornernonsimpleresonance}
Once in form \refeqn{halfcorner} we still have some freedom up to linear change of coordinates.
Assume $\beta=0$. In this case we can conjugate by a map of the form $(x,y,z) \mapsto (\lambda x,\mu y,\nu z)$ with $\nu^c=1$.
In this case we get $\wt{b}_y = \nu b_y$, $\wt{\gamma}= \nu \gamma$.
In particular their ratio is well defined up to homotheties (and it is in fact an invariant of conjugacy for {\halfcorner}s in form \refeqn{halfcorner}).
\end{rmk}

\begin{rmk}\label{rmk:halfcornerinfgen}
{\Halfcorner}s will appear in points contained in a unique irreducible component $D$ of the exceptional divisor, which we will assume having local equation $\{z=0\}$.

One can characterize {\halfcorner}s in terms of their infinitesimal generator also in this case, but the description is more intricated. We just remark that again the saturated infinitesimal generator is tangent to $D$.
Moreover, its linear part has a non-zero eigenvalue (whose eigenspace is tangent to $D$), and:
\begin{itemize}
\item  either a Jordan block associated to the zero eigenvalue in the {\hcsimple} case, with the kernel being tangent to $D$; or
\item a kernel of dimension $2$ in the {\hcnonsimple} case.
\end{itemize}
\end{rmk}

\subsection{From the resolution to special families}

We show here how, possibly up to further blow-up, the singularities appearing in the model $\pi_0:X_{\pi_0} \to (\nC^3,0)$ given by \refprop{exampleresolution}, belong to one of the families described in \refssec{specialfamilies1}

In fact, from the study done in \refssec{resolutionexample}, the lift $f_{\pi_0}:X_{\pi_0} \to X_{\pi_0}$ satisfies the following properties.
\begin{itemize}
\item At the singularity $p_1$, $f_{\pi_0}$ takes the form \refeqn{taylorIz}, which is a {\degspike} of the form \refeqn{degspikeoften} with respect to the coordinates $(x,y,z)$, with parameters $c=2$, $\alpha = \beta = -1$.
\item At the singularity $p_2$, $f_{\pi_0}$ takes the form \refeqn{taylorIzp2}, which is a {\degspike} of the form \refeqn{degspikeoften} with respect to the coordinates $(x,y,z)$, with parameters $c=2$, $\alpha = 1$, $\beta = -1$.
\item At the singularity $q_1$, $f_{\pi_0}$ takes the form \refeqn{taylorIyIIxIIIxIVx}, which is a {\spincorner} of the form \refeqn{spincornergen} with respect to coordinates $(z,y,x)$, with parameters $\alpha = 1$, $\beta = R_{040}$, $\gamma=0$, $b=2$, $c=7$.
\item At the singularity $q_2$, $f_{\pi_0}$ is a simple corner of the form \refeqn{simplecorner} with respect to coordinates $(x,y,w)$ with $w=z-\frac{1}{2}$ (notations of  \refeqn{taylorIyIIxIIIxIVx}), with parameters $a=7$, $b=2$, $c=0$, $\lambda = \frac{1}{2}$ and $\mu=-\frac{3}{2}$.
\item At the singularity $q_3$, $f_{\pi_0}$ takes the form \refeqn{taylorIyIIxIIIzIVz}, which is a simple corner of the form \refeqn{simplecorner} with respect to coordinates $(z,y,x)$, with parameters $a=7$, $b=2$, $c=2$, $\lambda = -1$ and $\mu=1$.
\item At the singularity $q_4$, $f_{\pi_0}$ is a simple corner of the form \refeqn{simplecorner} with respect to coordinates $(x,z,v)$ with $v=y-\frac{1}{2R_{040}}$ (notations of  \refeqn{taylorIyIIxIIIzIVz}), with parameters $a=2$, $b=7$, $c=0$, $\lambda = \frac{3}{2}$ and $\mu=-\frac{1}{2}$ (up to common factors, see \refeqn{linearpartq4}).
\item At the singularity $q_5$, $f_{\pi_0}$ takes the form \refeqn{taylorIyIIxIIIzIVy}, which is a simple corner of the form \refeqn{simplecorner} with respect to coordinates $(z,y,x)$, with parameters $a=4$, $b=7$, $c=2$, $\lambda = 1$ and $\mu=-1$ (up to a common factor $R_{040}$).
\end{itemize}
The only singularities not falling in one of the families described in \refssec{specialfamilies1} are $p_3$ and $p_4$.
By symmetry (see \refrmk{symmetry}), we will only deal with $p_3$, the case of $p_4$ being completely analogous.

On suitable coordinates $(u,v,z)$ centered at $p_3$, the germ $f_{\pi_0}$ takes the form:

\begin{equation*}\tag{\ref{eqn:taylorIzp3}}
f_{\pi_0}(u,v,z) =
\begin{pmatrix}
u + z^2 v(-2u+v-u^2) + z^3 (P^{(4)}-(1+u)R^{(4)})(1+u,1+v,1) + \langle z^4\rangle\\[3mm]
v + z^2 (1+u)(2u-v+u^2-v^2) + z^3 (Q^{(4)}-(1+v)R^{(4)})(1+u,1+v,1) + \langle z^4\rangle\\[3mm]
z + z^3(1+u)v + z^4 R^{(4)}(1+u,1+v,1) + \langle z^5\rangle
\end{pmatrix}
\text.
\end{equation*}
In this case, the homogeneous part of smallest degree of $z^{-2}(f_{\pi_0} - \id)$ is linear, with associated matrix
$$
\begin{pmatrix}
0 & 0 & \alpha\\
2 & -1 & \beta\\
0 & 0 & 0
\end{pmatrix}
\text,
$$
where $\alpha = (P^{(4)} - R^{(4)})(1,1,1)$ and $\beta= (Q^{(4)} - R^{(4)})(1,1,1)$.

The computation of singular directions depend on weather $\alpha$ vanishes or not.
In both cases, $v_{3,1}=[0:1:0]$ is a singular direction (associated to the eigenvalue $1$), as is $v_{3,2}=[1:2:0]$ (with multiplier $0$).
If $\alpha \neq 0$, the generalized eigenspace associated to the eigenvalue $0$ is associated to a Jordan block of size $2$. It follows that $v_{3,1}$ and $v_{3,2}$ are the only singular directions (which are both exceptional).
If $\alpha=0$, the kernel has rank $2$, which gives a line of degenerate directions, generated by $[1:2:0]$ and $[0: \beta: 1]$.

For simplicity, we will assume that $\alpha = P^{(4)}(1,1,1)-R^{(4)}(1,1,1) \neq 0$.
\medskip

\textbf{Blow-up of $p_3$.}

We consider $\wt{\pi}_1:X_{\wt{\pi}_1} \to X_{\pi_0}$ the blow-up of $p_3$ in $X_{\pi_0}$.
We consider the chart in $X_{\wt{\pi}_1}$ so that $\wt{\pi}_1(x,y,z)=(xy,y,yz)$.
The lift $\wt{f}_1$ of $f_{\pi_0}$ is given by:
$$
\wt{f}_1(x,y,z) =
\begin{pmatrix}
x + y^2z^2\Big(x-2x^2 +y(3x^2-3x) + z(\alpha-\beta x) + y \langle y,z \rangle\Big)\\[3mm]
y + y^3z^2 \Big(-1+2x+y(3x^2-x-1) +\beta z +y\langle y,z\rangle\Big)\\[3mm]
z + y^2z^3\Big(1-2x+y(2+x-3x^2) -\beta z +y \langle y,z\rangle\Big)
\end{pmatrix}
\text.
$$
This is clearly a simple corner at $p_{3,1}$ (which corresponds to the origin in this chart).
It is with respect to coordinates $(z,y,x)$, with $a=b=2$ and $c=0$, $\lambda = 1$ and $\mu = -1$.

The point $p_{3,2}$ corresponds in this chart to $(\tfrac{1}{2},0,0)$. By setting $x=\tfrac{1}{2}+u$, we get
\begin{equation}\label{eqn:taylorp32}
\wt{f}_1(u,y,z) =
\begin{pmatrix}
u + y^2z^2\Big(-u -\tfrac{3}{4}y + (\alpha-\tfrac{1}{2}\beta) z + \mf{m}^2\Big)\\[3mm] 
y + y^3z^2 \Big(2u-\tfrac{3}{4}y + \beta z + \mf{m}^2\Big)\\[3mm]
z + y^2z^3\Big(-2u+\tfrac{7}{4}y - \beta z + \mf{m}^2\Big)
\end{pmatrix}
\text.
\end{equation}
This is a {\spincorner} of the form \refeqn{spincornergen} with respect to the coordinates $(u,y,z)$, with parameters $b=2$, $c=2$.

\begin{rmk}
When we change coordinates so that the linear part of the saturated vector field has $a_x x$ as first coordinate, then we get the parameters $b_y=-\frac{9}{4}$, $c_y=\frac{13}{4}$, $b_z=2\alpha$ and $c_z = -2\alpha$.
Notice that in general one needs to replace $b_y$ with $b_y-b_x\frac{a_z}{a_x}$ and similarly for $b_z$, $c_y$, $c_z$.
\end{rmk}
We proved the following:

\begin{prop}\label{prop:exampleresolutionforms}
Let $f:(\nC^3,0) \to (\nC^3),0)$ be a germ of the form \refeqn{example} with $R_{040} \neq 0$ and $P^{(4)}(\pm 1,1,1) \neq \pm R^{(4)}(\pm 1, 1, 1)$.
Let $\wt{\pi}_0:X_{\wt{\pi}_0} \to (\nC^3,0)$ be the regular modification obtained as the composition $\wt{\pi}_0=\pi_0 \circ \wt{\pi}_1 \circ \wt{\pi}_2$, where $\wt{\pi}_1$ is the blow-up of $p_3$ and $\wt{\pi}_2$ is the blow-up of $p_4$.
	
Then the lift $f_{\wt{\pi}_0}$ of $f$ to $X_{\wt{\pi}_0}$ has finitely many singular points, where it is either a simple corner, a {\degspike}, or a {\spincorner}.
\end{prop}

\begin{figure}[h]
\centering
\begin{minipage}[htbp]{0.75\columnwidth}
\def\svgwidth{\columnwidth}
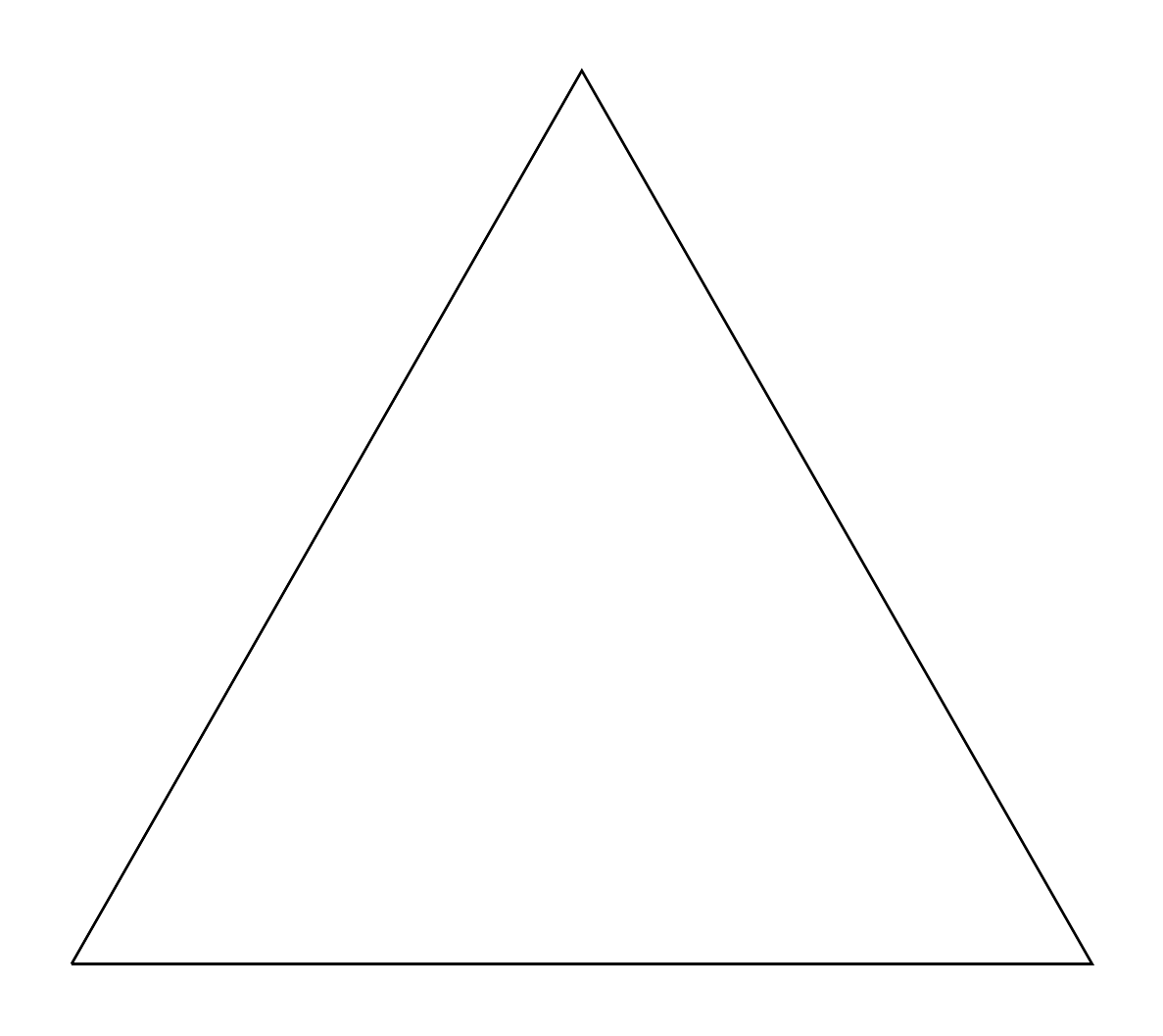
\end{minipage}
\caption{Singular points of the saturated infintesimal generator at $X_{\wt{\pi}_0}$. We have {\degspike}s at $p_1$ and $p_2$, {\spincorner}s at $p_{3,2}$, $p_{4,2}$ and $q_1$, and simple corners at the other marked points.}
\label{fig:resolution2}
\end{figure}

\subsection{Birational study}\label{ssec:resolbiratstudy}

Here we describe the behavior of the families introduced in \refssec{specialfamilies1} under point blow-up.

\subsubsection{Simple corners}

The situation for simple corners is already known, we summarize here their behavior under point blow-up.

\begin{prop}[{\cite[Proposition 4.1]{abate-tovena:paraboliccurvesC3}}]\label{prop:simplecornerblowup}
Let $f:(\nC^3,0)\to (\nC^3,0)$ be a simple corner, and denote by $\wt{f}$ the blow-up of $f$ at $0$.
Then
\begin{enumerate}[label=(\roman*),leftmargin=0pt, itemindent=40pt]
\item $0$ is never $2$-dicritical;
\item the singular directions of $f$ are always simple corners of $\wt{f}$.
\end{enumerate}
\end{prop}

We will need the behavior of simple corners with respect to any admissible blow-up, and to do so we need to be more explicit on the geometry of the singular directions of a simple corner.

\begin{prop}\label{prop:simplecornerblowupexplicit}
Let $f:(\nC^3,0)\to (\nC^3,0)$ be a simple corner of the form \refeqn{simplecorner}, write $R=\alpha x + \beta y + \gamma z + \mf{m}^2$, with $\alpha = \beta = 0$ if $c > 0$. Then we get the following singular directions:
\begin{itemize}
\item $[\lambda-\gamma:0:\alpha]$ if $\alpha$ and $\lambda -\gamma$ are not both vanishing; 
\item $[p:0:r]$ for all $[p:r] \in \nP_\nC^1$, if $\alpha=\lambda -\gamma=0$;
\item $[0:\mu-\gamma:\beta]$ if $\beta$ and $\mu -\gamma$ are not both vanishing; 
\item $[0:q:r]$ for all $[q:r] \in \nP_\nC^1$, if $\beta=\mu -\gamma=0$;
\item $[0:0:1]$.
\end{itemize}
All directions are exceptional, and simple corners.
\end{prop}
\begin{proof}
The computation of singular directions is strightforward, since it corresponds on determining the eigenspaces of the linear map represented
$$
\begin{pmatrix}
\lambda & 0 & 0\\
0 & \mu & 0 \\
\alpha & \beta & \gamma
\end{pmatrix}\text.
$$
Since we will need this computation later, we verify that the singularities arising are again simple corners (property we already know from \refprop{simplecornerblowup}), at least for the case of non-isolated singular directions.

By working on the $z$-chart, and developing in formal power series, the lift $\wt{f}$ of $f$ takes the form
\begin{equation}\label{eqn:simplecornerblowup}
\wt{f}(x,y,z)
=
\begin{pmatrix}
x\Big(1+x^ay^bz^s\big(\lambda -\gamma -\alpha x - \beta y  + \langle z \rangle\big)\Big)\\[2mm]
y\Big(1+x^ay^bz^s\big(\mu -\gamma -\alpha x - \beta y  + \langle z \rangle\big)\Big)\\[2mm]
z\Big(1+x^ay^bz^s\big(\gamma +\alpha x + \beta y  + \langle z \rangle\big)\Big)
\end{pmatrix}
\text,
\end{equation}
where $s=a+b+c$.
At the origin, corrisponding to the direction $[0:0:1]$, we get a simple corner with respect to either $(x,y,z)$ or $(y,x,z)$, depending on whether $\lambda \neq \gamma$ or $\mu \neq \gamma$ (at least one of the two holds).

If $\lambda=\gamma$ and $\alpha =0$, we get singularities at all points $(x_0,0,0)$.
By replacing $x=x_0+u$, we get simple corners of the form \refeqn{simplecorner} with respect to coordinates $(y,z,u)$.
The other cases are analogous and left to the reader.
\end{proof}

We depict the situation in the next diagram.
Exceptional directions are depicted in red, while non-exceptional (degenerate) directions will be depicted in blue (there are none for simple corners).
We also indicate the type of {\tid} germ we get at each characteristic point (in this case, all simple corners).
Finally, we indicate the geometry of singular points in case they come in a family (in this case, with a parameter $z_0 \in \nC$).
$$
\begin{xy} ;<5mm,0mm>:
\POS(0,0)
{$\overtype{\ZeroC}{0}$}
\ar +(2,4)*!L{\overtype{\ZeroC}{{\red [\lambda-\gamma:0:\alpha]}}\text{ if } (\lambda-\gamma, \alpha) \neq (0,0)}
\ar +(2,2)*!L{\overtype{\ZeroC}{{\red [1:0:z_0]}}\text{ if } \lambda-\gamma = \alpha = 0}|{\nC}
\ar +(2,0)*!L{\overtype{\ZeroC}{{\red [0:\mu-\gamma:\beta]}}\text{ if } (\mu-\gamma, \beta) \neq (0,0)}
\ar +(2,-2)*!L{\overtype{\ZeroC}{{\red [0:1:z_0]}}\text{ if } \mu-\gamma = \beta = 0}|{\nC}
\ar +(2,-4)*!L{\overtype{\ZeroC}{{\red [0:0:1]}}}
\end{xy}
$$

\begin{rmk}
Notice that the \emph{resonances} given by $\lambda-\gamma=\alpha=0$ and $\mu-\gamma = \beta = 0$ cannot happen both at the same time, since we would have $\lambda = \mu$, which is not allowed.

Notice also that \refeqn{simplecornerblowup}, with respect to coordinates $(z,y,u)$ with $x=x_0+u$, takes the form of \refeqn{simplecorner}, with $R\in \langle x,y \rangle$ (here we are using the notations of \refeqn{simplecorner}).
In particular $\gamma = 0$ in this case, and all these points are not resonant.
\end{rmk}

\subsubsection{{\Degspike}s}

\begin{lem}\label{lem:degeneratedirection}
Let $f:(\nC^3,0) \to (\nC^3,0)$ be a {\degspike} of the form \eqref{eqn:degspike}.
Then $f$ has three singular directions, given by:
\begin{itemize}
\item $\vect{v}=[0:0:1]$, which is non-exceptional and degenerate;
\item $\vect{w_1}=[1:0:0]$, and $\vect{w_2}=[0:1:0]$, which are exceptional, with multipliers $\lambda$ and $\mu$ (seen as singular directions).
\end{itemize}
\end{lem}
\begin{proof}
The proof is a direct computation, left to the reader.
\end{proof}
\begin{rmk}
For maps of the form \refeqn{degspikeoften}, we have $\vect{v}=\left[-\frac{b_z}{b_x}:-\frac{a_z}{a_y}:1\right]$, and $\vect{w_j}=\left[\sqrt{a_y}:(-1)^{j} \sqrt{b_x}:0\right]$ for $j=1, 2$ (for some determinations of the square roots of $a_y$ and $b_x$).
\end{rmk}

\begin{prop}\label{prop:degspikeblowuppoint}
Let $f:(\nC^3,0) \to (\nC^3,0)$ be a {\degspike} of the form \eqref{eqn:degspike}.
Let $\pi:X \to (\nC^3,0)$ be the blow-up at the origin in $\nC^3$. For the lift $\wt{f}$ of $f$ to $X$, we have that:
\begin{itemize}
\item $\vect{v}=[0:0:1]$ is a {\degspike};
\item $\vect{w_1\!}=[1:0:0]$ and $\vect{w_2\!}=[0:1:0]$ are simple corners.
\end{itemize}
\end{prop}
The following diagram portrays the situation for {\degspike}s.
We recall that exceptional directions are depicted in red and non-exceptional degenerate directions are depicted in blue.
To help the reader, we also indicate with a subscript the chart in which we make the computations, i.e., the equation of the exceptional divisor obtained with the last blow-up.
$$
\begin{xy} ;<5mm,0mm>:
\POS(0,0)
{$\overtype{\DS}{0}$}
\ar +(2,1)*!L{\overtype{\DS}{{\blue \vect{v_z\!}}}}
\ar +(2,-1)*!L{\overtype{\ZeroC}{{\red\ \vect{w_x\!}}}}|{2}
\end{xy}
$$
\begin{proof}
\mbox{}
	
\begin{trivlist}
\itemchdir{[0:0:1]}

We make computations in the $z$-chart, so that $\pi(x,y,z)=(xz,yz,z)$.
For the lift $\wt{f}$ of $f$, we obtain
\begin{equation}\label{eqn:degspikeblow0z}
\wt{f}(x,y,z)
=
\begin{pmatrix}
\displaystyle \frac{x+z^c \big(\lambda x + z^{-1}P\circ \pi\big)}{1+z^{c} R\circ \pi}\\[6mm]
\displaystyle \frac{y+z^c \big(\mu y + z^{-1}Q \circ \pi\big)}{1+z^{c} R \circ \pi}\\[6mm]
\displaystyle z\big(1+z^{c} R \circ \pi\big)
\end{pmatrix}
\text.
\end{equation}
By developing in formal power series, we get
$$
\wt{f}(x,y,z)
=
\begin{pmatrix}
x+z^c \big(\lambda x + a_{002} z + \mf{m}\big)\\[2mm]
y+z^c \big(\mu y + b_{002} z + \mf{m}\big)\\[2mm]
z+z^{c+1} \mf{m}
\end{pmatrix}
\text,
$$
which is again a {\degspike}.

\itemchdir{[1:0:0], [0:1:0]}

We study $[1:0:0]$, the case $[0:1:0]$ being obtained by exchanging the role of $x$ and $y$.
We make computations in the $x$-chart, so that $\pi(x,y,z)=(xz,yz,z)$, and get
$$
\wt{f}(x,y,z)
=
\begin{pmatrix}
\displaystyle x\Big(1+x^cz^c\big(\lambda + x^{-1}P \circ \pi\big)\Big)\\[4mm]
\displaystyle \frac{y+x^cz^c \big(\mu y + x^{-1}Q \circ \pi\big)}{1+x^cz^c\big(\lambda + x^{-1}P \circ \pi\big)}\\[6mm]
\displaystyle z\frac{1+x^{c}z^{c} R \circ \pi}{1+x^cz^c\big(\lambda + x^{-1}P\circ \pi\big)}
\end{pmatrix}
\text.
$$
By developing in formal power series, we get
$$
\wt{f}(x,y,z)
=
\begin{pmatrix}
x+x^{c+1}z^c\big(\lambda + \langle x \rangle\big)\Big)\\[2mm]
y+x^cz^c \big((\mu-\lambda) y + \langle x \rangle \big)\\[2mm]
z+x^{c}z^{c+1} \big(-\lambda + \langle x \rangle \big)
\end{pmatrix}
\text,
$$
which is a simple corner with respect to coordinates $(x,z,y)$.
\end{trivlist}
\end{proof}

\subsubsection{{\Spincorner}s}

\begin{prop}\label{prop:spincornerblowup}
Let $f:(\nC^3,0)\to(\nC^3,0)$ be a {\spincorner} of the form \refeqn{spincorner}.
The singular directions of $f$ are $[1:0:0]$ (non-degenerate) and the points of the line $[0:p:q]$, with $[p:q]\in \nP_\nC^1$ (all degenerate).

Let $\pi:X \to (\nC^3,0)$ be the blow-up at the origin. For the lift $\wt{f}$ of $f$ to $X$, we have that:
\begin{enumerate}\setcounter{enumi}{1}
\item $[1:0:0]$ is a simple corner;
\item $[0:1:0]$ and $[1:0:0]$ are {\spincorner}s;
\item $[0:p:q]$ are {\halfcorner}s for any $p,q$ with $pq\neq 0$.
\end{enumerate}
\end{prop}
We sum up the situation for {\spincorner}s.
$$
\begin{xy} ;<5mm,0mm>:
\POS(0,0)
{$\overtype{\SpinC}{0}$}
\ar +(2,3)*!L{\overtype{\ZeroC}{{\red [1:0:0]_x}}}
\ar +(2,1)*!L{\overtype{\SpinC}{{\red [0:1:0]_y}}}
\ar +(2,-1)*!L{\overtype{\SpinC}{{\red [0:0:1]_z}}}
\ar +(2,-3)*!L{\overtype{\HalfC}{{\blue [0:y_0:1]_z}}}|{\nC^*}
\end{xy}
$$

\begin{proof}
The list of singular directions is easily obtained by the fact that the homogeneous part of smallest degree of $f-\id$ is given by $y^bz^c\vV{x,0,0}$.

\begin{trivlist}
\itemchdir{[1:0:0]}
We make computations in the $x$-chart, and we obtain
$$
\wt{f}(x,y,z)=
\begin{pmatrix}
x\Big(1+x^sy^bz^c\big(1+x^{-1}P \circ \pi\big)\Big)\\[4mm]
\displaystyle y \frac{1+x^{s}y^{b}z^{c}Q \circ \pi}{1+x^{s}y^bz^c\big(1+x^{-1}P\circ \pi\big)}\\[6mm]
\displaystyle z\frac{1+x^{s}y^{b}z^{c}R\circ \pi}{1+x^{s}y^bz^c\big(1+x^{-1}P\circ \pi\big)}
\end{pmatrix}\text,
$$
where $s=b+c$.
This gives a simple corner.
	
\itemchdir{[0:0:1], [0:1:0]}
We study $[0:0:1]$, the case $[0:1:0]$ being obtained by exchanging the role of $y$ and $z$.
Making computations in the $z$-chart, we get
$$
\wt{f}(x,y,z)=\begin{pmatrix}
\displaystyle \frac{x+y^bz^{s}(x+z^{-1}P\circ \pi)}{1+y^bz^{s} R \circ \pi}\\[6mm]
\displaystyle y\frac{1+y^{b}z^{s}Q \circ \pi}{1+y^bz^{s}R \circ \pi}\\[6mm]
z\big(1+y^bz^{s}R\circ \pi\big)
\end{pmatrix}\text.
$$
We develop in formal power series, obtaining
\begin{equation}\label{eqn:spincornerz}
\wt{f}(x,y,z)=\begin{pmatrix}
x+y^bz^{s}\Big(x+z\big(P^{(2)}-xR^{(1)}\big)(x,y,1) + \langle z^2\rangle\Big)\\
y+y^{b+1}z^{s+1}\Big(\big(Q^{(1)}-R^{(1)}\big)(x,y,1) + z\big(Q^{(2)}-R^{(2)}\big)(x,y,1)+ \langle z^2 \rangle\Big) \\
z+y^bz^{s+2}\Big(R^{(1)}(x,y,1) +zR^{(2)}(x,y,1) + \langle z^2 \rangle\Big)
\end{pmatrix},
\end{equation}
where for any $k \in \nN^*$, $P^{(k)}$ denotes the homogeneous part of degree $k$ of $P$ (and analogously for $Q$ and $R$).

In particular, $\wt{f}$ is a {\spincorner} of the form \refeqn{spincornergen} with respect to coordinates $(x,y,z)$.

\itemchdir{[0:y_0:1]}
It remains to study the germ of $\wt{f}$ at points of the form $[0:y_0:1]$, with $y_0 \in \nC^*$.
We write $Q=b_x x +b_y y +b_z z + \mf{m}^2$, and $R=c_x x +c_y y +c_z z + \mf{m}^2$. We center coordinates at $[0:y_0:1]$ by setting $y=y_0+v$, and from \refeqn{spincornerz} we get:
\begin{equation}\label{eqn:spincornerztrasl}
\hspace{-2mm}
\begin{pmatrix}
x+y_0^bz^{s}\Big(x+z P^{(2)}(0,y_0,1) + \langle x,z\rangle \mf{m}\Big)\\[2mm]
v+y_0^{b+1}z^{s+1}\Big(\wt{\beta} + x (b_x-c_x) + v \big((b+1)y_0^{-1}\wt{\beta}+(b_y-c_y)\big) + z \big(Q-R\big)^{(2)}(0,y_0,1) + \mf{m}^2\Big) \\[2mm]
z+y_0^bz^{s+2}\Big(\wt{\gamma} + x c_x + v \big(by_0^{-1}\wt{\gamma}+c_y\big) + z R^{(2)}(0,y_0,1) + \mf{m}^2\Big)
\end{pmatrix},
\end{equation}
where $\wt{\beta}=\wt{\beta}(y_0)=b_z-c_z+y_0(b_y-c_y)$ and $\wt{\gamma}=\wt{\gamma}(y_0)=c_z+y_0c_y$.
This is a {\halfcorner}, {\hcnonsimple} or {\hcsimple} depending on the vanishing of $\wt{\beta}(y_0)$.
\end{trivlist}
\end{proof}

\begin{rmk}\label{rmk:beta}
In what follows, we will be interested in the existence of {\hcnonsimple} {\halfcorner}s, hence in the vanishing of the coefficient $\wt{\beta}(y_0)$.
Three situations can occur:
\begin{itemize}
\item if $b_y=c_y$ and $b_z=c_z$, then all {\halfcorner}s are {\hcnonsimple};
\item if exactly one of the two equalities above hold, then all {\halfcorner}s are {\hcsimple};
\item if none of the two equalities above hold, then there exists a unique $y_0$ at which $\wt{f}$ is {\hcnonsimple}, and all the others produce {\hcsimple} {\halfcorner}s. 
\end{itemize}

Notice that the value of $\wt{\beta}(y_0)$ has the same formula for {\spincorner}s of the form \refeqn{spincornergen} with $a_y=a_z=0$ (i.e., where we allow $a_x$ to be different from $1$).
\end{rmk}

\subsubsection{{\Halfcorner}s}

\begin{prop}\label{prop:halfcornerblowup}
Let $f:(\nC^3,0)\to(\nC^3,0)$ be a {\halfcorner} of the form \refeqn{halfcorner}.
Its singular directions are given by:
\begin{itemize}
\item $[1:0:0]$, exceptional; 
\item $[0:1:0]$, exceptional; 
\item $[0:y_0:1]$ for all $y_0 \in \nC$, non-exceptional degenerate (when $f$ is {\hcnonsimple}).
\end{itemize}
Let $\wt{f}$ be the lift of $f$ to the blow-up of the origin in $\nC^3$. Then
\begin{itemize}
\item $\wt{f}$ is a simple corner at $[1:0:0]$,
\item $\wt{f}$ is a {\spincorner} at $[0:1:0]$,
\item if $f$ is \hcnonsimple, then $\wt{f}$ is a {\halfcorner} at $[0:y_0:1]$ for all $y_0\in \nC$.
\end{itemize}
\end{prop}
Here is a depiction of the situation for {\halfcorner}s.
$$
\begin{xy} ;<5mm,0mm>:
\POS(0,0)
{$\overtype{\HalfC}{0}$}
\ar +(2,2)*!L{\overtype{\ZeroC}{{\red [1:0:0]_x}}}
\ar +(2,0)*!L{\overtype{\SpinC}{{\red [0:1:0]_y}}}
\ar +(2,-2)*!L{\overtype{\HalfC}{{\blue [0:y_0:1]_z}} \text{ if } \beta=0}|{\nC}
\end{xy}
$$

\begin{proof}
The list of singular directions is easily obtained by the fact that the homogeneous part of smallest degree of $f-\id$ is given by $z^c\vV{x,\beta z,0}$.

\begin{trivlist}
\itemchdir{[1:0:0]}
We make computations in the $x$-chart, and get
$$
\wt{f}(x,y,z)=
\begin{pmatrix}
x\big(1+x^cz^c(1+x^{-1}P \circ \pi)\big)\\[4mm]
\displaystyle \frac{y+x^cz^{c+1}(\beta + Q\circ \pi)}{1+x^cz^c(1+x^{-1}P\circ \pi)}\\[6mm]
\displaystyle z\frac{1+x^{c+1}z^{c+1}R\circ \pi}{1+x^cz^c(1+x^{-1}P\circ \pi)}
\end{pmatrix}\text.
$$
Since $x^2\mid P\circ \pi$, by direct computation we get
$$
\wt{f}(x,y,z)=\begin{pmatrix}
x+x^{c+1}z^c(1+\mf{m}))\\
y+x^cz^c\mf{m}\\
z+x^cz^{c+1}(-1+\mf{m})
\end{pmatrix},
$$
which is a simple corner with respect to coordinates $(x,z,y)$.

\itemchdir{[0:1:0]}
In the $y$-chart, we get
$$
\wt{f}(x,y,z)=\begin{pmatrix}
\displaystyle \frac{x+y^cz^c(x+y^{-1}P\circ \pi)}{1+y^{c}z^{c+1}(\beta+Q \circ \pi)}\\[6mm]
y\big(1+y^{c}z^{c+1}(\beta+Q\circ \pi)\big)\\[4mm]
\displaystyle z \frac{1+y^{c+1}z^{c+1}R\circ \pi}{1+y^{c}z^{c+1}(\beta+Q \circ \pi)}
\end{pmatrix}\text.
$$
By developing in formal power series, we get
\begin{equation}\label{eqn:halfcornerytaylor}
\wt{f}(x,y,z)=
\begin{pmatrix}
x+y^cz^c\big(x+a_{020}y + \mf{m}^2\big)\\[2mm]
y+y^{c+1}z^c \big(z\beta + \langle yz\rangle \big)\\[2mm]
z+y^cz^{c+1} \big(-\beta z + \langle yz \rangle\big)
\end{pmatrix}\text,
\end{equation}
and $\wt{f}$ is a {\spincorner} at $[0:1:0]$.

\itemchdir{[0:y_0:1]}

Finally, suppose $\beta=0$. By doing computation in the $z$-chart we get
$$
\wt{f}(x,y,z)=
\begin{pmatrix}
\displaystyle \frac{x+z^c(x+z^{-1}P\circ \pi)}{1+z^{c+1}R\circ \pi}\\[6mm]
\displaystyle \frac{y+z^c Q\circ \pi}{1+z^{c+1}R\circ \pi}\\[6mm]
z\big(1+z^{c+1}R\circ \pi\big)
\end{pmatrix}\text.
$$
Write $Q=b_x x + b_y y + b_z z + \mf{m}^2$ and $R=\gamma + c_x x + c_y y + c_z z + \mf{m}^2$, and expand $\wt{f}$ in formal power series:
$$
\wt{f}(x,y,z)=
\begin{pmatrix}
x+z^c\big(x+ zP^{(2)}(0,y,1)+ z \langle x,z\rangle\big)\\[2mm]
y+z^{c+1}\Big(b_z + (b_y-\gamma) y + b_x x + z \big(Q^{(2)}(0,y,1)-c_z y-c_y y^2\big) + z\langle x,z\rangle\Big)\\[2mm]
z+z^{c+2}\Big(\gamma + z(c_z+c_y y) + z\langle x,z\rangle \Big)
\end{pmatrix}\text.
$$
We develop at the direction $[0:y_0:1]$ for some $y_0 \in \nC$, by setting $y=y_0+v$, and we get
\begin{equation}\label{eqn:halfcornerR3fam}
\begin{pmatrix}
x+z^c\big(x+ zP^{(2)}(0,y_0,1) + z \mf{m}\big)\\[2mm]
v+z^{c+1}\Big(b_z + (b_y-\gamma) y_0 + b_x x + (b_y-\gamma) v + z \big(Q^{(2)}(0,y_0,1)-c_z y_0-c_y y_0^2\big) + z\mf{m}\Big)\\[2mm]
z+z^{c+2}\Big(\gamma + z(c_z+c_y y_0) + z\mf{m} \Big)
\end{pmatrix}\text.
\end{equation}
By \refrmk{halfcornergen}, $\wt{f}$ is again a {\halfcorner}, {\hcnonsimple} or {\hcsimple} according to the vanishing of $\wt{\beta}(y_0)=b_z + (b_y-\gamma) y_0$.
\end{trivlist}	
\end{proof}

\section{Blow-up of singular curves}\label{sec:patterns}

We study here the behavior of the families introduced in the previous two sections when blowing-up curves contained in the singular locus $S_\pi$ of $f_\pi$ the lift of $f$ at a model $X_\pi$ (i.e., the singular locus of its saturated infinitesimal generator).

\subsection{Patterns}

We start by describing the structure of $S_\pi$ when $\pi$ is a point modification (adapted to $f$) dominating $X_{\pi_0}$.
To do so we will use the following terminology.
\begin{defi}
A (rational) \emph{pattern} is a triple $(X,C,f)$, where $X$ is a smooth $3$-fold, $C$ is a smooth compact rational curve inside $X$, and $f:(X,C) \to (X,C)$ is a holomorphic germ at $C$, fixing $C$ pointwise, and defining {\tid} germs at $p$ for any $p \in C$.
Moreover, if $\hat{\chi}$ is the saturated infinitesimal generator of $f$, we impose that its singular set $S$ contains $C$.
The curve $C$ is called the \emph{core} of the pattern. 

If $\mc{G}$ is a family of {\tid} germs, we say that a pattern $(X,C,f)$ is of \emph{type $\mc{G}$} (or a \emph{$\mc{G}$-pattern}) if the germ of $f$ at $p$ belongs to $\mc{G}$ for all but finitely many $p \in C$.
Any such point $p$ is called a \emph{generic point} of the pattern, while any point at which the germ of $f$ does not belong to $\mc{G}$ is called a \emph{special point}.
The \emph{generic locus} of the pattern is the set of generic points of $C$, while the \emph{special locus} is its complement.
\end{defi}

If we need to express the fact that special points of a $\mc{G}$-pattern belong to some classes $\mc{S}$, we will talk about $\mc{S}$-$\mc{G}$-patterns.
A $\mc{G}$-$\mc{G}$-pattern is a $\mc{G}$-pattern without special points.

\begin{rmk}
One should think of patterns as germs of dynamical systems on germs of $3$-dimensional manifolds around the core. These could be also described in more algebraic geometrical terms (by using formal schemes for example).
\end{rmk}

\begin{prop}\label{prop:patterns}
Let $f:(\nC^3,0) \to (\nC^3,0)$ be a germ of the form \refeqn{example} satisfying the conditions of \refprop{exampleresolutionforms}.
Let $\pi:X_\pi \to (\nC^3,0)$ be any point modification adapted to $f$ and dominating $X_{\pi_0}$.
Let $S_\pi$ be the singular set of the saturated infinitesimal generator $\hat{\chi}_\pi$ of the lift $f_\pi$ of $f$ at $X_\pi$.
Then any positive-dimensional irreducible component $C_\pi$ of $S_\pi$ is a rational curve, and $(X_\pi, C_\pi, f_\pi)$ is either a $\ZeroC$-$\ZeroC$-pattern or a $\SpinC$-$\HalfC$-pattern.
\end{prop}
\begin{proof}
By \refprop{exampleresolutionforms}, the model $X_{\pi_0}$ has finitely many singularities, which are either simple corners, {\degspike}s or {\spincorner}s.

By blowing-up points over such families, we either stay in such families, or we obtain {\halfcorner}s.
Non-isolated singularities may arise only when blowing-up simple corners (and in this case we get $\ZeroC$-$\ZeroC$-patterns), or {\spincorner}s and {\halfcorner}s (and in both cases we get $\SpinC$-$\HalfC$-patterns).
To conclude, we need to control the strict transform of the cores $C$ of such patterns, when blowing-up points $p$ in the core.

Since the singularities above simple corners are theirselves simple corners, when we blow-up points in the core of $\ZeroC$-$\ZeroC$-patterns we still get $\ZeroC$-$\ZeroC$-patterns.

For the case of $\HalfC$-patterns, we need to determine the equations of the core $C$ at any point $p \in C$ with respect to the local coordinates at $p$ used to describe {\spincorner}s and {\halfcorner}s.

It is easy to check that for $\SpinC$-$\HalfC$-patterns coming from the blow-up of either a {\spincorner} or a {\halfcorner}, the core is given by $C=\{x=z=0\}$ (both at the special points where we have {\spincorner}s, or at generic points where we have {\halfcorner}s), see \refprop{spincornerblowup} and \refprop{halfcornerblowup}.

If we blow-up any point $p \in C$, the strict transform $\wt{C}$ of $C$ intersects the exceptional divisor necessarily at the {\spincorner} at $p=[0:1:0]$, and it is locally given by $\wt{C}=\{x=y=0\}$.

This situation is stable by further blow-ups, and we are done.
\end{proof} 

We now study the behavior of these patterns under blow-up of their cores.

\subsection{Blow-up of $\ZeroC$-$\ZeroC$-patterns}

\begin{lem}\label{lem:R0patterngenform}
Let $(X,C,f)$ be a $\ZeroC$-$\ZeroC$-pattern given by \refprop{patterns}.
Then for point $p \in C$ there exists local coordinates $(x,y,z)$ at $p$ so that $C=\{x=y=0\}$ and $f$ is of the form \refeqn{simplecorner}, with $R \in \langle x,y\rangle$.
\end{lem}
\begin{proof}
From \refprop{patterns} $\ZeroC$-patterns arise when blowing up simple corners, and a direct computation shows that locally $f$ can be written as in \refeqn{simplecorner} with $C=\{x=y=0\}$.
Imposing that points in $C$ are singular for $f$ imply that $R$ vanishes at all points in $C$, which is equivalent to asking $R \in \langle x,y\rangle$. 
\end{proof}
	
\begin{prop}\label{prop:R0pattern}
Let $(X,C,f)$ be a $\ZeroC$-$\ZeroC$-pattern given by \refprop{patterns}, and let $\pi:\wt{X} \to (X,C)$ be the blow-up of $C$.
Denote by $E=\pi^{-1}(C)$ the exceptional divisor,
and by $\wt{S}$ the set of singularities of the lift $\wt{f}$ of $f$ at $\wt{X}$.
Then $E \cap \wt{S}$ consists of exactly two sections $\wt{C}_0$ and $\wt{C}_\infty$ of $\pi|_E:E \to C$, not intersecting each-other.
Finally, for $t=0$ and $t=\infty$, $(\wt{X}, \wt{C}_t, \wt{f})$ defines a $\ZeroC$-$\ZeroC$-pattern, satisfying the same conditions as in \reflem{R0patterngenform}.
\end{prop}

\begin{proof}
To study the fiber above $p$, we have to consider two charts of $X$, where in local coordinates $\pi$ acts respectively as $\pi(x,y,z)=(x,xy,z)$, and $\pi(x,y,z)=(xy,y,z)$.
	
In the first case, $\wt{f}$ takes the form
\begin{equation}\label{eqn:simplecornersingcurveAx}
\wt{f}(x,y,z)=
\begin{pmatrix}
x+(x^{a+b} y^b z^c) x\big(\lambda + \langle x,z\rangle \big)\\
y+(x^{a+b} y^b z^c) y\big(\mu-\lambda + \langle x,z \rangle \big)\\
z+(x^{a+b} y^b z^c) \langle x \rangle
\end{pmatrix},
\end{equation}
where the rest in the latter coordinate belongs to $\langle xz \rangle$ whenever $c > 0$.
We study \refeqn{simplecornersingcurveAx} at points $(0,y_0,0)$ with $y_0 \in \nC$.
	
At $y_0=0$, we have a singular point and we clearly get a simple corner with the wanted properties.
When $y_0 \neq 0$, we get a regular point, since $\mu-\lambda \neq 0$.
	
The computations on the second chart are completely analogous, and left to the reader. We get another simple corner at the point associated to the direction $[0:1]$.
\end{proof}

\subsection{Blow-up of $\SpinC$-$\HalfC$-patterns}

\begin{lem}\label{lem:R3patterngenform}
Let $(X,C,f)$ be a $\SpinC$-$\HalfC$-pattern given by \refprop{patterns}.
For any point $p \in C$, there are coordinates $(x,y,z)$ so that $C=\{x=z=0\}$, $p=(0,y_0,0)$ and $f$ has the form:
\begin{equation}\label{eqn:R3Pattern}
f(x,y,z)=
\begin{pmatrix}
x+y^bz^c(x+P)\\
y+y^{B}z^{c+1}Q\\
z+y^bz^{c+2}R
\end{pmatrix},
\end{equation}
with $c \geq 1$ and $P \in \langle z \rangle$.
Moreover either $B-1=b\geq 1$, or $b=B=0$.
\end{lem}
\begin{proof}
From \refprop{patterns}, $\HalfC$-patterns arise when blowing up {\spincorner}s and (\hcnonsimple) {\halfcorner}s.
A direct computation shows that there one can find coordinates $(x,y,z)$ at $p$ so that $f$ is of the form \refeqn{spincorner} or \refeqn{halfcorner}, and $C=\{x=z=0\}$, or $C=\{x=y=z\}$ for {\spincorner}s. Being \refeqn{spincorner} symmetric on $y,z$, we may assume we are in the first case.
The statement follows from rewriting \refeqn{spincornerztrasl} of \refprop{spincornerblowup} under the form \refeqn{R3Pattern}, and from \refeqn{halfcornerR3fam} of \refprop{halfcornerblowup}.
\end{proof}

\begin{prop}\label{prop:R3pattern}
Let $(X,C,f)$ be a $\SpinC$-$\HalfC$-pattern given by \refprop{patterns}, and let $\pi:\wt{X} \to (X,C)$ be the blow-up of $C$.
Denote by $E=\pi^{-1}(C)$ the exceptional divisor, and by $\wt{S}$ the set of singularities of the lift $\wt{f}$ of $f$ at $\wt{X}$.
Then $E \cap \wt{S}$ consists of exactly two sections $\wt{C}_0$ and $\wt{C}_\infty$ of $\pi|_E:E\to C$, not intersecting eachother.
Finally,
\begin{itemize}
\item $(\wt{X}, \wt{C}_\infty, \wt{f})$ defines a $\ZeroC$-$\ZeroC$-pattern, satisfying the same conditions as in \reflem{R0patterngenform};
\item $(\wt{X}, \wt{C}_0, \wt{f})$ defines a $\SpinC$-$\HalfC$-pattern, admitting local coordinates of the form \refeqn{R3Pattern}.
\end{itemize}
\end{prop}
\begin{proof}
Let $p \in C$ be any point in the core, and pick $(x,y,z)$ local coordinates so that $f$ is written as in \refeqn{R3Pattern}.
We write $P=z \big(\alpha(y) + \langle x,z\rangle\big)$.
To study the fiber above $p$, we have to consider two charts of $X$, where in local coordinates $\pi$ acts respectively as $\pi(x,y,z)=(x,y,xz)$, and $\pi(x,y,z)=(xz,y,z)$.

In the first case, $\wt{f}$ takes the form
\begin{equation}\label{eqn:R3Patternsingcurvex}
\wt{f}(x,y,z)=
\begin{pmatrix}
x\Big(1+x^cy^bz^c\big(1+z\alpha(y) + \langle xz \rangle\big)\Big)\\
y+x^{c+1}y^{B}z^{c+1}Q\circ \pi\\
z\Big(1+x^cy^bz^c\big(-1-z\alpha(y)+\langle xz \rangle\big)\Big) 
\end{pmatrix}.
\end{equation}
The singular points if $\wt{f}$ in the exceptional divisor $E=\{x=0\}$ are of the form $(0,y_0,z_0)$ with $z_0(1+z_0\alpha(y_0))=0$.

When $y_0$ varies, the closure of points $z_0=0$ define a rational curve $C_\infty$.
From \refeqn{R3Patternsingcurvex} we deduce that $(\wt{X},C_\infty,\wt{f})$ is a $\ZeroC$-$\ZeroC$-pattern satisfying the conditions of \reflem{R0patterngenform}.

To study the points satisfying $z_0\alpha(y_0)=-1$, we work on the second chart. We get
\begin{equation}\label{eqn:R3Patternsingcurvez}
\wt{f}(x,y,z)=
\begin{pmatrix}
x+y^bz^c\big(x+\alpha(y) + \langle z \rangle\big)\Big)\\
y+y^{B}z^{c+1}Q\circ \pi\\
z+y^bz^{c+2}R\circ \pi
\end{pmatrix}.
\end{equation}

In this chart, the singularities in $E=\{z=0\}$ have the form $q_0=(x_0,y_0,0)$ with $x_0=-\alpha(y_0)$.
These points form a rational curve $C_0$ not intersecting $C_\infty$, for which $(\wt{X},C_0,\wt{f})$ is a $\SpinC$-$\HalfC$-pattern.
More precisely, $\wt{f}$ is a {\spincorner} at $q_0$ exactly when $y_0=0$ and $b\geq 1$, i.e., if and only if $f$ is a {\spincorner} at $p$.

By the change of coordinates $(x,y,z) \mapsto (x+\alpha(y),y,z)$, we get an expression of the form \refeqn{R3Pattern}.
\end{proof}

We sum up the study of blow ups of singular points and patterns in \reffig{resumefamilies}.

\def\frat{0.32}
\def\fver{5mm}
\begin{figure}
	\centering
	\begin{minipage}[t]{\columnwidth}
		\def\svgwidth{\frat\columnwidth}
		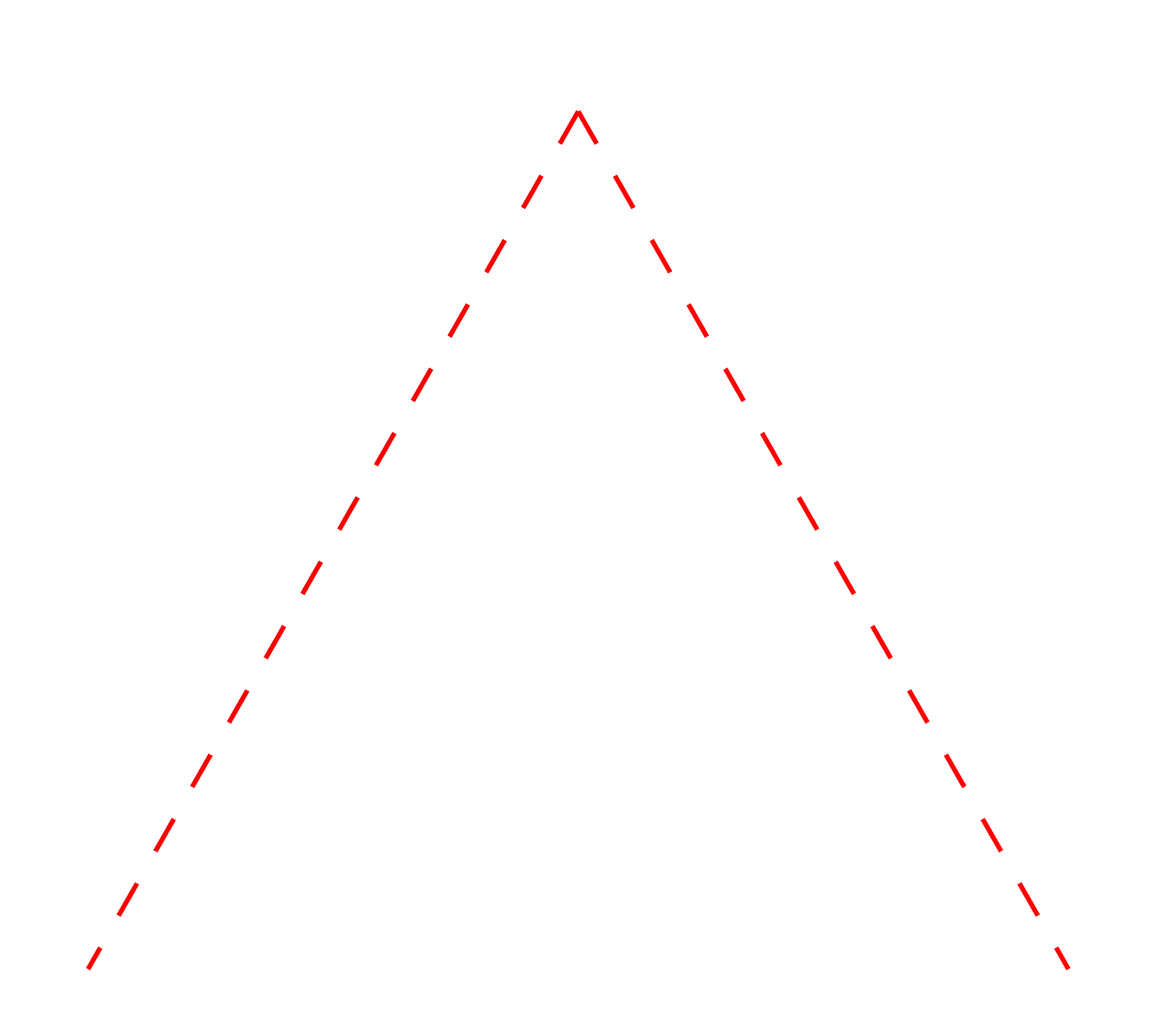
		\ 
		\def\svgwidth{\frat\columnwidth}
		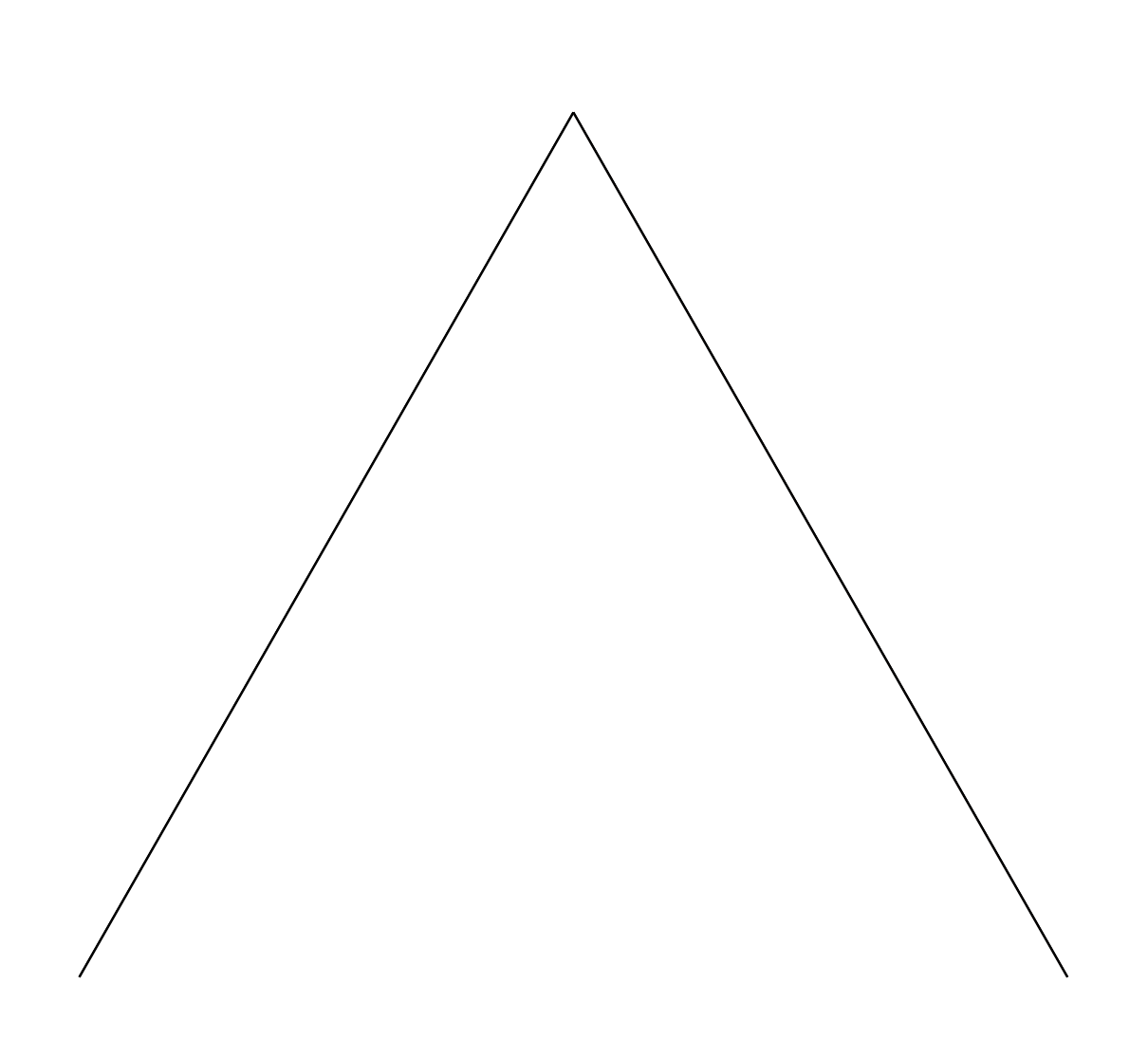
		\
		\def\svgwidth{\frat\columnwidth}
		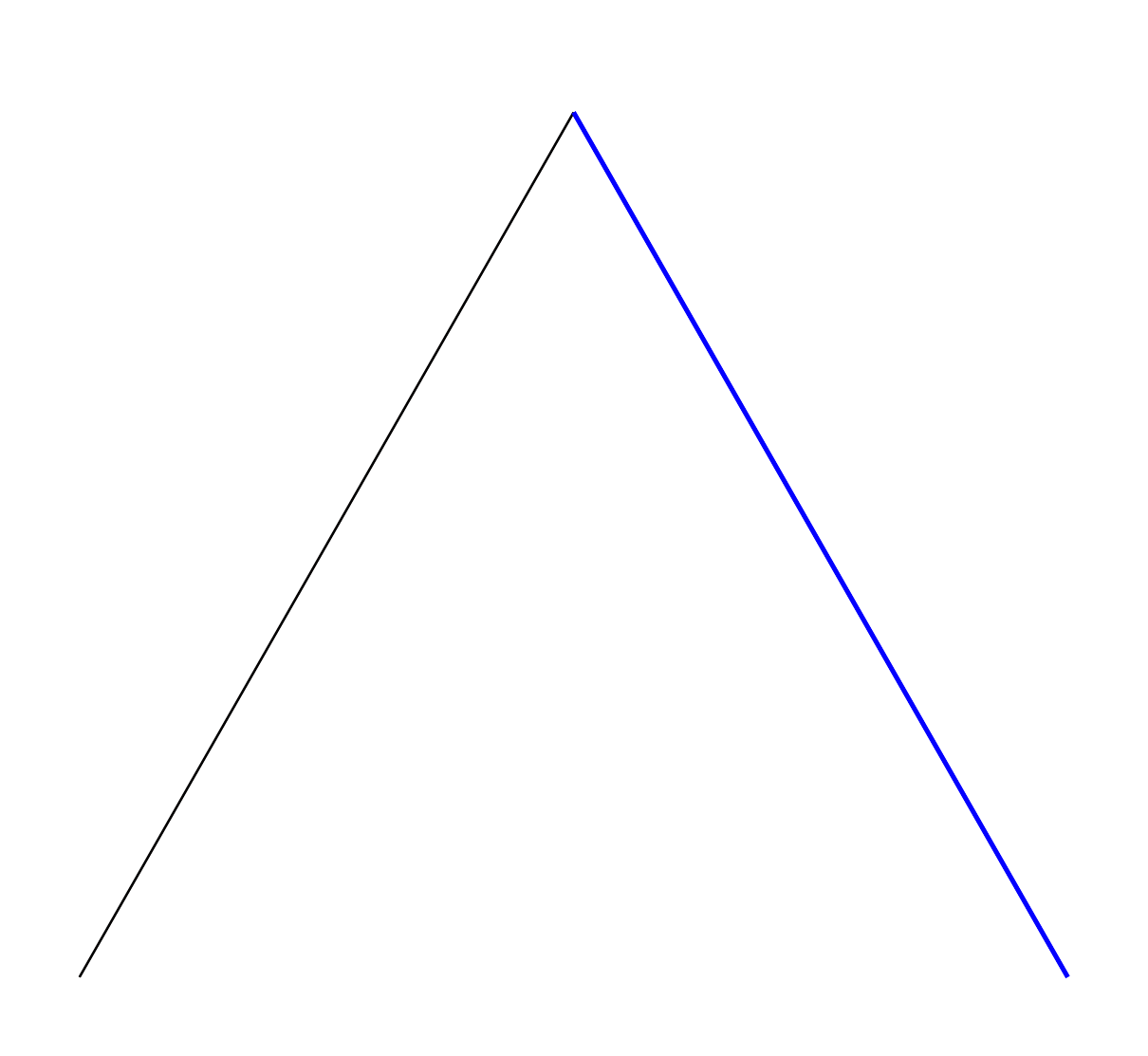
		\vspace{\fver}
		
		\def\svgwidth{\frat\columnwidth}
		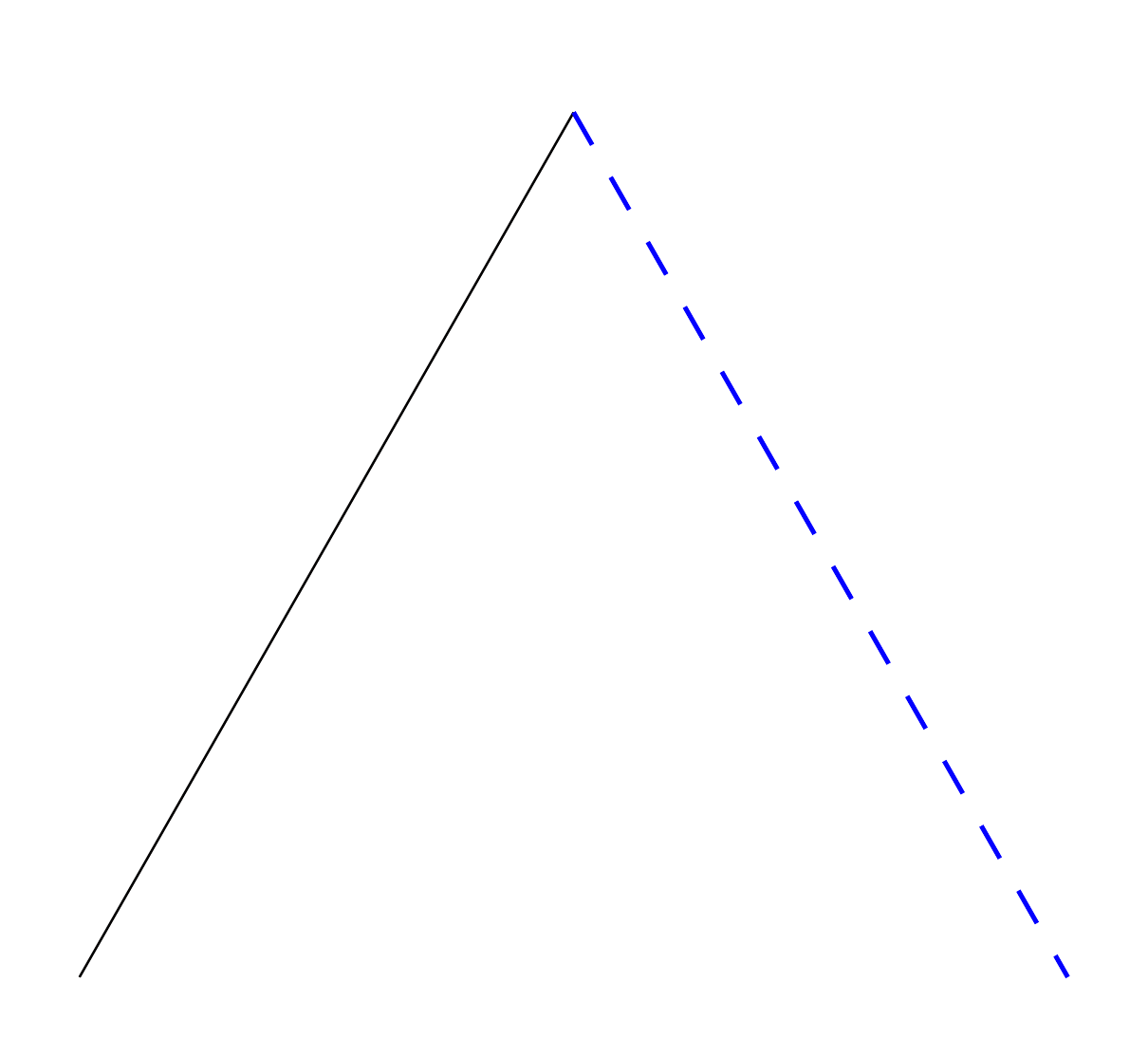
		\ 
		\def\svgwidth{\frat\columnwidth}
		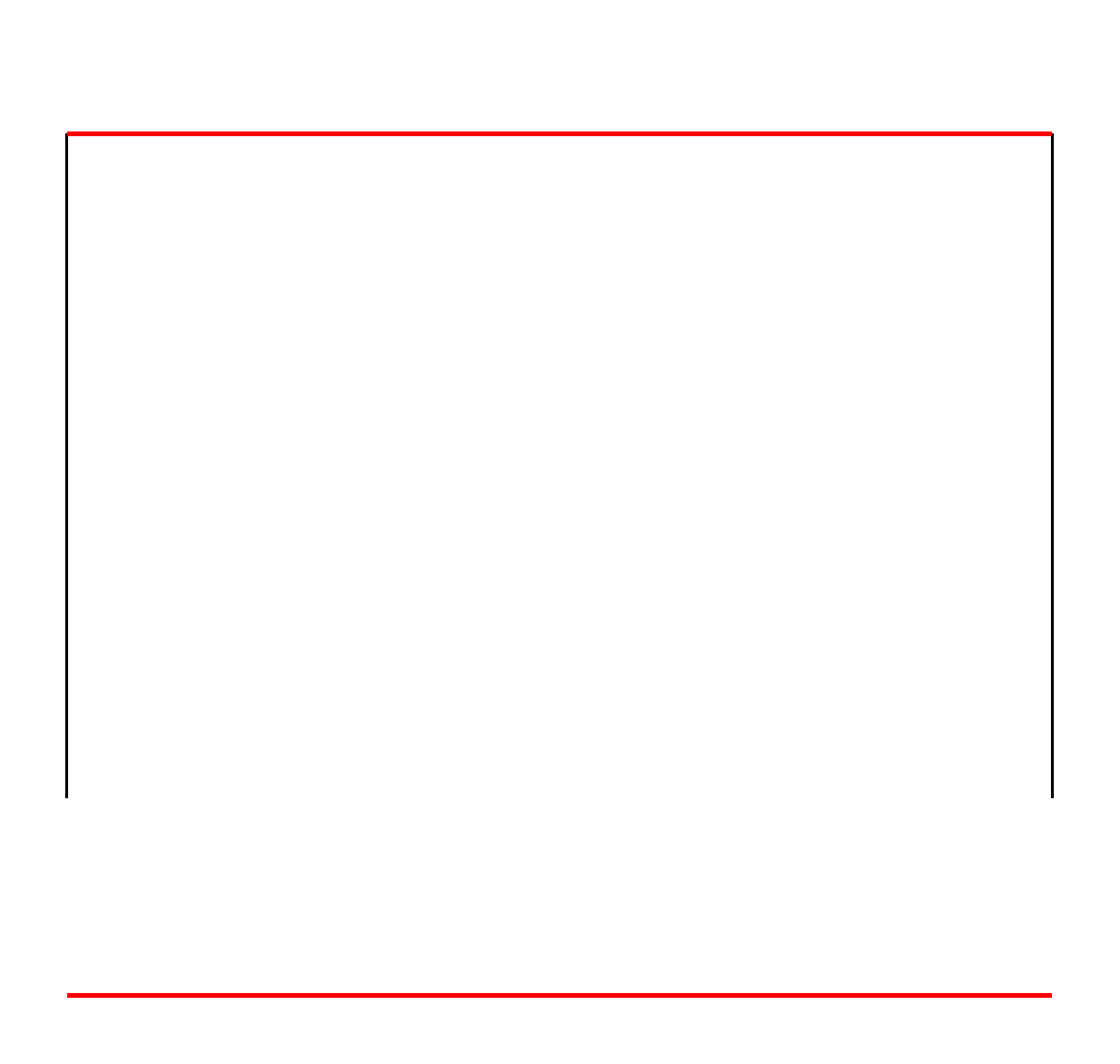
		\ 
		\def\svgwidth{\frat\columnwidth}
		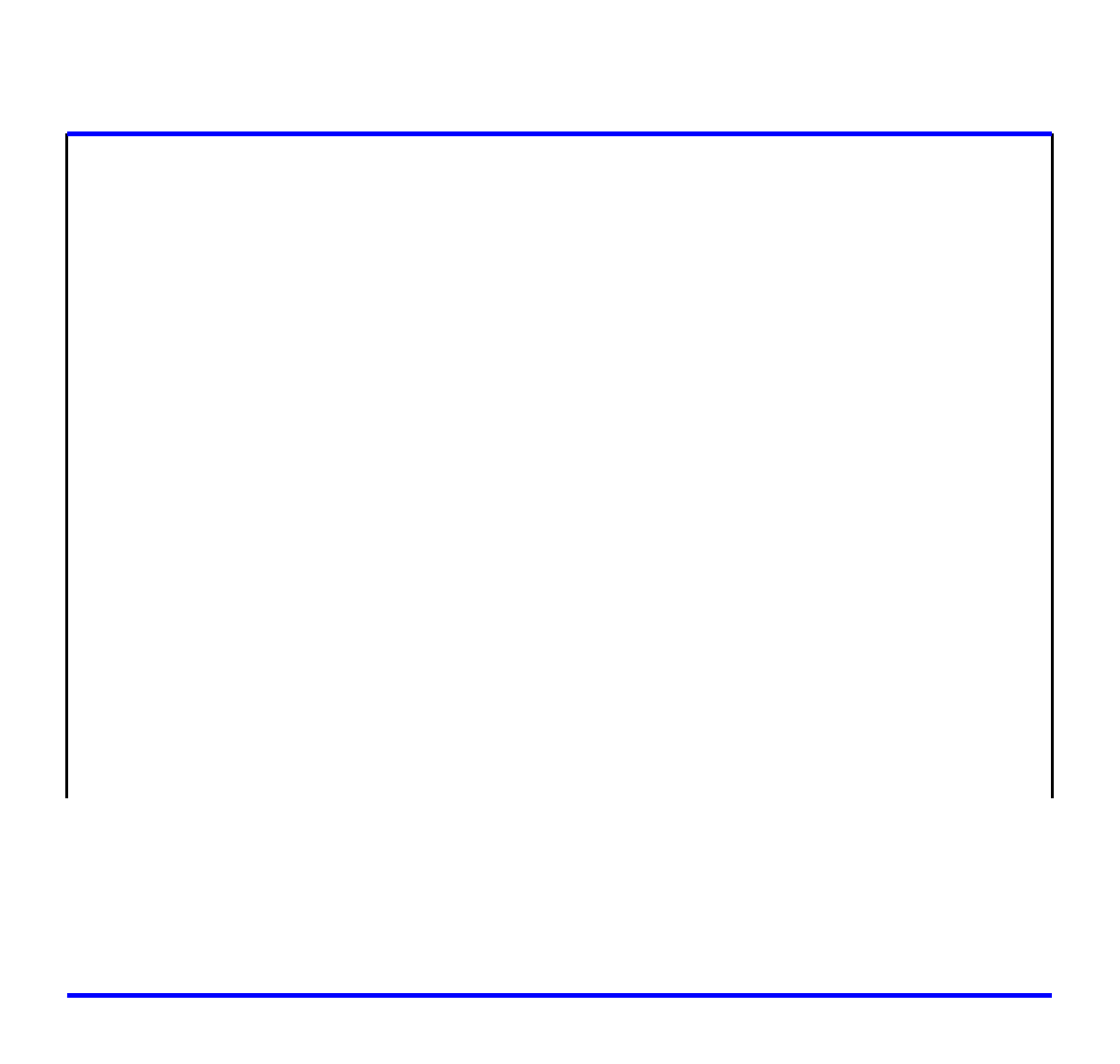
		\
		\vspace{\fver}
		\hfill
		\text{}
		\caption{Blow-up of special families and patterns.}
		\label{fig:resumefamilies}
	\end{minipage}
\end{figure}

\subsection{Proof of \refthm{mainnonondeg}}

Let $f:(\nC^3,0)\to (\nC^3,0)$ be a generic germ of the form \refeqn{example} (i.e., with parameters $P,Q,R$ satisfying the conditions of \refprop{exampleresolutionforms}).

Any regular modification $\pi:X_\pi \to (\nC^3,0)$ adapted to $f$ and dominating $\pi_0$ is either a point modification, or it dominates $\wt{\pi}_0$ given by \refprop{exampleresolutionforms}.

By \refprop{patterns}, in the first case the only patterns that appear are $\ZeroC$-$\ZeroC$-partterns or $\SpinC$-$\HalfC$-patterns.
In the second case, patterns may appear from regular modifications adapted to the dynamics above simple corners or {\spincorner}s, which are again $\ZeroC$-$\ZeroC$-partterns or $\SpinC$-$\HalfC$-patterns.
By \refprop{R0pattern} and \refprop{R3pattern}, no new patterns arise when blowing-up cores these two type of patters, and similarly the blow-up of points doesn't provide new type of special points in a pattern.
Hence for any such modification $\pi$, we have only simple corners, {\degspike}s, {\spincorner}s and {\halfcorner}s, which admit no non-exceptional non-degenerate singular directions.

\section{Invariant curves and parabolic manifolds}\label{sec:invariantcurves}

\subsection{Invariant curves}\label{ssec:inariantcurves}

\subsubsection{{\Degspike}s}

\begin{prop}\label{prop:degspikeinvcurve}
Let $f:(\nC^3,0) \to (\nC^3,0)$ be a {\degspike} of the form \refeqn{degspike}. Then there exists a unique $f$-invariant formal curve $C$ not contained in $E:=\{z=0\}$.
Moreover, $C$ is smooth and transverse to $E$.
\end{prop}
\begin{proof}
By \refprop{degspikeblowuppoint} (see also \reffig{resumefamilies}), there exists a unique sequence of infinitely near points $\mf{p}$ consisting of singular points for the lifts of $f$. By \refprop{invariantcurve}, these points induce a formal invariant curve $C_\mf{p}$ which is $f$-invariant, smooth and transverse to $E$.

Let now $C$ be a formal $f$-invariant curve.
Since curves are resolved by point blow-ups, there exists a point modification $\pi:X_\pi \to (\nC^3,0)$ so that the curve $C$ lifts to $C_\pi$ which is smooth and transverse to the exceptional divisor $E_\pi$ of $\pi$.
Denote by $f_\pi$ the lift of $f$ at $X_\pi$.

Then $C_\pi$ must intersect $E_\pi$ transversely at a point $p$, and $f_\pi$ must be a {\degspike} at $p$.
In fact, by \refthm{parabolicmanifolds} $f_\pi$ admits a parabolic manifold tangent to $C_\pi$, and by \refcor{singularsuffices} we deduce that $p$ must be a singular point for $f_\pi$. Since simple corners don't admit formal invariant curves (not lying in the exceptional divisor), we must have that $p$ is a {\degspike}.

Since there is a unique sequence $\mf{q}$ of infinitely near points consisting of singular points and satisfying the conditions of \refprop{constructioncurve} above a {\degspike}, we deduce that $C_\pi \equiv C_\mf{q}$, and by projecting down, we get $C \equiv C_\mf{p}$.
\end{proof}

\subsubsection{{\Halfcorner}s}

\begin{prop}\label{prop:halfcornerinvcurve}
Let $f:(\nC^3,0) \to (\nC^3,0)$ be a {\halfcorner} of the form \refeqn{halfcorner}.
Write $Q=\beta + b_x x + b_y y + b_z z + \mf{m}^2$, and $R=\gamma+\mf{m}$.
Set $E=\{z=0\}$.
\begin{itemize}
\item If $\beta \neq 0$ (i.e., the {\halfcorner} is {\hcsimple}), then $f$ does not admit any formal $f$-invariant curve not contained in $E$.
\item If $\beta = 0$ (i.e., the {\halfcorner} is {\hcnonsimple}), and
\begin{equation}\label{eqn:halfcornernonres}
b_y \not \in \gamma \nN^*\text,
\end{equation}
then $f$ admits a unique smooth $f$-invariant formal curve $C$ transverse to $E$.
\end{itemize}
\end{prop}
\begin{proof}
Let us write $f$ under the following form:
\begin{equation}\label{eqn:halfcornerinvcurve1}
f(x,y,z)=\begin{pmatrix}
x+z^c(x+ P)\\
y+z^{c+1}Q\\
z+z^{c+2}R
\end{pmatrix}
\end{equation}
with $P=\sum_{k \geq 1} z^k a_k(y) + \langle xz\rangle$.

\thmstep{Step $1$}
We want to show that up to formal conjugacy, we can suppose that $a_k \equiv 0$ for all $k$, hence $P \in \langle xz \rangle$.

When $k=1$, the condition $a_1 \equiv 0$ corresponds to having that the $\SpinC$-$\HalfC$-pattern obtained after blowing-up $\{x=z=0\}$ has core which corresponds to the intersection of the strict transform of $\{x=0\}$ and the exceptional divisor.

Using \refeqn{R3Patternsingcurvez} and arguing by induction, having $a_1 \equiv \ldots \equiv a_k \equiv 0$ corresponds to the analogous statement for the iterated blow-up $h$-times, $h=1, \ldots, j$, of the cores of the $\SpinC-\HalfC$-patterns we meet at each step.
We set $X_0=\nC^3$ (as a germ at the origin), and $X_k$ to be the blow-up of $X_{k-1}$ along $\{x=z=0\}$.
Since a change of coordinates of the form $x'=x+\alpha(y)$ in $X_k$ corresponds to a change of coordinates of the form $x'=x+z^k\alpha(y)$, in $X_0$, these change of coordinates converge to a formal change of coordinates $x'=x+A(y,z)$.

\thmstep{Step $2$}
By Step $1$, we may assume $P \in \langle xz \rangle$.
This corresponds to having the surface $S=\{x=0\}$ invariant by $f$.
Set $g=f|_S:S \to S$, and let $\hat{\xi}$ be the saturated infinitesimal generator of $g$.

If $\beta \neq 0$, we get that $\hat{\xi}$ is regular at the origin, and tangent to the exceptional divisor $E=\{z=0\}$. In particular there are no complex separatrices for $\hat{\xi}$ besides $E \cap S$.

If $\beta = 0$, we get a singularity at the origin, whose linear part is $(b_y y + b_z z) \partial_y + \gamma z \partial_z$.
As long as $b_y$ and $\gamma$ do not both vanish, we get a log-canonical singularity.
The condition $b_y \not \in \gamma \nN^*$ ensures that the singularity is in fact canonical, and we have exactly two complex separatrices: one given by $E \cap S$, and the other transverse to $E$ in $S$. 
\end{proof}

\begin{rmk}\label{rmk:halfcornersinvcurvetangent}
The existence of formal invariant curves for {\hcnonsimple} {\halfcorner}s can be deduced directly from \refprop{halfcornerblowup}.
In fact, the computations made in the proof, show that when blowing-up such a germ, we obtain {\halfcorner}s with parameters
$$
\wt{\beta}(y_0)=b_z+(b_y-\gamma)y_0\text,\qquad
\wt{b}_y=b_y-\gamma\text,\qquad
\wt{\gamma}=\gamma\text,
$$
where $y_0 \in \nC$.
In particular, as long as
$b_y \not \in \gamma \nN^*\text,$
we may construct an increasin sequence of infinitely near points which are {\hcnonsimple} {\halfcorner}s, which identify a formal invariant curve by \refprop{invariantcurve}. 

One can also replace Step $2$ of \refprop{halfcornerinvcurve} by a direct computation, following the techniques developed in  \cite{ruggiero:rigidification, ruggiero:rigidgerms, ruggiero:superattrdim1charp}.
This would correspond to parametrize a curve $C$ transverse to $E$ inside $S$ as $(0,\hat{y}(t),t^e)$ for some $e \geq 1$ and formal power series $\hat{y}=\sum_{n \geq 1} y_n t^n \in \nC\fps{t}$. We then impose the invariance condition
\begin{equation}\label{eqn:halfcorner_invariancecond}
y \circ f(0,y(t),t^e) = \hat{y}\Big(\big(z \circ f(0,y(t),t^e)\big)^{\frac{1}{e}}\Big)\text,
\end{equation}
and solve this equation by expanding everything in formal power series on $t$.

When $\beta \neq 0$, the contradiction to the existence is obtained by checking \refeqn{halfcorner_invariancecond} at order $e(c+1)$.
When $\beta = 0$, for $e=1$ and for any $n > c+1$, \refeqn{halfcorner_invariancecond} contains a term of the form
$$
\big(b_y+(n-c-1)\gamma\big)y_{n-c-1} = \on{l.o.t.}
\text,
$$ 
where $\on{l.o.t.}$ is a polynomial expression depending on $y_h$ for $h < n-c-1$.
We deduce from this the existence and uniquenes of $\hat{y}$ solution of \refeqn{halfcorner_invariancecond}.
\end{rmk}

\subsubsection{{\Spincorner}s}

In the following result, we say that a curve $C$ is transverse to $E=\{yz=0\}$ if it is transverse to any irreducible component of $E$.
In other terms, if the strict transforms of $E$ and $C$ do not intersect on the exceptional divisor of the blow-up of the origin.

\begin{cor}\label{cor:spincornerinvcurve}
Let $f:(\nC^3,0)\to(\nC^3,0)$ be a {\spincorner} of the form \refeqn{spincorner}.
Write $Q=b_x x + b_y y + b_z z + \mf{m}^2$, and $R=c_x x + c_y y + c_z z + \mf{m}^2$.
Set $E=\{yz=0\}$.
\begin{itemize}
\item If $b_y=c_y$ and $b_z = c_z$, and they are not all vanishing, then there exists infinitely many $f$-invariant formal smooth curves transverse to $E$. 
\item If $b_y \neq c_y$ and $b_z \neq c_z$, and 
\begin{equation}\label{eqn:spincornernonres}
(c_z-b_z)(b_y-c_y) \not \in (b_yc_z-b_zc_y) \nN^*,
\end{equation}
then there exists a unique formal $f$-invariant curve smooth and transverse to $E$.
\item If exactly one of the two equalities $b_y=c_y$ and $b_z=c_z$ is satisfied, then there are no formal $f$-invariant curves transverse to $E$.
\end{itemize}
\end{cor}
\begin{proof}

Consider the blow-up of the origin. From \refprop{spincornerblowup}, the points $p(y_0)$ corresponding to the directions $[0:y_0:1]$ for $y_0 \in \nC^*$ have a {\hcnonsimple} {\halfcorner} when $y_0$ satisfies $b_z-c_z + y_0(b_y-c_y)=0$.
The parameters of the {\halfcorner} are given (up to a factor $y_0^b$) by:
$$
\wt{b}_y=y_0(b_y-c_y)=c_z-b_z\text,\qquad
\wt{\gamma}=c_z+y_0c_y\text,
$$
see \refeqn{spincornerztrasl}.

\begin{itemize}
\item If $b_y=c_y$ and $b_z=c_z$, then $p_0$ is a {\hcnonsimple} {\halfcorner} for all values of $y_0 \in \nC^*$, with parameters $\wt{b}_y=0$ and $\wt{\gamma}=c_z + y_0 c_y$.
As long as we do not have $c_y=c_z=0$, then for all $y_0$ but at most one special value, the corresponding {\hcnonsimple} {\halfcorner} at $p(y_0)$ satisfies the non-resonance condition \refeqn{halfcornernonres}, and there exists a unique invariant curve at $p(y_0)$ and transverse to the exceptional divisor.

\item If $b_y \neq c_y$ and $b_z \neq c_z$, the only {\hcnonsimple} {\halfcorner}  is obtained at $p(y_0)$ with $\displaystyle y_0=\frac{c_z-b_z}{b_y-c_y}$.
In this case, we have $\displaystyle \wt{\gamma}=\frac{\delta}{b_y-c_y}$, where $\delta=b_yc_z-c_yb_z$.

If $\delta=0$, being $\wt{b}_y \neq 0$, the condition \refeqn{halfcornernonres} is satisfied.
If $\delta \neq 0$, then the condition \refeqn{halfcornernonres} gives exactly \refeqn{spincornernonres}.
\item If exactly one of the two equalities $b_y=c_y$ and $b_z=c_z$ is satisfied, then $p(y_0)$ is a {\hcsimple} {\halfcorner} for all $y_0 \in \nC^*$.
By \refcor{spincornerinvcurve}, we have no invariant formal curves transverse to the exceptional divisor, and hence no $f$-invariant formal curves transverse to $E$.
\end{itemize}
\end{proof}

\begin{rmk}\label{rmk:spincornerinvcurve}
\refcor{spincornerinvcurve} does not deal with the existence of formal invariant curves that may be tangent to the exceptional divisor.

Given a {\spincorner} in the form \refeqn{spincorner}, and using the notations of \refcor{spincornerinvcurve}, we set $A=\begin{pmatrix}b_y&b_z\\c_y&c_z\end{pmatrix}$.

Without further mention, germs or patterns that we blow-up will be considered in the special coordinates used to obtain \reffig{resumefamilies}.

{\Spincorner}s may arise either blowing-up other {\spincorner}s, at the point associated to $[0:1:0]$ and $[0:0:1]$; or by blowing-up a {\halfcorner}, at the point associated to $[0:1:0]$. Finally they are also obtained by blowing-up the core of a $\SpinC$-$\HalfC$-pattern.

\begin{trivlist}
\itemcase{{\Spincorner}s}
Assume $f$ is a {\spincorner}, and consider the lift $\wt{f}$ with respect to the blow-up of the origin, at the point associated to $[0:0:1]$.
Then matrix associated to $\wt{f}$ is  $\wt{A}=\begin{pmatrix}0&b_z-c_z\\0&c_z\end{pmatrix}$.
We deduce that if $b_z \neq 2c_z$, there are no invariant curves transverse to the exceptional divisor for $\wt{f}$, while if $b_z=2c_z \neq 0$, then there exists infinitely many invariant curves.

By repeating this argument, we get infinitely many invariant curves as long as $b_z/c_z \in \nN^*$, or $c_y / b_y \in \nN^*$ (this last condition is obtained by exchanging the role of $y$ and $z$ and studying the direction $[0:1:0]$).
		
\itemcase{{\Halfcorner}s} We need to study the direction $[0:1:0]$. In this case, we get $\wt{A}=\begin{pmatrix}0&\beta \\0 &-\beta\end{pmatrix}$.
Hence, for {\hcsimple} {\halfcorner}s, we have $\beta \neq 0$, and no invariant curve transverse to the exceptional divisor exists.
For {\hcnonsimple} {\halfcorner}s, we have $\beta=0$, and the existence of invariant curves depend on the terms of higher degrees of $f$.

\itemcase{$\SpinC$-$\HalfC$-patterns}
In this case, it is easy to check from \refeqn{R3Patternsingcurvez} that the {\spincorner} $\wt{f}$ above a {\spincorner} of the core of the pattern satisfies $\wt{A}=A$, and we can apply directly \refcor{spincornerinvcurve}.
\end{trivlist}

One can also use formal computation techniques (see \ref{rmk:halfcornersinvcurvetangent}), which show again how the existence of invariant curves may depend on the higher order terms of $P,Q,R$.
\end{rmk}

\subsection{Parabolic manifolds}

\subsubsection{{\Degspike}s}

\begin{prop}\label{prop:RSreddegspike}
Let $f:(\nC^3,0) \to (\nC^3,0)$ be a {\degspike} of the form \refeqn{degspike}, and let $C$ be the unique $f$-invariant formal curve given by \refprop{degspikeinvcurve}.
Suppose that $C$ is not pointwise fixed by $f$.

For any $n \in \nN^*$, consider $\pi_n: X_n \to (\nC^3,0)$ the point modification, obtained recursively starting by $\pi_1$ the blow-up of $p_0=0$, and $\pi_n$ obtained from $\pi_{n-1}$ by blowing-up the point $p_{n-1}:=\pi_{n-1}^{-1}(0) \cap C_{n-1}$, where $C_{n-1}$ is the strict transform of $C$ by $\pi_{n-1}$.

Denote by $f_n$ the lift of $f$ at $X_n$, as a germ at $p_n$. 
Then for $n \gg 0$, the pair $(f_n,C_n)$ is in Ramis-Sibuya normal form.
\end{prop}
\begin{proof}
Up to a formal change of coordinates, we may assume that $C=\{x=y=0\}$, and $f$ is of the form \refeqn{degspike} with $P,Q \in \langle x,y \rangle\mf{m}$.
We can write the third coordinate of $f$ as
$$
z \circ f = z + z^{c+1} R = z+z^{c+1}\big(h(z) + \langle x,y \rangle\big)\text.
$$
Notice that $f|_C(z)=z+z^{c+1}h(z)$: up to a polynomial change of coordinates in the variable $z$, we may assume that
\begin{equation}\label{eqn:RSformalzdegspike}
h(z)=-z^e+\beta z^{c+2e} + \langle z^{c+2e+1} \rangle\text,
\end{equation}
with $e=\ord_0(h)$.
Notice that performing this change of coordinates changes the values of $\lambda$ and $\mu$, but their ratio stays invariant (see \refrmk{RSreddegspike}).

The blow-ups $\pi_n$ can be computed with respect to the $z$-chart, and the point $p_n$ corresponds to the origin in this chart.
By direct computation we get
\begin{equation}\label{eqn:degspikentimes}
f_n(x,y,z)
=
\begin{pmatrix}
\displaystyle \frac{x+z^c \big(\lambda x + z^n \langle x,y \rangle\big)}{\Big(1+z^{c} \big(h(z) + z^n \langle x,y \rangle\big) \Big)^n}\\[6mm]
\displaystyle \frac{y+z^c \big(\mu y + z^n \langle x,y \rangle\big)}{\Big(1+z^{c} \big(h(z) + z^n \langle x,y \rangle\big) \Big)^n}\\[6mm]
\displaystyle z\Big(1+z^{c} \big(h(z) + z^n \langle x,y \rangle\big) \Big)
\end{pmatrix}
\text.
\end{equation}

Set $r=c+e \geq c+1$, and take $n > c+2e$. Then \refeqn{degspikentimes} can be rewritten as
\begin{equation}\label{eqn:degspikentimesfps}
f_n(x,y,z)
=
\begin{pmatrix}
x(1+z^c \lambda + nz^r) + \langle z^{r+1} \rangle\\[4mm]
y(1+z^c \mu + nz^r) + \langle z^{r+1} \rangle\\[4mm]
z-z^{r+1} + \beta z^{2r+1} + \langle z^{2r+2} \rangle
\end{pmatrix}
\text,
\end{equation}
which is on the form \refeqn{RamisSibuya} with $\lambda^{-1} d_1(z) \equiv \mu^{-1} d_2(z) = z^c + \langle z^{2c}\rangle$. 
\end{proof}

\begin{rmk}\label{rmk:RSreddegspike}
When we change coordinates to obtain \refeqn{RSformalzdegspike}, the values of $\lambda$ and $\mu$ are replaced by $\lambda h_e^{-c/r}$, where $h(z)=h_e z^e + \langle z^{e+1} \rangle$.

Suppose we have a {\degspike} of the form
\begin{equation}\label{eqn:degspikeverygen}
f(x,y,z)
=
\begin{pmatrix}
x+z^c a(x,y,z)\\[2mm]
y+z^c b(x,y,z)\\[2mm]
z+z^{c+1} R(x,y,z) 
\end{pmatrix}
\text,
\end{equation}
and we want to put it under the form used in the computations of \refprop{RSreddegspike}.
This boils down to first put the linear part of $(a,b)$ (evaluated in $z=0$) in diagonal form, and then perform a change of coordinates $x \mapsto x + \alpha(z)$ and $y \mapsto y + \beta(z)$ for suitable formal power series $\alpha, \beta \in z\nC\fps{z}$.
In particular, if we need to know the action of $f|_C$ (where $C$ is the unique formal $f$-invariant curve transverse to $\{z=0\}$) up to order $c+1+e$, we only need to know the values of $\alpha$ and $\beta$ up to order $e$.
\end{rmk}

\begin{cor}\label{cor:parmflddegspike}
Let $f:(\nC^3,0) \to (\nC^3,0)$ be a {\degspike} (of Siegel type) of the form \refeqn{degspike}, and let $C$ be the unique $f$-invariant formal curve given by \refprop{degspikeinvcurve}.
Suppose that $C$ is not pointwise fixed by $f$, and let $r+1$ be the multiplicity of $(f-\id)|_{C}$ at the origin.

Then $f$ admits $r$ parabolic domains $\Delta_k$, 
which are of dimension $1$ or $2$.
\end{cor}
\begin{proof}
By \refprop{RSreddegspike}, we may assume up to point blow-ups that $f$ is in the Ramis-Sibuya normal form \refeqn{degspikentimesfps}.
Let $\xi=e^{2\pi \ui k/r}$ be a $r$-th root of unity.
Denote by $R_1$ and $R_2$ the invariants associated to the Ramis-Sibuya normal form given by \refeqn{RSdefRj}.
Since $d_1(z)$ and $d_2(z)$ are proportional, and their ratio is $\lambda/\mu \in \nR_{<0}$ if follows that $R_1(\xi) + R_2(\xi)=0$.
Hence either $R_j(\xi)=0$ for $j=1,2$, and in this case $\xi$ is a saddle direction for both coordinates, or $R_j(\xi) \neq 0$, and in this case $\xi$ is a node direction for exactly one of the two coordinates.
We conclude by \refthm{RSparabolicmanifolds}.
\end{proof}

\begin{rmk}
Notice that if $\lambda \xi^{c} \in \ui \nR^*$, then its square is a non-vanishing real number. It follows that $R_j(\xi) \neq 0$ whenever $2c < r$, and the only case where we can have parabolic domains of dimension $1$ is when $2c \geq r$ (i.e., $e \leq c$), and there exists a $r$-th root of unity so that $\lambda \xi^r$ belongs to $\ui \nR$.
\end{rmk}

\subsubsection{{\Halfcorner}s}\label{sssec:parmfldhalfcorner}

\begin{prop}\label{prop:RSredhalfcorner}
Let $f:(\nC^3,0) \to (\nC^3,0)$ be a {\hcnonsimple} {\halfcorner} of the form \refeqn{halfcorner}.
Suppose that $f$ admits a $f$-invariant smooth formal curve $C$ transverse to $\{z=0\}$.
Suppose that $C$ is not pointwise fixed by $f$.
	
For any $n \in \nN^*$, consider $\pi_n: X_n \to (\nC^3,0)$ the point modification, obtained recursively starting by $\pi_1$ the blow-up of $p_0=0$, and $\pi_n$ obtained from $\pi_{n-1}$ by blowing-up the point $p_{n-1}:=\pi_{n-1}^{-1}(0) \cap C_{n-1}$, where $C_{n-1}$ is the strict transform of $C$ by $\pi_{n-1}$.
	
Denote by $f_n$ the lift of $f$ at $X_n$, as a germ at $p_n$. 
Then for $n \gg 0$, the pair $(f_n,C_n)$ is in Ramis-Sibuya normal form.
\end{prop}
\begin{proof}
Up to a formal change of coordinates, we may assume that $C=\{x=y=0\}$, and $f$ is of the form \refeqn{halfcorner} with $P \in \langle x,y \rangle \mf{m}$ and $Q \in \langle x,y \rangle$.
We can write the third coordinate of $f$ as
$$
z \circ f = z + z^{c+2} R = z+z^{c+2}\big(h(z) + \langle x,y \rangle\big)\text,
$$
with $h(z)=h_e z^e + \langle z^{e+1}\rangle$ for some $e \geq 0$.
Notice that $f|_C(z)=z+z^{c+2}h(z)$: up to a polynomial change of coordinates in the variable $z$, we may assume that
\begin{equation}\label{eqn:RSformalz}
h(z)=-z^e+\beta z^{c+2e+1} + \langle z^{c+2e+2} \rangle\text.
\end{equation}
In this case, the first coordinate of $f$ becomes $x+z^c(\alpha^c x + \langle x,y\rangle \mf{m})$, where $\alpha^r=1/h_e$.

The blow-ups $\pi_n$ can be computed with respect to the $z$-chart, and the point $p_n$ corresponds to the origin in this chart.
By direct computation we get
\begin{equation}\label{eqn:halfcornerntimes}
f_n(x,y,z)
=
\begin{pmatrix}
\displaystyle \frac{x+z^c \big(\alpha^c x + z^{n} \langle x,y \rangle\big)}{\Big(1+z^{c+1} \big(h(z) + z^n \langle x,y \rangle\big) \Big)^n}\\[6mm]
\displaystyle \frac{y+z^{c+n} \langle x,y \rangle\big)}{\Big(1+z^{c+1} \big(h(z) + z^n \langle x,y \rangle\big) \Big)^n}\\[6mm]
\displaystyle z\Big(1+z^{c+1} \big(h(z) + z^n \langle x,y \rangle\big) \Big)
\end{pmatrix}
\text.
\end{equation}

Set $r=c+1+e \geq c+1$, and take $n > c+2e+1$. Then \refeqn{halfcornerntimes} can be rewritten as
\begin{equation}\label{eqn:halfcornerntimesfps}
f_n(x,y,z)
=
\begin{pmatrix}
x(1+(\alpha z)^c+ nz^r) + \langle z^{r+1} \rangle\\[4mm]
y(1+ nz^r) + \langle z^{r+1} \rangle\\[4mm]
z-z^{r+1} + \beta z^{2r+1} + \langle z^{2r+2} \rangle
\end{pmatrix}
\text,
\end{equation}
which is on the form \refeqn{RamisSibuya} with $d_1(z) = (\alpha z)^c+\langle z^{c+1} \rangle$ and $d_2(z) \equiv 0$. 
\end{proof}
Notice that \refprop{RSredhalfcorner} applies in particular when a {\hcnonsimple} {\halfcorner} satisfies \refeqn{halfcornernonres} with $\gamma \neq 0$, and in this case $e=0$ and $r=c+1$.

\begin{cor}\label{cor:parmfldhalfcorner}
Let $f:(\nC^3,0) \to (\nC^3,0)$ be a {\hcnonsimple} {\halfcorner} satisfying the same hypotheses of \refprop{RSredhalfcorner}.
Suppose that $C$ is not pointwise fixed by $f$, and let $r+1 \geq c+2$ be the multiplicity of $(f-\id)|_{C}$ at the origin.

Then $f$ admits $r$ parabolic manifolds $\Delta_k$, which are of dimension $1$ or $2$.
\end{cor}
\begin{proof}
By \refprop{RSreddegspike}, we may assume up to point blow-ups that $f$ is in the Ramis-Sibuya normal form \refeqn{halfcornerntimesfps}.
Denote by $R_1$ and $R_2$ the invariants associated to the Ramis-Sibuya normal form given by \refeqn{RSdefRj}.
Let $\xi=e^{2\pi \ui k/r}$ be a $r$-th root of unity.
Since $d_2=0$, then $R_2=0$, and all directions are a saddle in the second coordinate.
For the first coordinate, we get a node or a saddle depending on the sign of the real part of $\alpha^c \xi^c$.
We conclude by \refthm{RSparabolicmanifolds}.
\end{proof}

\subsection{Proof of \refthm{mainparmfld}}

Here we apply the results of the previous sections to our example \refeqn{example}.
From \refprop{exampleresolutionforms}, we get a model with $11$ singularities. Among those, we get two {\degspike}s at $p_1$ and $p_2$, and three {\spincorner}s at $p_{3,2}$, $p_{4,2}$ and $q_1$.
The others are simple corners and do not give rise to parabolic manifolds, see \cite{abate-tovena:paraboliccurvesC3}.

\begin{trivlist}
\itemcase{$p_1$}
At $p_1$ the lift of $f$ takes the form \refeqn{taylorIz}, which is a {\degspike}.
To compute the parameters appearing in \refprop{RSreddegspike} and \refcor{parmflddegspike}, we need some further change of coordinates (see \refrmk{RSreddegspike}).

After performing the change of coordinates $x'=x+y$, $y'=x-y$, we get
\begin{equation}\label{eqn:taylorIzprime}
f_1(x',y',z) =
\begin{pmatrix}
x' + z^2 \big(-x'+ (P_{004}+Q_{004})z + \mf{m}^2\big)\\[2mm]
y' + z^2 \big(y'+ (P_{004}-Q_{004})z  + \mf{m}^2\big)\\[2mm]
z + z^3\big(-\tfrac{1}{2}x'-\tfrac{1}{2}y'+ z R_{004} + \mf{m}^2\big)\end{pmatrix}
\text.
\end{equation}
After the change of coordinates $x''=x'-(P_{004}+Q_{004})z$, $y''=y'+(P_{004}-Q_{004})z$ we finally get
\begin{equation}\label{eqn:taylorIzsecond}
f_1(x'',y'',z) =
\begin{pmatrix}
x'' + z^2 \big(-x'' + \mf{m}^2\big)\\[2mm]
y'' + z^2 \big(y''+ \mf{m}^2\big)\\[2mm]
z + z^3\big(-\tfrac{1}{2}x''-\tfrac{1}{2}y''+ z (R_{004} -Q_{004})+ \mf{m}^2\big)\end{pmatrix}
\text.
\end{equation}
We deduce that if $R_{004} \neq Q_{004}$, the parameters of \refprop{RSreddegspike} are $c=2$, $e=1$, and we get $r=c+e=3$ parabolic manifolds attached to $p_1$.

Here $\lambda = -1$, $\mu=1$, $h_e=R_{004}-Q_{004}$, and by a direct check we get that all parabolic manifolds have dimension $2$, unless $h_e^2 \in \ui \nR$, in which case one of the three parabolic manifolds has dimension $1$, while the others have dimension $2$.

\itemcase{$p_2$}
For the {\degspike} at $p_2$, computations are similar and left to the reader.
In this case we get again $c=2$, $e=1$, and $r=3$ parabolic manifolds whenever $h_e:=R^{(4)}(0,1,1) \neq 0$. They are all of dimension $2$, unless $h_e^2 \in \nR$, where one of the three parabolic manifolds has dimension $1$.

\itemcase{$p_{3,2}$}
At the point $p_{3,2}$ we have a {\spincorner} of the form \refeqn{taylorp32}. After the change of coordinates $x=-u-\tfrac{3}{4}y + (\alpha-\tfrac{1}{2}\beta)z$, we get the form
$$
\wt{f}_1(x,y,z) =
\begin{pmatrix}
x + y^2z^2\Big(-x + \mf{m}^2\Big)\\[3mm] 
y + y^3z^2 \Big(2x-\tfrac{9}{4}y + 2\alpha z + \mf{m}^2\Big)\\[3mm]
z + y^2z^3\Big(-2x+\tfrac{13}{4}y - 2\alpha z + \mf{m}^2\Big)
\end{pmatrix}
\text.
$$ 
In particular, the parameters of \refcor{spincornerinvcurve} are given by $b_y-c_y = -\frac{11}{2}$, $c_z-b_z=-4\alpha$, and $\delta=b_yc_z-c_y b_z = -2\alpha$.

The ratio $(b_y-c_y)(c_z-b_z)/\delta$ of \refeqn{spincornernonres} equals $-11$, which is not a positive integer, hence the conditions of \refcor{spincornerinvcurve} are satisfied, and there exists a unique formal $f$-invariant curve smooth and transverse to the exceptional divisor.

If we blow-up the origin via the map $\pi(x,y,z)=(xz,yz,z)$, we get the lift
$$
\wt{f}_2=
\begin{pmatrix}
x + y^2z^4\Big(-x + \langle z \rangle\Big)\\[3mm] 
y + y^3z^5 \Big(4x-\tfrac{11}{2}y + 4\alpha + \langle z \rangle\Big)\\[3mm]
z + y^2z^6\Big(-2x+\tfrac{13}{4}y - 2\alpha + \langle z \rangle\Big)
\end{pmatrix}\text.
$$
At the point $y_0=\frac{8}{11}\alpha$ we get a {\hcnonsimple} {\halfcorner}, with parameters $c=4$, $\gamma=\tfrac{4}{11}\alpha$. Being $\alpha \neq 0$, we get $e=0$.
By \refcor{parmfldhalfcorner}, we get $r=c+1+e=5$ parabolic manifolds, which are all of dimension $2$, unless $\alpha^4 \in \ui \nR$, in which case one of the five parabolic manifolds has dimension $1$.

\itemcase{$p_{4,2}$} Since the germ $f_1$ at $p_3$ is conjugated to the one at $p_4$ (see \refrmk{symmetry}), a similar situation arises above $p_{4,2}$.

\itemcase{$q_1$}
Finally, at the point $q_1$ we have a {\spincorner} of the form \refeqn{taylorIyIIxIIIxIVx}.
By conjugating by the map $\phi(x,y,z)=(z+R_{040}y, x, y)$, we get
\begin{equation}\label{eqn:taylorq1}
f_4(x,y,z) =
\begin{pmatrix}
x+ y^7z^2\Big(x+\mf{m}^2\Big)\\[3mm]
y+ y^8z^2\Big(x +R_{040}z + \mf{m}^2\Big)\\[3mm]
z+ y^7z^3\Big(-3x -3R_{040}z + \mf{m}^2\Big)
\end{pmatrix}
\text.
\end{equation}   

In this case we have $b_y=c_y=0$ and $c_z=-3 b_z \neq 0$. In particular, there are no formal $f_4$-invariant curves transverse to $E$ by \refcor{spincornerinvcurve} (see also \refrmk{spincornerinvcurve}), but we cannot exclude $f_4$-invariant curves tangent to $E$ (see \refsssec{overq1}).

\end{trivlist}

\subsection{Further remarks}

\subsubsection{Curve blow-ups over {\degspike}s}

When studying resolution of singularities for vector fields, it is often natural to consider (possibly weighted) blow-ups of centers that are invariant by the dynamics (and not necessarily contained in the singular locus).

In our setting, this would correspond to allowing the blow-up of curves that are invariant by the saturated infinitesimal generator $\hat{\chi}$ of $f$ (in a given model), and contained in the exceptional divisor (obtained from previous blow-ups).

In the case of {\degspike}s (of Siegel type), the study can be easily done, since we can determine explicitly such curves.
In fact, if $f$ is a {\degspike} of the form \refeqn{degspike}, then the restriction of the saturated infintesimal generator $\hat{\chi}$ on $E=\{z=0\}$ gives a canonical singularity of siegel type, which admits exactly two (strong) complex separatrices.
Up to a (possibly formal, since the coordinates of $\hat{\chi}$ do not converge in general) change of coordinates, we may assume that these curves are $x=0$ and $y=0$.
Hence we may assume that the conditions $x|P(x,y,0)$ and $y|Q(x,y,0)$ are satisfied.
The next proposition gives the description of the lift of a {\degspike} when we blow-up one of the two complex separatrices (the other is completely analogous, we just need to interchange the role of $x$ and $y$).

\begin{prop}\label{prop:degspikecurve}
Let $f:(\nC^3,0) \to (\nC^3,0)$ be a {\degspike} of the form
\begin{equation}\label{eqn:degspikecurve}
f(x,y,z)=
\begin{pmatrix}
x+z^c\big(\lambda x (1+a(x,y)) + zP\big)\\
y+z^c\big(\mu y(1+b(x,y)) + zQ\big)\\
z+z^{c+1} R
\end{pmatrix},
\end{equation}
with $a,b \in \mf{m}_2$, $P,Q,R \in \mf{m}$. 
Let $\pi:X \to (\nC^3,0)$ be the blow-up of the line $\{x=z=0\}$ in $\nC^3$, and denote by $\wt{f}$ the lift of $f$ in $X$.
	
Then the saturated infinitesimal generator of $\wt{f}$ has two singularities on the fiber above the origin, namely $[1:0]$ and $[0:1]$
Moreover for $\wt{f}$ we have that:
\begin{itemize}
\item $[0:1]$ is a {\degspike}.
\item $[1:0]$ is a simple corner.
\end{itemize}
\end{prop}
\begin{proof}
Computations are analogous to the ones performed in the previous sections, and left to the reader.

\end{proof}

We deduce that no non-degenerate non-exceptional characteristic directions may appear in any model $\pi:X_\pi \to (\nC^3,0)$ obtained via modification (not necessarily strongly) adapted to the dynamics, dominating either $p_1$ or $p_2$.

\subsubsection{Point modifications}\label{sssec:pointmodif}

With the same techniques adopted in \refsec{biratstudyaboveresol}, it is possible to study characteristic directions on any model $\pi:X_\pi \to (\nC^3,0)$ obtained via \emph{point modifications}.

If we only allow point modifications, we cannot resolve the singularities of the infinitesimal generator $\chi$ of $f$, and this leads to having to deal with singularities of the saturated infinitesimal generator $\wt{\chi}_\pi$ (on a given model $X_\pi$) that are not log-canonical.

We omit definitions and computations in this case because they would stretch the length of this paper excessively.
Just to give a hint of what happens in this case, let us follow the resolution of singularity above $p_5$.
At $p_5$, the map $f_1$ obtained as lift of $f$ by the blow-up of the origin takes the form \refeqn{taylorIy}, for which the linear part of the saturated infinitesimal generator is nilpotent, of rank $2$ if we assume $R_{040} \neq 0$.
Let us say that this germ is a {\iforma} ($N$ stands for \emph{nilpotent}).

After blowing-up $p_5$, we get a second form at the point $p_{5,1}$ corresponding to $[1:0:0]$. In this case the linear part of the infinitesimal generator has still rank $2$, but the exceptional divisor locally consists of two irreducible components. Say that we get a {\iiforma}.

Blowing-up $p_{5,1}$, we get a line $L$ of singularities, corresponding to the singular directions $[p:0:r]$ with $[p:r] \in \nP_\nC^1$.
In this case, at $[1:0:0]$ we get another {\iiforma}. At $[0:0:1]$ we obtain a singularity for the saturated infinitesimal generator with vanishing linear part, that we call {\aforma} ($H$ stands for \emph{higher order}). At $[p:0:r]$ with $p,r \neq 0$, we get another nilpotent singularity, call it {\iiiforma}. In other terms, we got a $\IIIfor$-pattern, special points $\IIfor$ and $\OneC$, and with core $L$.
If we blow-up $L$, we get the resolution of singularities $\pi_0$ described by \refprop{exampleresolution}.
If we blow-up points, we need for example to deal with the blow-up of {\aforma}s, which is quite intricate.
Fundamental for the definition of these classes is the identification of the right non-resonant conditions, in the same spirit of the ones appearing for simple corners, as well as suitable conditions on the higher order terms of the saturated infinitesimal generator.

The birational study can be completed for point modifications, and one can show that no non-degenerate characteristic directions can appear in this case.
However, the additional forms, and the appearence of several new types of patterns, make the birational study for all possible modifications (strongly) adapted to $f$ combinatorially much more involved. We suspect that no non-degenerate characteristic directions can be found in this way either.

\subsubsection{Dynamics over $q_1$}\label{sssec:overq1}

We have shown that at the point $q_1$ there are no formal $f_4$-invariant curves transverse to the exceptional divisor.
When blowing-up $q_1$, we find a $\SpinC$-$\HalfC$-pattern (denote by $\wt{f}_4$ the lift of $f$), and we showed in \refprop{halfcornerinvcurve} that we can find a (smooth) formal $\wt{f}_4$-invariant surface $\wt{S}$ transverse to the exceptional divisor.
The surface $\wt{S}$ is the strict transform of a smooth formal $f_4$-invariant surface $S$ at $q_1$.
The map $f_4|_S$ gives a (formal) $2$-dimensional {\tid} germ, with saturated infinitesimal generator $\hat{\xi}$ of order $2$.
A direct computation shows that $f_4|_S$ has exactly two characteristic directions, corresponding to the two irreducible components of the exceptional dvisor.
If $\hat{\xi}$ admits another complex separatrix, we can apply again the arguments of \refsssec{parmfldhalfcorner} to reduce $f$ in Ramis-Sibuya normal form and find parabolic manifolds.
If $\hat{\xi}$ has no other complex separatrices, one needs to study the dynamics of $f_4$ more in details, and to extend the results of \cite{lopezhernanz-ribon-sanzsanchez-vivas:stablemanifoldsbiholoCn} to the case where we have invariant surfaces instead of curves.

Notice also that $S$ descends to a $f$-invariant surface $S_0$, which is not smooth.
Hence, even in the case of $S$ being convergent, we cannot apply the results of \cite{lopezhernanz-rosas:chardirdim2} to find paraboic curves for $f|_{S_0}$.

\small
\bibliographystyle{alpha}
\bibliography{biblio}
\vspace{-1mm}

\end{document}